\pdfoutput=1 
\documentclass[a4paper, 11pt, svgnames]{article} 
\usepackage[english]{babel}
\usepackage[T1]{fontenc}
\usepackage{lineno}
\usepackage{authblk}

\title{The golf model on $\znz$ and on $\Z$}
\author{Zoé Varin\thanks{Univ. Bordeaux, CNRS, Bordeaux INP, LaBRI, UMR 5800, F-33400 Talence, France}}

\usepackage[left=3.cm,right=3.cm,top=2.5cm,bottom=2.5cm]{geometry}
\RequirePackage[dvipsnames]{xcolor}

\usepackage[shortlabels]{enumitem}

\usepackage{graphicx}
\usepackage{authoraftertitle}
\usepackage{amsmath,amsfonts,amsthm,amssymb,dsfont}
\usepackage{mathrsfs}
\usepackage{mathtools}
\usepackage{pgf,tikz}

\usepackage{stmaryrd}
\usepackage{xspace}
\usepackage{verbatim}
\usepackage{appendix}
\usepackage{ifthen}

\usepackage{bm}

\usepackage{subcaption}

\usepackage{tikz}	
\usetikzlibrary{backgrounds,fit}  
\usetikzlibrary{arrows,automata}
\usetikzlibrary{shapes,decorations,arrows}
\usetikzlibrary{math} 
\usetikzlibrary{calligraphy,backgrounds}

\usepackage[normalem]{ulem} 

\usepackage[pagebackref]{hyperref}
\makeatletter
\hypersetup{
	colorlinks=true,
	linkcolor=blue,
	urlcolor=blue,
}
\makeatother

\usepackage{cleveref}

\newtheorem{theorem}{Theorem}[section]
\newtheorem{lemma}[theorem]{Lemma}
\newtheorem{claim}[theorem]{Claim}
\newtheorem{corollary}[theorem]{Corollary}
\newtheorem{proposition}[theorem]{Proposition}
\newtheorem{question}[theorem]{Question}
\newtheorem{definition}{Definition}
\newtheorem{remark}{Remark}

\renewcommand{\P}{\mathbb{P}}%
\newcommand{\Prob}[1]{\P\left(#1\right)}
\newcommand{\E}{\mathbb{E}}
\newcommand{\Esp}[1]{\E\left[#1\right]}
\newcommand{\Var}[1]{\text{Var}\left(#1\right)}
\newcommand{\Cov}[1]{\text{Cov}\left(#1\right)}
\newcommand{\eps}{\varepsilon}

\newcommand{\card}[1]{\left|#1\right|}
\newcommand{\smallcard}[1]{|#1|}
\newcommand{\N}{\mathbb{N}}
\newcommand\R{\mathbb{R}}
\newcommand{\ind}[1]{\mathds{1}_{#1}}

\newcommand \si{\text{if }}

\newcommand \et{\text{ and }}

\newcommand{\fonction}[5]{\begin{array}{lrcl} 
		#1: & #2 & \rightarrow & #3 \\
		& #4 & \longmapsto & #5 \end{array}}
\newcommand\limit[2]{
	\overset{#1}{\underset{#2}{\longrightarrow}}
}

\DeclareMathOperator{\argmin}{argmin}
\newcommand{\minargmin}{\argmin}
\renewcommand\bar{\overline}
\newcommand{\ceil}[1]{\lceil #1 \rceil}
\newcommand{\floor}[1]{\lfloor #1 \rfloor}
\newcommand\zkz{{\Z/k\Z}}
\renewcommand{\mod}{\text{ mod}~} 

\newcommand\LL{\Lukasie }

\newcommand{\TLdet}{{\Holesetdet^F}}
\newcommand{\intzo}{[0,1]}
\newcommand{\Czo}{C\intzo}
\newcommand{\Uzo}{\mathcal{U}([0,1])}
\newcommand{\Ballsetdet}{{B}}
\newcommand{\Ballset}{{\bf \Ballsetdet}}
\newcommand{\Sources}{{S}} 
\newcommand{\Clockdet}{C}
\newcommand{\Clock}{{\bf \Clockdet}}
\newcommand\tmoins{{t^-}}
\newcommand{\limdelta}{{\bm D}}
\newcommand \pcercle{\P^{n,\nballs,\nholes,p}}
\newcommand \pcerclenl{\P^{n,\nl(n),p}}
\newcommand \pcerclenllambda{\P^{n,\lambda \sqrt{n},p}}
\newcommand \pcerclen{\P^{n,\nballs(n),\nholes(n),p}}
\newcommand \pcercleparking{\P^{{\sf Parking},n,\nballs,p}}
\newcommand \pcercleparkingi{\P^{{\sf Parking},\ell_i+1,\ell_i,p}}
\newcommand{\configinitcercle}{C^{n,\nballs, \nholes}}
\newcommand{\pcerclei}{\P^{\ell_i,b_i, b_i+1,p}}
\newcommand{\mudbdt}{\mu^{\db ,\dt}}
\newcommand{\muclocks}{\mu^{\sf Clocks}}
\newcommand \pdbdh{\P^{\Z,\db,\dt, p}}
\newcommand \pdhdh{\P^{\Z,\dt,\dt, p}}
\newcommand \ptdhdh{\P^{\Z,t\dt,\dt, p}}
\newcommand \ptepsdhdh{\P^{\Z,t_\eps\dt,\dt, p}}
\newcommand{\golfsequences}{{\sf GolfSequences}^G\left[ \Holesetdet, \Sources \right]}
\newcommand{\golfsequencesrestricted}{{\sf GolfSequences}^G\left[ \Holesetdet, \Sources, \TLdet \right]}
\newcommand{\golfsequencesrestrictedX}{{\sf GolfSequences}^{\znz}\left[ \Holesetdet, \Sources, X \right]}
\newcommand{\golfsequencesrestrictedO}{{\sf GolfSequences}^{\Z/(\ell_i+1) \Z}\left[ H_{i}, S_{i}, \{0\} \right]}
\newcommand{\weight}{{\sf Weight}}
\newcommand{\forests}[2]{\mathcal{F}(#1,#2)}
\newcommand{\markedforests}[2]{\mathcal{F}^{\bullet}(#1,#2)}
\newcommand{\forestsnr}{\forests{n}{k}}
\newcommand{\randforests}[2]{{\bf F}_{#1,#2}}
\newcommand{\randforestsmarked}[2]{{\bf F}^{\bullet}_{#1,#2}}
\newcommand{\randforestsnr}{\randforests{n}{k}}
\newcommand{\Catalan}[1]{C_{#1}}
\newcommand{\Deltan}{\Delta^{(n)}}
\newcommand{\Deltannl}{\Delta^{(n,\nl)}}
\newcommand \Ii{\llbracket x_i , x_ {i+1} \rrbracket}
\newcommand \cardBi {\card{\Ballset^0 \cap \rrbracket x_i , x_ {i+1} \llbracket}}

\newcommand \cardBidet {\card{\Ballsetdet \cap \rrbracket x_i , x_ {i+1} \llbracket}}
\newcommand \cardTidet {\card{\Holesetdet \cap \rrbracket x_i , x_ {i+1} \llbracket}}
\newcommand{\height}{{\sf h}}

\newcommand{\grosevtdefdet}{{$X \subseteq H$, $\forall i, \cardTidet = b_i$ and $ \forall i, \cardBidet = b_i$}}
\newcommand{\grosevtnom}{{\sf Cond}(\Tinit, \Binit, X, (b_i)) }
\newcommand{\grosevtnomdet}{{\sf Cond}(H, B, X, (b_i)) }

\newcommand\znlz{{\Z/\nl \Z}}
\newcommand\Binit{\Ballset^0}
\newcommand\nballs{{N_\ball}}
\newcommand\nholes{{N_\hole}}
\newcommand\dballs{{d_\ball}}
\newcommand\dholes{{d_\hole}}
\newcommand \free{{\sf f}}
\newcommand\nl{{N_\free}} 
\newcommand\db{\dballs}
\newcommand\dt{\dholes}
\newcommand\pz{\P^{\Z,\db,\dt,p}}
\newcommand\zzlambda{{\bm L}^{(\lambda)}}
\newcommand{\TL}{{\Holeset^1}}
\newcommand{\PMC}{P}
\newcommand{\Holesetdet}{{H}}
\newcommand{\Holeset}{{\bf \Holesetdet}}
\newcommand \Tinit{\Holeset^0} 
\newcommand \ball{{\sf b}}
\newcommand \hole{{\sf h}}

\newcommand{\Z}{\mathbb{Z}}
\newcommand\znz{{\Z/n\Z}}

\newcommand{\subscript}[2]{$#1 _ #2$}

\newcommand{\Lpol}{\text{\L}}

\newcommand{\bff}{{\bf f}}
\newcommand{\aci}{A_c^i}
\newcommand{\nbc}{N_c}
\newcommand{\nbca}{N_{\alpha_a}}

\newcommand{\fna}{f_n^{\alpha_a}}
\newcommand{\fnc}{f_n^c}
\newcommand{\acone}{A_c^1}
\newcommand{\actwo}{A_c^2}
\newcommand{\rotn}{{{ \sf rot}^{(n)}}}
\newcommand{\rot}{{{ \sf rot}}}
\newcommand{\suiv}{{\sf next}}
\newcommand{\Excursions}{{\sf Exc}}
\newcommand{\sorted}[1]{{\sf Sorted}\left(#1\right)}
\newcommand{\Max}[1]{{\sf Max}\left(#1\right)}
\newcommand{\Sep}{{\sf Sep}}
\newcommand{\Lukasie}{{\L}ukasiewicz\ }

\newcommand{\Veta}V
\newcommand{\Vxi}W

\usepackage[weather]{ifsym} \newcommand{\Snowflake}{\text{\raisebox{-2.5pt}{\Snow}}}

\colorlet{bleu}{Blue!70!Cyan}
\newcommand{\redcolor}[1]{\textcolor{red}{#1}}
\newcommand{\orangecolor}[1]{\textcolor{orange}{#1}}

\usepackage[french,colorinlistoftodos,textsize=scriptsize,textwidth=30mm]{todonotes}
\usepackage{marginnote}
\newcommand{\todog}[2][]{{%
		\let\marginpar\marginnote
		\reversemarginpar
		\renewcommand{\baselinestretch}{0.8}%
		\todo[#1]{#2}}}

\newcommand{\addition}[1]{{#1}} 
\newcommand{\bigremoval}[1]{{\color{lightgray} }} 

\begin{document}
	\maketitle
	\begin{abstract}
		
		We introduce a particle model, which we call the {\textit{golf model}}. Initially, on a graph $G$, balls and holes are placed at random on some distinct vertices. The balls then move one by one, doing a random walk on $G$, starting from their initial vertex and stopping at the first empty hole they encounter, which they fill. On finite graphs, under reasonable assumptions (if there are more holes than balls, and if the Markov chain characterizing the random walks is irreducible), a final configuration is reached almost surely. In this paper, we are mainly interested in $\TL$, the set of remaining holes. We give the distribution of $\Holeset^1$ on $\znz$, and describe a phase transition for the largest distance between two consecutive holes when the number of remaining holes is of order $\sqrt{n}$. We show that the model on $\Z$ is well-defined if each vertex contains either a ball with probability $\db$, a hole with probability $\dt$, or nothing, independently of the other vertices, as long as $\db \leq \dt$, and we describe the law of $\TL$ in this case.
	\end{abstract}

\section{Introduction}

\subsection{Foreword}

In this paper, we define an interacting particle system called the golf model. It is defined on a connected graph $G$ (in the paper, $G$ is mainly $\znz$ or $\Z$) which initially contains balls and free holes at distinct vertices. Each ball is given a random clock uniformly distributed over $\intzo$. 
When its clock rings (at time $t$), a ball located at some vertex $u$ at time 0 is randomly played by a clumsy player until it falls into a free hole $v$; it then remains in this hole until the end of the process (at time $1$); at time $t$, the vertices $u$ and $v$ become neutral vertices, i.e.\ vertices without any free hole or ball (all steps of a ball's displacement are considered instantaneous, and correspond in distribution to an irreducible Markov chain on the graph).

On $\znz$, this model is close to the parking model, for which the ``block size process'' has been studied intensively (it is the process of the distances between consecutive free parking spaces). In particular, a phase transition for this process has been established by Chassaing-Louchard \cite{chassaing2002phase} (we discuss this in more detail in Section \ref{subsect:parking}).

In this paper, we mainly study two models: the golf model on $\znz$ and on $\Z$.

On $\znz$, this model is trivially well-defined; we are interested in $\TL$, the set of free holes at time 1. We characterize the distribution of $\TL$ (Theorem \ref{thm:loitrousresiduels}) and, equivalently, the joint distribution of the distances between consecutive holes (which we call the block sizes, and which play a similar role to the sizes between consecutive parking spaces in the parking model).

Computing the distribution of $\TL$ reveals the algebraic structure of the block distribution, which turns out to be crucial for making combinatorial connections with models of random forests, and for obtaining asymptotic results for the block sizes (as $n$ goes to infinity). In particular, we identify a phase transition for the distribution of the largest block sizes (see Theorem \ref{thm:phaseT}): when the number of free holes at time 1 is of order $\sqrt{n}$, the largest block sizes are of order $n$.

We can also define a golf process with balls that have more general moving strategies (instead of doing a random walk, they can, for example, choose to go to the right or to the left, depending on the parity of the distance to the first free hole on the right). Since the proofs are simpler to state for the golf process, and generalize straightforwardly to these variants, for pedagogical reasons, we give the results and the proofs for the golf process, and then explain in Section \ref{subsect:variantes} how we can generalize these results.

We then study the golf model on $\Z$, this time defined by independently choosing the state of each vertex to be a hole with probability $\dt$ or a ball with probability $\db$. Since, an infinite number of balls have moved before any time $t>0$, an argument is needed to prove that the golf model is well-defined on $\Z$ (see the beginning Section \ref{subsect:intro_Z}). Finally, we compute the distribution of $\TL$ on $\Z$ (see Theorem \ref{thm:loiblocsZ}), using some of the results on $\znz$.

\subsection{Definition of the model}

All along the paper, $G = (V,E)$ is a connected graph, with a finite or countable set of vertices $V$, and $\PMC$ is the transition matrix of an irreducible and recurrent Markov chain on $G$ (that is, $P = \left(P_{u,v}, u,v\in V\right)$ and $P_{u,v} >0 $ implies that $ \{u,v\}$ belongs to the edge set $E$). {We call $\PMC$-Random Walk a Markov chain with kernel $P$}. 

We define a random process $\eta = \left(\eta^t\right)_{t\in \intzo}$, which we call a {golf process}, and which models the evolution of a particle system as time passes from 0 to 1. It is a process defined on $G$, with two distinct sources of randomness : the initial configuration $\eta^0$, and its evolution towards a final configuration. For each $t$, $\eta^t = \left(\eta_v^t, v\in V\right)$ encodes the states of all vertices. It takes its values in the configuration space $\mathcal{S} = \{\hole, 0, \ball, \Snowflake\}^V$. We say that (at time $t$):
\begin{itemize}[noitemsep,topsep=0pt] 
	\item[\labelitemii] $v$ is a \textit{free hole} (or simply a hole) if $\eta^t_v = \hole$,
	\item[\labelitemii] $v$ contains a \textit{ball} if $\eta^t_v = \ball$
	\item[\labelitemii] $v$ is a \textit{neutral} vertex if $\eta^t_v = 0$. As the process evolves over time, there are several types of neutral vertices: \textit{occupied holes}, which were free holes at the beginning ($\eta^t_v = 0$ but $\eta^0_v = \hole$), vertices that contained a ball that has moved before time $t$ ($\eta^t_v = 0$ but $\eta^0_v = \ball$) and the original neutral vertices ($\eta^0_v = 0$).
	\item[\labelitemii] $v$ is \textit{frozen} when $\eta^t_v = \Snowflake$. In this case, the whole process is frozen until the end: $\forall t' \geq t, \forall u\in V, \eta^{t'}_u = \Snowflake$. This frozen state is a kind of cemetery state, where the whole process goes when something goes wrong. Well-defined golf processes will avoid this state with probability 1.
\end{itemize}

Finally, we define the set of free holes at time $t$, $\Holeset^t \coloneqq \{v : \eta^t_v = \hole\}$ and similarly the set of remaining balls $\Ballset^t \coloneqq \{v : \eta^t_v = \ball\}$. Note that, by definition, $\Ballset^0 \cap \Holeset^0 = \emptyset$. Sometimes, for simplicity, if a vertex $v$ is such that $\eta^t_v = \ball$, then we say that $v$ is a ball.

\paragraph{Distribution of the initial configuration:}
\begin{itemize}
	\item The initial state $\eta^0$ can be either random or deterministic. We will only consider models for which $\card{\Ballset^0} \leq \card{\Holeset^0}$, i.e. there are more holes than balls (when $G$ is infinite, we allow $\card{\Ballset^0}$ and $\card{\Holeset^0}$ to be both infinite).
	
	\item Each vertex $v$ has a random \textit{activation clock} $\Clock(v)\in (0,1)$. We will mainly consider the case where the $\Clock(v)$ are i.i.d.\ of uniform distribution on $\intzo$, hereafter denoted $\Uzo$ (which ensures that almost surely all the $\Clock(v)$ are different). We denote by $\muclocks$ = $\mathcal{L}\left(\Clock(v), v\in V\right)$ the clock distribution.\footnote{We will not deal with the case when clocks can be equal, but if one wants to define the golf process without loss of generality, then it suffices to use new random variables to break ties when there are some.}
	
\end{itemize} 
\paragraph{Random evolution of the configuration:} informally, we define the process $\left(\eta^t\right)_{t\in \intzo}$ so that for each vertex $v$ containing a ball, when the clock rings at time $\Clock(v)$, the corresponding ball instantaneously does a $\PMC$-Random Walk until it reaches a still free hole, which it fills (still at time $\Clock(v)$). Then $\eta^{\Clock(v)}$ is updated. We now give a definition of the process $\eta$ (the formal aspects concerning measurability will be discussed just before Definition \ref{defi:def_bien_def}):

We define a jump process $(\eta^t)$, whose jump times belong to the set $\{\Clock(v), v \in \Ballset^0\}$.
Let $t$ be the activation time of a ball: there exists $v$ such that $\Clock(v) = t \et \eta^t_v=\ball $. 
We assume that $\eta^\tmoins = (\eta^\tmoins_w)_{w\in V}$ is well-defined. Then, we consider $(W^v_k)_{k\geq 0}$ a $\PMC$-random walk starting at $v$, and stopped when it hits a free hole, that is at time $\tau = \inf \{k\geq 0 : {W^v_k} \in \Holeset^\tmoins \}$. If $\tau = \infty$ (the ball has not found any free hole), the whole process is \textit{frozen} until time 1 and we set $\eta^t_w \coloneqq \Snowflake, \forall w \in V$ (and thus nothing happens after time $t$). Otherwise, the ball jumps to the hole $u \coloneqq W^v_\tau$ (the first free hole it has reached) and we update the states so that the ball in $v$ fills the hole $u$: $\eta^t_u \coloneqq 0$, $\eta^t_v \coloneqq 0$ (both $v$ and $u$ become neutral) and the state of the other vertices do not change ($\forall w \notin\{v,u\}, \eta^t_w \coloneqq \eta^\tmoins_w$).
(All the random walks of the balls are independent of each other and of the initial configuration).

\paragraph{Final configuration:} At time 1, either the process is frozen or every ball has found a hole, so the process has reached a \textit{final configuration}, characterized (if it is not frozen) by the set of free holes $\Holeset^1$ and the set of occupied holes $\Holeset^0 \backslash \Holeset^1$.

\paragraph{Golf model:} To define a golf process $\eta$, we need 5 parameters: the graph $G$, the kernel $\PMC$ for the random walks, the clocks distribution $\muclocks$, and $\Ballset^0$ and $\Holeset^0$ (the initial configuration $\eta^0$ is characterized by the set of balls $\Ballset^0$ and the set of holes $\Holeset^0$, with $\Holeset^0 \cap \Ballset^0 = {\emptyset}$).
\medskip

The random variables $(\Clock(v))$, $\Ballset^0$, $\Holeset^0$ and the $\PMC$-Random Walks are assumed to be defined on a common probability space $(\Omega,{\cal A}, \mathbb{P})$, which is assumed to be large enough to accommodate also all the random variables used in the sequel.
\medskip

In the following, we are mainly interested in describing the distribution of $\TL$, the set of free holes at time 1 (which contains the same information as $\eta^1$).

\medskip
By nature, even if we are interested in the final configuration of the golf process, in order to get a formal definition we need to define this process as a continuous time process indexed by $[0,1]$, equipped with an ad hoc sigma-field. This introduces some additional complications that can be ignored on $\znz$ (because the process only has only a finite number of jumps).

Since $\eta$ is a continuous-time process which is not continuous, we decide to work on $D_\mathcal{S}[0,1]$ the Polish space of \textit{càdlàg} functions from $[0,1]$ to $\mathcal{S}$, equipped with the Skorokhod topology, where the configuration space $\mathcal{S} = \{\hole, 0, \ball, \Snowflake\}^V$ is equipped with the product topology. Recall that \textit{càdlàg} means right-continuous with left limits.

\begin{definition}\label{defi:def_bien_def}
	We say that $(\eta^t)$ is valid if it satisfies the following conditions : it is almost surely well-defined and not frozen at time 1, it is measurable, and it almost surely belongs to $D_\mathcal{S}[0,1]$.
\end{definition}

When there is a finite number of balls, since $\eta$ jumps only a finite number of times, the golf process $\eta$ is valid under reasonable assumptions:

\begin{proposition}\label{prop:bonne_def_graphe_fini_ET_commutation}
	We recall that we assume that $\PMC$ is irreducible and recurrent.  
	If $\eta$ is a golf process on a finite graph $G$, such that $\card{\Ballset^0} \leq \card{\Holeset^0}$, then the process $\eta$ is valid.
	
	The same results holds for infinite graphs, if $\card{\Ballset^0} \leq \card{\Holeset^0}$ and $\card{\Ballset^0}$ is finite.
\end{proposition}

\begin{remark}[\textit{Commutation property}]\label{remark:prop_commutation}
	We have the following Commutation property: on finite graphs, the distribution of $\eta^1$ does not depend on $\muclocks$. \addition{In fact, this property is even stronger : fix an initial configuration $\eta^0$. Then, the distribution of $\eta^1$ is the same for any choice of distinct clocks $(\Clock(v))_{v\in V}$.}
	
	This classical result is due to Diaconis-Fulton \cite[Proposition 4.1]{diaconis1991growth} (see Section \ref{subsect:litterature} for more details on their paper). A proof of this property relies on an alternative representation of the process: \addition{
		we associate to each vertex $u$ an infinite sequence $(V_i^u, i\geq 1)$, which we see as a stack of arrows/choices associated with $u$, such that $(V_i^u,i\geq 1)$ is a sequence of i.i.d.\` random neighbors of $u$ with common distribution $\P(V_1^u = v) =P_{u,v}$. Hence, for every $i$, $u\to V_i^u$ has the same distribution of a step of a random walk starting at $u$. Again, the balls move successively according to their clocks. When a ball at vertex $u$ has to take a step, it moves according to the first $V_i^u$ that has not yet been used. Diaconis and Fulton prove that for any realization of the $V_i^u$ (up to a negligible subset for which the balls move forever without finding holes), the final configuration is unchanged if we swap the clocks.}

	The Commutation property may help the reader to figure out that one can forget about the clocks (at least on finite graphs) if we are only interested in the final configuration. However, as we will see in Section \ref{subsect:variantes}, our analysis still applies to time dependent variants, which do not have the Commutation property. 
	
	Note that $\eta^1$ does not record the allocation of the balls (i.e. which ball fills which hole). If it did, Remark \ref{remark:prop_commutation} would no longer hold. Similarly, if we consider the entire process $\eta$ instead of $\eta^1$, then it also depends on the $\muclocks$.
\end{remark}

\begin{remark}
	We will see (for example in Section \ref{subsect:intro_Z}) that it becomes much more complicated when an infinite number of balls are involved. However, we will show that a golf process can be valid on $\Z$ under certain hypotheses (see Theorem \ref{thm:defZ} in Section \ref{subsect:intro_Z}).
\end{remark}

\subsection{Model and main results on $\znz$}\label{subsect:intro_znz}
We now focus on the golf process on the cycle $G = (\znz,E_n)$, with edge set $E_n = \left\{ \{x, x+1\}, x\in \znz \right\}$, for some $n\geq 1$. For simplicity, we often directly call this cycle $\znz$. 

We fix $\nballs \leq \nholes$, and we let $\configinitcercle$ be the set of all initial configurations with $\nholes$ holes and $\nballs$ balls:
\begin{equation}
	\configinitcercle \coloneqq \left\{ (y_0,\ldots,y_{n-1}) \in \{\hole, 0, \ball\}^{\znz} : \card{\{ i : y_i = \ball\}} = \nballs \et \card{\{i : y_i = \hole\}}= \nholes \right\}
\end{equation}
and thus $\card{\configinitcercle} = \binom{n}{\nballs, \nholes, n-\nballs-\nholes}$, the classical multinomial coefficient.

We draw the initial configuration $\eta^0$ uniformly in $\configinitcercle$.
We also fix $p \in \intzo$, and consider $\PMC$ such that $\forall x \in \znz, \PMC_{x, x+1} = p\et \PMC_{x, x-1} = 1-p$, so that $\PMC$ is the transition matrix of a biased random walk on $\znz$. Finally, we let $\nl = \nholes - \nballs$ be the number of free holes at time 1. 

\begin{remark}[Discussion on the distribution of the initial configuration]
	
	We choose to study the golf model with this particular initial distribution because 
	it is arguably the simplest exchangeable distribution that respects the constraint of having more holes than balls (and every distribution that verifies this is a mixture of uniform distributions on $\configinitcercle $, with $\nballs \leq \nholes$)
	. Moreover,
	it allows to study the golf process on $\Z$ (described in Section \ref{subsect:intro_Z}) as a limit of the golf process on $\znz$. 
{Actually, we will see in Section \ref{subsect:variantes} that our approach also enables to study a wide range of initial configurations.}
\end{remark}

\paragraph{Some notation.}We choose to encode subsets $X$ of $\znz$ with cardinality $\card{X} = k$ as sequences indexed by $\Z/ k \Z$, which describes the order in which we discover the elements of $X$ when {turning around the circle} (this will simplify some considerations further in the paper). Figure \ref{fig:notations_sur_znz} illustrates this notation.

\begin{figure}[]
\centering
\begin{tikzpicture}[line cap=round,line join=round,>=triangle 45,x=1.0cm,y=1.0cm,scale = 1]
	\clip(-3.8,-2.6) rectangle (3.5,2.3);
	\draw(0,0) circle (2cm);
	\begin{scriptsize}
		\draw [color=orange,fill = orange] (-1.41,1.41) circle (3.5pt);
		\node [black,below right, xshift = 3pt] at (-1.41,1.41) {$x_3$}; 
		\draw [color=orange,fill = orange] (-1.85,-0.77) circle (3.5pt);
		\node [black,right, xshift = 5pt] at (-1.85,-0.77) {$x_0 = x_4$};
		\draw [color=orange,fill = orange] (1.41,-1.41) circle (3.5pt);
		\node [black,above left, xshift = -4pt] at (1.41,-1.41) {$x_{1}$};
		\draw [color=orange,fill = orange] (1.85,-0.77) circle (3.5pt);
		\node [black,above left, xshift = -3pt, yshift = 1pt] at (1.85,-0.77) {$x_{2}$};
		\node[above] at (0,-1.9) {\redcolor{$0$}};
		\draw[color=bleu,above right] (1.19,1.88) node {$\Delta_2 X =6$};
		\draw[bleu,left] (-2.25,0.49) node {$\Delta_3 X = 2 $};
		\draw[bleu,below] (-0.49,-2.25) node {$\Delta_0 X = \Delta_4 X = 4 $};
		\draw[bleu,right] (2.0,-1.39) node {$\Delta_1 X = 0 $};
		\draw[bleu] (0,0) ++(120:2.2)
		arc(120:-7.5:2.2)
		-- ++(-187.5:0.4)
		arc(-7.5:120:1.8)
		-- cycle;
		
		\draw[bleu] (0,0) ++(-60:2.2)
		arc(-60:-142.5:2.2)
		-- ++(37.5:0.4)
		arc(-142.5:-60:1.8)
		-- cycle;
		
		\draw[bleu] (0,0) ++(187.5:2.2)
		arc(187.5:150:2.2)
		-- ++(-30:0.4)
		arc(150:187.5:1.8)
		-- cycle;
		\draw[bleu] (0,0) ++(326.25:1.8) -- ++(326.25:0.4);
		
	\end{scriptsize}
	\draw [black,fill = black] (0:2) circle (1pt);
	\draw [black,fill = black] (22.5:2) circle (1pt);
	\draw [black,fill = black] (45:2) circle (1pt);
	\draw [black,fill = black] (67.5:2) circle (1pt);
	\draw [black,fill = black] (90:2) circle (1pt);
	\draw [black,fill = black] (112.5:2) circle (1pt);
	\draw [black,fill = black] (157.5:2) circle (1pt);
	\draw [black,fill = black] (180:2) circle (1pt);
	\draw [black,fill = black] (225:2) circle (1pt);
	\draw [black,fill = black] (247.5:2) circle (1pt);
	\draw [black,fill = black] (270:2) circle (1pt);
	\draw [black,fill = black] (292.5:2) circle (1pt);
\end{tikzpicture} 

\caption{Illustration of our notation : for  $n=16$ and the set $X = \{2,3,10,13\} \subseteq \Z/n\Z$ (orange points on the figure), which we identify with the nicely ordered sequence $(x_1,x_2,x_3,x_4) =(2,3,10,13)$, we show the blocks (in blue) and their length $(\Delta_i X)_{i\in\Z/4\Z} = (x_{i+1} - x_i - 1 \mod n)_{i\in\Z/4\Z}$. Notice that $B_0 = \{14,15,0,1\}$.}
\label{fig:notations_sur_znz}
\end{figure}

We denote by $\pi$ the canonical projection $\znz \to \llbracket 1, n\rrbracket$ (here and below, $\llbracket a, b \rrbracket = [a,b]\cap\Z$).
A sequence $x=(x_i)_{i\in \zkz} \in \znz^k$ is said to be \textit{nicely ordered} if $\pi(x_1) < \ldots < \pi(x_{k-1}) < \pi(x_k) = \pi(x_0)$. For every set $X \subseteq \znz$ such that $\card{X} = k$, there exists a unique nicely ordered sequence $x$ of size $k$ containing all the elements of $X$, and we identify $X$ with this sequence $x$.

For $i \in \Z/k\Z$, we define the \textit{$i$th-block} associated to a nicely ordered sequence $x$ (and thus to $X$) as the interval {$B_i = \llbracket x_i, x_{i+1}\rrbracket \backslash\{x_i,x_{i+1}\}$, where $\llbracket a,b \rrbracket$ is the set $\{a \mod n, a+1 \mod n, \ldots, a+i \mod n\} \subseteq \znz$ and where $i$ is the smallest non-negative integer such that $a+i \mod n = b$ (thus $B_i$ is empty if $x_{i+1} = x_i + 1$).}
The size of the $i$th block is then $\Deltan_{i} X = \Deltan_{i} x = \pi(x_{i+1}) - \pi(x_i) - 1 \mod n$.

Note that $0$ belongs to $B_0 \cup \{x_0\}$.\\

The following theorem is one of the key results of the paper, and is proved in Section \ref{subsect:znzloiTL}.  In the theorem and elsewhere in the paper, we will use the notation $\pcercle$, to illustrate the dependence of the distribution of the golf process on the parameters of the model.

\begin{theorem}{(Distribution of set of remaining holes $\TL$.)} \label{thm:loitrousresiduels}
For every $X \subseteq {\znz}$ of size $\card{X} = \nl>0$ and every $i \in \Z/\nl \Z$, we set $\ell_i \coloneqq \Deltan_{i}X$. Then, we have: 
\begin{equation}
	\pcercle\left(\TL = X\right) = \frac{1}{{\binom{n}{\nballs, \nholes, n-\nballs-\nholes}}} \sum_{(b_i)\in B^{(\nballs, X)}}
	\prod_{i\in \Z/\nl\Z}^{} \frac{1}{b_i + 1} \binom{\ell_i}{b_i, b_i, \ell_i-2b_i} \label{eqn:loi_TL_cercle}
\end{equation}
where $B^{(\nballs, X)}$ is the finite set $\left\{(b_i)_{1\leq i \leq \nl} :\sum_{i}b_i = \nballs  \et \forall i, 0 \leq b_i \leq \ell_i/2 \right\}$, which corresponds to the set of possible assignments of $\nballs$ balls into $\nl$ intervals (compatible with the given set $X$), having at most $\ell_i/2$ balls in the $i$th interval.

In particular, the distribution does not depend on $p$.
\end{theorem} 

\begin{remark} 
The identification of the product form of the distribution will allow to pass to the limit eventually. In particular, the multinomial coefficients in the product appear in combinatorial formulas counting trees and we can make connections with the distribution of random forests (this is what we do in the proof of Theorem \ref{thm:phaseT}). 
\end{remark}

We now describe the statistical properties of the \textit{block-size process} $\Deltan\TL$ as $n$ goes to infinity, in two cases: when there is no neutral vertex at the beginning (i.e. $\nballs + \nholes = n$) and when $\nballs$ and $\nholes$ are fixed.

For any sequence $c = (c_i)_{i\in \Z/k \Z}$, we let $\sorted{c} = (\hat{c}_{1},\cdots,\hat{c}_{k})$ be this sequence rearranged in decreasing order, so that $\sorted{\Deltan \TL}$ is the sequence of the block sizes at time 1, in decreasing order, and $\Max{\Deltan \TL} \coloneqq \sorted{\Deltan \TL}_1$ is the size of the largest block in $\TL$.   

The following theorem (proven in Section \ref{subsect:cercle_asymptotiques}) allows one to see that there is a phase transition in the size of the largest block; ``giant blocks'' arise when there is an excess of $\nl = \nholes - \nballs \sim \lambda \sqrt{n}$ holes.

We will call the case where the number of balls goes to infinity, the \textit{dense case}.

\begin{theorem}[Dense case]\label{thm:phaseT} 
Let $(\nholes(n))$ and $(\nballs(n))$ be two sequences such that $n = \nballs(n) + \nholes(n)$, $ \nholes(n) \to \infty$ and $\nballs(n) \to \infty$. Under $\pcerclen$, 

\begin{enumerate}[$(i)$]
	\item \label{thm:phaseT_first_item}if $\nl(n)=  an + o(n)$ for some constant $a>0$, then there exists two explicit constants $\alpha, \beta>0$ such that
	$\P(\alpha \leq \Max{\Deltan \TL}/\log n \leq \beta)\underset{n\to\infty}\to 1$;
	\item \label{thm:phaseT_second_item}if $\nl(n)/\sqrt{n}\to \lambda>0$, then the process $(\sorted{\Deltan \TL}_i/n, i \geq 1)$ converges in distribution for the product topology to a process $(y^{(\lambda)}_j,j\geq 1)$, satisfying $y^{(\lambda)}_j>0$ a.s., for all $j$, and which can be represented as the excursion sizes above the current minimum of a Brownian motion $B_t$ on $\intzo$ conditioned by hitting $-\lambda$ at time $1$.\footnote{A formal definition of this Brownian motion is given in Proposition \ref{prop:cv_marchecondit_vers_browniencondit}.}
	\item \label{thm:phaseT_third_item} if $\nl(n) / \sqrt{n} \to +\infty$, then $\Max{\Deltan \TL}/n\to 0 $ in probability,
	\item \label{thm:phaseT_fourth_item} if $\nl(n) / \sqrt{n} \to 0$, then $\Max{\Deltan \TL}/n\to 1 $ in probability.
\end{enumerate}
\end{theorem}

The analysis relies on the fact that the sequences of block sizes in our model have the same distribution as the sequences of tree sizes in some uniform random forest with a marked node. We detail this correspondence in Section \ref{subsect:cercle_asymptotiques_preliminaries} and then prove Theorem \ref{thm:phaseT} in Section \ref{subsect:cercle_asymptotiques}.

Finally, we give another theorem about the asymptotic behavior of the golf process on $\znz$, this time when $\nballs$ and $\nholes$ are fixed (the \textit{sparse case}). Note that the last block is a deterministic function of the others (since the sum of all the block sizes is equal to $n-\nl$).

\begin{theorem}[Sparse case]\label{thm:blocs_cercle_nb_nt_fixes}
Let $\nballs$ and $\nholes$ be some fixed integers such that $\nballs < \nholes$.
Under $\pcercle$, the vector $({\Deltan_i \TL}/{n})_{0\leq i < \nl}$ converges in distribution to $\limdelta = (\limdelta_0,\ldots, \limdelta_{\nl-1})$ as $n\to\infty$, where $\limdelta$ is a random variable taking values in 
\begin{equation*}
	{\sf Simplex}(\nl) \coloneqq \left\{ x = \left(x_i\right)_{0 \leq i < \nl} : \sum x_i = 1  \et \forall i, x_i \geq 0\right\},
\end{equation*}
where $\limdelta_{\nl-1} = 1- \sum_{i=0}^{\nl - 2}\limdelta_i$ and $(\limdelta_0,\ldots, \limdelta_{\nl-2})$ has density function (with respect to the Lebesgue measure on $\R^{\nl - 1}$)
\begin{equation}
	f(x_0,\ldots,x_{\nl - 2}) \coloneqq x_0 \nballs! \nholes ! \sum_{(b_i) :\sum_{i}b_i = \nballs}\ \prod_{k \in \Z/\nl\Z} \ \frac{1}{b_k!(b_k + 1)!}{x_k}^{2 b_k}\ind{(x_0,\ldots,x_{\nl - 1})\in {\sf Simplex}(\nl)}, \label{eqn:densite_sparse_case}
\end{equation} 
where $x_{\nl - 1} = 1 - \sum_{i < \nl - 1}x_i$.
\end{theorem}
We prove this theorem in Section \ref{subsect:asympt_nb_et_nt_fixes}.

\begin{remark}\label{rem:loi_de_Dirichlet} The distribution of $\limdelta$ is a mixed Dirichlet distribution. 

Recall that the Dirichlet distribution of parameters $K \geq 2$ and $\alpha = \left(\alpha_i\right)_{0\leq i < K}$ is a distribution on ${\sf Simplex}(K) $ with probability density function $f(x_0,\ldots,x_{K-2}) = \frac{1}{\mathrm{B}(\alpha)} \prod_{i=0}^{K-1} x_i^{\alpha_i - 1}$, where $x_{K-1} = 1-\sum_{i<K-1}x_i$ (see for example \cite[Section 0.3.2.]{pitman2006combinatorial}). The normalizing constant $\mathrm{B}(\alpha) = \frac{\prod_{i=0}^{K-1} \Gamma(\alpha_i)}{\Gamma(\sum_{i=0}^{K-1}\alpha_i)}$ is called the multivariate beta function, as the Dirichlet distribution generalizes in higher dimensions the beta distribution.
\end{remark}

\begin{remark}
This limit is the distribution of the final configuration of several continuous analogues of the golf process defined on the unit circle $C = \R\backslash\Z$. We describe them in Section \ref{subsect:remarque_limit_sparse_case}.
\end{remark}

\subsection{The $p$-parking process on $\znz$}\label{subsect:parking}

We recall the definition of \textbf{the parking process}, which also appears in the computer science literature under the name ``hashing with linear probing'' (which is an important and efficient method to store data in an array).

We consider the usual cycle $\znz$ and we take $m < n$. Each vertex of $\znz$ is a parking space, which is available at time 0. We assume that at each time step $i$ (with $1 \leq i \leq m$), a car chooses uniformly at random a vertex $v_i \in \llbracket 1,n \rrbracket$ (independently of the other cars) and parks at the first available parking space \textit{to the right} (i.e.\ the first among $v_i, v_i + 1 \mod n, v_i + 2 \mod n \ldots)$. This slot is then occupied until the end of the process.

Sometimes, the parking model is defined on a path instead of a cycle: a car chooses $v_i \in \llbracket1,n-1\rrbracket$ (also uniformly at random, when we consider the probabilistic version of the process), and cannot park if it reaches the end of the path (vertex $n$) without finding an available parking space. Parking functions are the sequences $(v_1, \ldots, v_m)$ that leave the vertex $n$ empty at the end. Parking on a path and on a cycle are two problems that are combinatorially equivalent since rotating the cycle with respect to the last empty place yields a linear parking.

This problem has a long history, see for example \cite[Section 1.1]{chassaing2002phase} for a short summary. The reader will find some additional combinatorial information in \cite[Section 4]{flajolet1997Analysis} or in \cite{knuth1973art} (in particular, the number of parking functions of size $n$ is $n^{n-2}$; it was obtained in \cite{Konheim1966AnOD}). One of the proofs of the importance of this model is that the distribution of the relative distances between consecutive free spaces in the partially filled parking is the same as the cluster sizes in the additive coalescent \cite{chassaing2002phase}.

We now consider a variant of the parking process, which we call \textbf{the \textit{$p$-parking process}}. As in the parking process, each vertex of $\znz$ is initially a parking space. A $\nballs$ cars enter successively in the system, the $i$th car chooses a vertex $v_i$ uniformly at random (independently of the other ones) and from this, performs a $p$-Random Walk (from $k$ it goes to $k+1 \mod n$ with probability $p$, and to $k-1\mod n$ with probability $1-p$). Note that for $p=1$, the process is exactly the parking process.

In the following, we write $\pcercleparking$ to explicitly show the parameters of this model. As usual, $\TL$ denotes the number of holes at the end of the process, after all the cars have moved.

We now give the distribution of $\TL$ (which has size $\nl = n - \nballs$).
\begin{theorem}\label{thm:parking_loi_TL}
For every $X \subseteq {\znz}$ of size $\card{X} = \nl>0$, we set, for every $i \in \Z/\nl \Z$, $\ell_i \coloneqq \Delta^{(n)}_{i}X$. Then, we have:
\begin{align}
	\pcercleparking\left(\TL = X\right) = \frac{1}{n^{\nballs}}
	\binom{\nballs}{\ell_1,\ldots,\ell_{\nl}} 
	\prod_{i \in \Z/\nl\Z}^{} \left(\ell_i + 1\right)^{\ell_i-1}.
\end{align}
Once again, this result does not depend on $p$.
\end{theorem}
The proof of this theorem is very similar to the proof of Theorem \ref{thm:loitrousresiduels}, which gives the distribution of $\TL$ for the golf model. We detail the main arguments of the proof in Section \ref{subsect:parking_loi_TL}.

{We can reinterpret this parking model as a golf process (replacing cars by balls and parking spaces by holes) where we allow the initial configuration to have, on each vertex, one hole along with a random non-negative number of balls. If there are more holes than balls, this model is again valid.\footnote{In the golf process we defined at the beginning, the balls are present at time 0 and do not move until their clock rings. Here, the temporality is a bit different, but it does not change anything concerning the behavior of the process. Furthermore, there is again a commutation property, which we can state in this way: the distribution of $\TL$ depends only on the vector $W = (W_1,\ldots, W_n)$, such that $W_k = \card{\{i : v_i = k\}}$ is the number of cars that choose vertex $k$ (and does not depend on the order of arrival of the balls); this is again a consequence of \cite[Proposition 4.1]{diaconis1991growth}. The vector $W$ has the multinomial distribution with parameter $(\nballs; 1/n, \ldots, 1/n)$.}}

For $p=1$, this model is the usual parking process, for which asymptotics are already known (results due to \cite{pittel1987linear} and \cite{chassaing2002phase}) and are also valid when $p \neq 1$, thanks to Theorem \ref{thm:parking_loi_TL}.

\begin{corollary}\label{cor:phaseTparking} 
Let $(\nballs(n))$ be a sequence such that $\nballs(n) \to \infty$. For any $p$, under $\pcercleparking$,	
\begin{enumerate}[$(i)$]
	\item \label{thm:phaseTparking_first_item}
	
	{if $a \coloneqq \nballs(n)/n$ is bounded away from 0 and 1, then, in probability,}
	\begin{equation}
		\Max{\Deltan \TL} = \frac{1}{a-1-\log a}\left(\log n - \frac32\log \log n \right) + O(1),
	\end{equation}
	\item \label{thm:phaseTparking_second_item}if $\nballs(n)/\sqrt{n}\to \lambda>0$, then the process $(\sorted{\Deltan \TL}_i/n, i \geq 1)$ converges in distribution for the product topology to a process $(y^{(\lambda)}_j,j\geq 1)$, satisfying $y^{(\lambda)}_j>0$ a.s., for all $j$, and which can be represented as excursion sizes above the current minimum in a Brownian-like process $B = (e_t-\lambda t)_{t\in \intzo}$, where $(e_t)_{t\in\intzo}$ is a Brownian excursion,
	\item \label{thm:phaseTparking_third_item} if $\nballs(n) / \sqrt{n} \to +\infty$, then $\Max{\Deltan \TL}/n\to 0 $ in probability,
	\item \label{thm:phaseTparking_fourth_item} if $\nballs(n) / \sqrt{n} \to 0$, then $\Max{\Deltan \TL}/n\to 1 $ in probability.
\end{enumerate}
\end{corollary}

\begin{proof}
Theorem \ref{thm:parking_loi_TL} allows to treat only the case $p=1$, and in this case the asymptotic behavior of the parking is known: \ref{thm:phaseTparking_first_item} is due to Pittel \cite[Statement (1.3)]{pittel1987linear}, while \ref{thm:phaseTparking_second_item}, \ref{thm:phaseTparking_third_item} and \ref{thm:phaseTparking_fourth_item} are due to Chassaing and Louchard \cite[Theorems 1.1 and 1.2]{chassaing2002phase}.
\end{proof} 

\begin{remark} \label{rem:remarque_parking}
This corollary is very similar to Theorem \ref{thm:phaseT} {(although \ref{thm:phaseTparking_first_item} is stronger in Corollary \ref{cor:phaseTparking})}. The proof is different because the discrete objects are of a different nature. In the sublinear case (Theorem \ref{thm:phaseT}.\ref{thm:phaseT_first_item}), we conjecture the same phenomenon as in Corollary \ref{cor:phaseTparking}.\ref{thm:phaseTparking_first_item}), but we did not pursue in this direction. We will see in Subsection \ref{subsect:connections_with_CL} that the limit Brownian processes appearing in \ref{thm:phaseTparking_second_item} are the same up to some transformation preserving the excursion sizes and it implies that the block-size processes $\Deltan \TL$ in the golf model and in the $p$-parking model have the same limit in law up to a rotation (see, in particular, Lemma \ref{lemma:generalisation_Vervaat}).
\end{remark}

\subsection{Model and main results on $\Z$}\label{subsect:intro_Z}
We now turn our attention to the {golf process} on $\Z$. We let $G = (\Z, E)$ with edge set $E = \left\{ \{x, x+1\}, x\in \Z\right\}$. We again fix $p \in \intzo$, and consider $\PMC$ the transition matrix of a Markov chain such that $\forall x \in \Z, \PMC_{x, x+1} = p\et \PMC_{x, x-1} = 1-p$. For the initial configuration, we draw the state of each vertex independently of the others: the $(\eta^0_x)_{x\in\Z}$ are i.i.d.\ with common distribution:  
\begin{align}
\forall x \in \Z, \eta^0_x = \left\{\begin{array}{lll}
	1 &\text{(a ball $\ball$)} &\text{ with probability } \db, \\
	-1 &\text{(a hole $\hole$)} &\text{ with probability } \dt, \\
	0 &\text{(a neutral vertex)} &\text{ with probability } 1-\db-\dt. \\
\end{array}\right. \label{eqn:mudbdt}
\end{align}

One can note that we have slightly changed our notation (replacing $\ball$ by $1$ for balls, and $\hole$ by $-1$ for holes), so that the absorption of a ball by a hole corresponds algebraically to the operation ``$-1+1$''. Sometimes, for clarity purposes, we will still write $\ball$ instead of $1$, and $\hole$ instead of $-1$.

The activation clocks are i.i.d.\ and uniform on $\intzo$ (and also independent of the affectation of the balls). Contrary to the finite case, the well-definition of $\eta$ is not trivial, since almost surely an infinite number of balls have moved before time $t$, for all $t>0$, so it is not obvious that every ball can reach a free hole at its activation time and {measurability issues arise} (see Section \ref{subsect:discuss_def_Z}, where we discuss examples with different initial conditions, showing that the golf process on $\Z$ is not always valid; it can even be frozen at time $0^+$). 
We thus give the following theorem, which we prove in Section \ref{subsect:defZ}.

\begin{theorem}\label{thm:defZ}
If $\db \leq \dt$, then for every $p$, the {golf model} $(\eta^t)_{t\in\intzo}$ on $\Z$ is valid (in the sense of Definition \ref{defi:def_bien_def}).
\end{theorem}
{For $p\in\{0,1\}$, the $\PMC$-Random Walks are not irreducible, but the golf model is valid in this case too.}

We use the notation $\pdbdh$ for the probabilities concerning the golf model on $\Z$, to {make explicit} the parameters of the model.

The following remark enables to define the block-size process.
\begin{remark}\label{rem:TL_vide_ou_tres_infini}
{Under $\pdbdh$, either $\TL$ is empty, or $\TL$ has an infinite number of holes on the left and on the right of 0.} In fact, if we let $L = \inf \TL$ and $R = \sup \TL$ be respectively the leftmost and rightmost holes of $\TL$, then these variables have a distribution, which is invariant by translation, and thus must have their support outside $\Z$. It implies that either $L = +\infty$ and $R= -\infty$ ($\TL$ is empty) or $L = -\infty$ and $R= +\infty$ ($\TL$ has an infinite number of holes on the left and on the right of 0).
\end{remark}

Under the hypotheses of Theorem \ref{thm:defZ}, the golf process is valid, so the set of remaining holes $\TL$ is also well-defined. In the following, when considering cases with an infinite number of remaining holes on the left and on the right of the origin, it will be convenient to view $\TL = (\TL_i, i \in \mathbb{Z})$ as an increasing process indexed by $\mathbb{Z}$, where $\TL_0$ is the largest non-positive hole. Thus, $\TL$ is therefore viewed as a process taking its values in $\Z^\Z$ equipped with the product topology.

Finally, when $\TL$ is empty, we have only one block (the block containing 0), of size $\Delta_0\TL = \infty$. 
Otherwise, we define the \textit{block-size process} $(\Delta_i \TL, i \in \mathbb{Z})$ by $\Delta_i \TL=\TL_{i+1}-\TL_i-1$, for all $i\in\Z$.

Since $\eta$ is translation invariant, it is equivalent to study $\Delta\TL$ or $\TL$ (but studying $\Delta\TL$ is easier when doing the proofs).\footnote{Indeed, if one wants to find the set $\TL$ instead of $\Delta\TL$, it suffices to draw the process $(\Delta_i\TL)_{i \in \Z}$, to choose the position of the first hole on the left of 0 uniformly in the interval $\llbracket -\Delta_0 \TL,0\rrbracket$, and finally to define all the holes in the unique way that is consistent with the block-size process.}
The following theorem gives the distribution of $\Delta\TL$ on $\Z$ when $\db + \dt = 1$; it is proven in Section \ref{subsect:ZloiTL}.

\begin{theorem}{(Distribution of $\TL$.)}\label{thm:loiblocsZ}
Under $\pdbdh$, with $\db + \dt = 1$ and $\db < \dt$, the distribution of $\left(\Delta_i \TL\right)_{i \in \Z}$ {is characterized by its finite-dimensional distributions:} 

For every $R>0$, for every $(b_i) \in \N^{2R+1}$,
\begin{align}
	\pz\left( \Delta_i \TL = 2b_i, -R\leq i \leq R\right) &= \frac{(2b_0+1)C_{b_0} \lambda^{b_0} }{\mathcal{H}(\lambda)} \prod_{i=-R, i\neq 0}^{R}\frac{ C_{b_i} \lambda^{b_i}}{ \mathcal{G}(\lambda)} \label{eqn:legigathm}
\end{align}
where for every $k\geq 0$, $C_k = \frac{1}{k+1} \binom{2k}{k}$ is the $k$th Catalan number, $\lambda = \db \dt$, and for every $a \in (0,1/4]$,  
\begin{align*}
	\mathcal{G}(a) = \sum_{k\geq0} C_k a^k = \frac{1-\sqrt{1-4a}}{2a} \ \et \ \mathcal{H}(a) = \sum_{k\geq 0} (2k+1) C_k a^k = 2a \mathcal{G}'(a) + \mathcal{G}(a).
\end{align*}

In particular, the $\Delta_i\TL$ are independent with common distribution ($ \forall i\neq 0, \forall b$,  $\P(\Delta_i\TL = 2b) = C_b \lambda^b / \mathcal{G}(\lambda)$), except for $\Delta_0\TL$ which is size-biased ($\forall b,\P(\Delta_0\TL = 2b) = (2b+1)C_b \lambda^b / \mathcal{H}(\lambda)$).
\end{theorem}

\begin{remark}
\begin{itemize}
	\item 
	$\mathcal{G}$ is the generating function of the Catalan numbers.
	
	\item 	$\lambda$ is the unique number such that the average length of a block $\frac{2\lambda \mathcal{G}'(\lambda)}{\mathcal{G}(\lambda)} $ is compatible with the densities of holes and balls. Indeed, for every $i\neq 0$, if ${B}_i = [r_i, l_i]$ is the $i$th block, then set $I_i = [r_i, l_{i}+1]$ (the intervals $I_i$ form a partitioning of $\Z$).
	Then, $I_i$ contains $\frac{\Delta_i \TL}{2}$ balls, $\frac{\Delta_i \TL}{2} + 1$ holes and has size $ \Delta_i \TL + 1$. We therefore need to have: $\E\left[\frac{\Delta_i \TL}{2}\right] = \db\left(\E\left[\Delta_i \TL\right] + 1\right)$ and $\E\left[\frac{\Delta_i \TL}{2} + 1\right] = \dt\left(\E\left[\Delta_i \TL\right] + 1\right)$. Solving this system gives $\E[\Delta_i \TL] = \frac{2\db}{2\dt-1}$.
	
	The equation $\frac{2\lambda \mathcal{G}'(\lambda)}{\mathcal{G}(\lambda)} = \frac{2\db}{2\dt-1}$ has a unique solution between 0 and $1/4$ (because $x\mapsto \frac{2x \mathcal{G}'(x)}{\mathcal{G}(x)}$ is a continuous and increasing function mapping 0 to 0 and going to $+\infty$ as $x$ tends to $1/4$), and it is straightforward to check that $\lambda = \db \dt$ is a solution.
	
	\item The block containing 0 is size-biased: the origin 0 can be in $2k+1$ positions in the block indexed by 0, if this block has size 2k. 
	We will see that on $\znz$ the block containing 0 is also size-biased (see for example Equation (\ref{eqn:jolie_formule_deltaTL_znz})).
	
	\item It is heavier (in terms of computation) but also possible to obtain an analogous theorem for the distribution of $\Delta\TL$ when $\db + \dt < 1$. This will be discussed in Section \ref{subsect:blocks_Z_cascomplexe}.
	
\end{itemize}

\end{remark}

It is somewhat remarkable that the critical case, in which the densities of balls and holes coincide, can be solved exactly:
\begin{theorem} \label{prop:Z_same_density_bh_empty_holeset}
When $\db = \dt$, under $\pdbdh$, $\TL = \emptyset$ almost surely.
\end{theorem}
We prove this theorem in Section \ref{subsect:Z_pas_de_trous_si_db=dt}.

\subsection{Related models}

\subsubsection{Variants and new results on them}\label{subsect:variantes}

\paragraph{Different moving strategies in the golf process.} It is natural to wonder about the rigidity of the model: since the distribution of $\TL$ does not depend on $p$, does it depend on the strategy of the ball ?

\addition{The \textit{strategy} of a ball is, in full generality, the distribution of the trajectory that this ball will follow. It can be deterministic or not, and it can depend on the whole configuration $\eta^t$ at the time of its activation, or on a subset of it.

We now allow the balls to move with a more complex strategy than the random walk: it can move according to some deterministic or random function, but as before, at each step, it can only move from one vertex to one of its neighbors (in particular it cannot jump over a free hole).} We borrow the following definitions from Nadeau \cite{nadeau_bilateral_parking_procedures}. We say that a strategy is \textit{shift invariant} if it is invariant by rotation on $\znz$ or invariant by translation on $\Z$. Moreover, a strategy is \textit{local} if a ball has a strategy that depends only on its clock and of the state of the process between the closest hole on its right and the closest hole on its left.

The following examples correspond to shift invariant and local strategies:
\begin{enumerate}
\item Each ball independently chooses a unique direction (left with probability $q$ or right with probability $1-q$) and moves in this chosen direction until it finds a free hole. \label{first_variant}
\item Each ball goes to the closest hole (and chooses one of them uniformly at random if the two holes are at the same distance from the ball).
\item Or something completely arbitrary: A ball with clock $C$ goes to the left if the closest hole on the right is at an even distance, and otherwise it does a random walk with parameter $C$.
\end{enumerate}

We can define a golf process with every shift invariant and local strategy on $\znz$ (still specifying $n$, $\nballs$ and $\nholes$ if necessary) and on $\Z$ (also specifying $\db$ and $\dt$). We prove the following.

\begin{proposition}\label{claim:variantes} We assume that each ball has a strategy that is shift invariant and local.
\begin{itemize}
	\item For the golf process on $\znz$ (with initial condition given at the beginning of Section \ref{subsect:intro_znz}), $\TL$ has the same distribution as in the golf process with random walks of parameter $p$: Theorems \ref{thm:loitrousresiduels}, \ref{thm:phaseT} and \ref{thm:blocs_cercle_nb_nt_fixes} apply to this more general model.
	
	\item For the parking process on $\znz$ (with initial configuration detailed in Section \ref{subsect:parking}), $\TL$ has the same distribution as in the $p$-parking process: Theorem \ref{thm:parking_loi_TL} and Corollary \ref{cor:phaseTparking} apply to this more general model.
	
	\item On $\Z$, the golf process with these strategies is valid, and again the distribution of $\TL$ is the same as in the golf process with random walks of parameter $p$: Theorems \ref{thm:defZ} and \ref{thm:loiblocsZ} hold for this model too.
\end{itemize}
\end{proposition}
We prove this proposition in Section \ref{subsect:proof_prop_variantes}.

\begin{remark}\label{rem:generalisation_local_shift_inv_hyp}\
The special case of the variant \ref{first_variant} for the parking process was already mentioned in the last section of \cite{Konheim1966AnOD} (with the parameter $q$ depending on the car). \addition{It was studied more thoroughly by Durmi{\'c} et al.\ \cite{Durmic_Probabilistic_parking_functions_2023}, who prove that the probability that a given vertex is free at the end does not depend on $p$ (which was already claimed by Konheim and Weiss \cite{Konheim1966AnOD}) and then study a statistic that does depend on $p$, namely the vertex at which the last car tries to park (conditional on the position of the last free hole, or equivalently, when the process is considered on the line). 
}\\
Moreover, in \cite{nadeau_bilateral_parking_procedures}, Nadeau studies the parking process on $\mathbb{Z}$ with a finite number of cars, by studying the parking on $\mathbb{Z}/n\mathbb{Z}$ that leave only one vertex free, and he shows that for any shift invariant and local strategy the number of parking functions is the same (in our setting, it says that the probability that some vertex is the only free one does not depend on the cars strategy).  
\end{remark}

\paragraph{Golf process on $\znz$ with several balls per site.} We have shown in Section \ref{subsect:parking} that the $p$-parking process can be seen as a variant of the golf process on $\znz$ where we allow the initial configuration to have one hole along with a random non-negative number of balls at each vertex. In the case of the parking, Theorem \ref{thm:parking_loi_TL} gives the distribution of $\TL$. 

Let $(N_j)_{-1\leq j < n}$ be such that $\sum_{j\geq-1}N_j = n$. We consider another variant, which consists in fixing $N_j$, the number of vertices containing $j$ balls, for every $j\geq -1$ (to simplify the notation, a hole is considered here as a $-1$) and then taking an initial configuration uniform in ${\sf Ini}(N_j,-1\leq j <n)$ the set of configurations that satisfy these constraints. As long as $N_{-1} \geq \sum_{j\geq 0} jN_j$, the model is valid (Proposition \ref{prop:bonne_def_graphe_fini_ET_commutation} can be easily generalized). 

The support of $\TL$ depends on $(N_j)_{-1\leq j < n}$. Given $(N_j)_{-1\leq j < n}$ and $X$, the following proposition gives the probability that $\TL = X$. It is a sum over the possible assignments of the balls that are consistent with this event. We still use the convention that $\ell_i = \Deltan X$ is the length of the $i$th interval. We denote by $b^i_j \geq 0$ the number of vertices in the $i$th interval containing $j$ balls. 

If we consider the vector $$B = (b^i_{j},-1\leq j < n)_{i\in \Z/\nl\Z},$$ then $B$ is compatible with the event $\TL = X$ if and only if:
\begin{itemize}
\item The $i$th interval contains as many holes as balls: $b^i_{-1} = \sum_{j=0}^k jb^i_j$.
\item There are $\ell_i$ vertices in the $i$th interval: $\sum_{j\geq-1} b^i_j = \ell_i$.
\item The total number of vertices containing $j$ balls is $N_j$: $\forall j, \sum_{i\in\Z/\nl \Z}b_j^i = N_j$.
\end{itemize}
We let $B(X)$ be the set of all these compatible vectors.

This model can be studied with exactly the same tools as those presented above \addition{(we thus omit the proof here)}, and we can obtain the following formula. 
\begin{proposition}\label{prop:bla} Let $(N_j)_{-1 \leq j < n}$ be fixed non-negative integers such that $\nl \coloneqq N_{-1} - \sum_{j\geq 0} jN_j > 0$ and $\sum_{j\geq -1}N_j = n$. Let $\TL$ be the set of remaining holes in {a golf model} starting from a uniform configuration taken in ${\sf Ini}(N_j,-1\leq j <n)$. Then, for any set $X \subseteq \znz$ of cardinality $\nl$, 
\begin{equation}
	\Prob{\TL = X} = \frac{1}{{\binom{n}{N_{-1},\ldots,N_n}}} \sum_{B \in B(X)} 
	\ \	\prod_{i\in \Z/\nl\Z}^{} \ \frac{1}{b^i_{-1}} \binom{\ell_i}{b^i_{-1},b^i_0,\ldots, b^i_n}. \label{eqn:loi_TL_variante}
\end{equation}
\end{proposition}

\begin{remark}\
\begin{itemize}
	\item Here again, this proposition is valid for any ball moving strategy verifying the hypotheses given in Remark \ref{rem:generalisation_local_shift_inv_hyp}.
	\item 
	This formula allows to compute the distribution of the remaining holes for all distribution on the set of configurations ${\sf Ini}$ (as long as it is invariant by any permutation of the vertices). In fact, {it suffices to compute the distribution of $(N_j)_{j\geq -1}$, and then conclude with the distribution of $\TL$ conditional on this $(N_j)_{j\geq -1}$, given by Proposition \ref{prop:bla}}.
\end{itemize}
\end{remark}

\subsubsection{Open questions}

\paragraph{Golf process on $\znz$ with several holes per site.} We have seen in the previous paragraph that when there is at most one hole per vertex, we can compute the distribution of $\TL$ under general hypotheses. When we allow a vertex to contain multiple holes (i.e. $\Tinit$ is a multiset), then it seems more complicated. We illustrate this fact with a small example. 

We fix $\nballs$ and $\nholes$, the number of balls and holes, and let $m_i(\Tinit)$ (esp.\ $m_i(\Binit)$) count the number of holes (resp.\ balls) at vertex $i$. We assume that $(m_i(\Tinit),i\in\znz)$ has the multinomial distribution with parameters $(\nholes; 1/n, \ldots, 1/n)$ and similarly, $(m_i(\Binit),i\in\znz)$ has the multinomial distribution with parameters $(\nballs; 1/n, \ldots, 1/n)$ (as for the parking process). This initial configuration corresponds to the model where $\nholes$ holes and $\nballs$ balls successively choose uniformly at random a vertex of $\znz$ (with repetitions allowed), independently of the other balls and holes.

In view of what we have said before, it is routine that the golf process with such an initial configuration is valid as long as $\nballs \leq \nholes$, and thus that the multiset $\TL$ is also well-defined.

When $n = 4$, $\nballs = 2$ and $\nholes = 4$, the distribution of $\TL$ is quite complicated to compute by hand, but can be easily done with a computer algebra system (for a fixed initial configuration, it is possible to compute the probability that a ball will reach some hole before the others, and thus to compute the probability of some final configuration; it is then sufficient to sum these probabilities over all the possible initial configurations). We can obtain, for example, the probability that the two remaining holes are at position 0 is:
\[
\Prob{\TL = \{0,0\}} = \frac{106 p^4 - 212 p^3 + 322 p^2 - 216 p + 107}{1024 (p^2 - p + 1)^2},
\]
while 
\[\Prob{\TL = \{0,1\}} = \frac{110 p^{4}-220 p^{3}+326 p^{2}-216 p +109}{1024 \left(p^{2}-p +1\right)^{2}}.\]
This time, these values depend non trivially on the parameter $p$. We leave as an open question the characterization of the distribution of $\TL$ in general.

\paragraph{Distribution of the occupied holes in the standard golf process} We focused on the distribution of $\TL$, which does not depend on $p$ and is explicit (see Theorem \ref{thm:loitrousresiduels}). It is not the case for the set of occupied holes $\Tinit\backslash\TL$. It is easy to observe on small examples that the joint distribution of $(\Tinit\backslash\TL, \TL)$ depends on $p$. We can give a simple example that shows this dependence on $p$ for $\Tinit\backslash\TL$ (also computed with a computer algebra system): if $n = 6$, $\nballs = 2$ and $\nholes = 4$, then 
\[\pcercle\left(\Tinit\backslash \TL = \{0,1\}\right) = \frac{-2 p^{6}+6 p^{5}-7 p^{4}+4 p^{3}+2 p^{2}-3 p +1}{30 \left(p^{2}-p +1\right) \left(p^{2}-p +\frac{1}{2}\right)} .\]

We also leave the characterization of $\Tinit \backslash \TL$ as an open question.

\paragraph{Definition of the golf process on $\Z^d$.}
A natural question that arises when reading Theorem \ref{thm:defZ}, and that remains open, is the following.
\begin{question}
For $d\geq 2$, if $\db \leq \dt$, is the golf model on $\Z^d$ valid? \label{question:def_Zd}
\end{question}
{When $\db < \dt$, it seems reasonable to conjecture that the golf process on $\Z^d$ is valid.} On $\Z^2$, it is well-known that the standard random walk is irreducible and recurrent, thus intuitively ``a ball can reach a free hole if such hole exists''; but as we explain in Section \ref{subsect:discuss_def_Z}, it is not easy to prove that such hole always exists. The proof of validity on $\Z$ relies heavily on \textit{separators}, that are holes that allow $\Z$ to be divided into independent finite intervals, and this proof cannot simply be generalized to higher dimensions. Question \ref{question:def_Zd} thus remains open, even on $\Z^2$ with balls doing standard random walks.

\subsubsection{Related models in the literature}\label{subsect:litterature}

In addition to the parking processes discussed in Section \ref{subsect:parking}, we discuss other models that are close to the one we study here.

First, the name ``golf model'' comes from Fredes and Marckert \cite{fredes2021aldousbroder}, who introduced \textit{golf sequences}, which are the sequences of trajectories of the balls.
These golf sequences were introduced as a tool for a combinatorial proof of Aldous-Broder theorem.

We mentioned in Remark \ref{remark:prop_commutation} the Commutation property, which can be seen as a consequence of \cite[Proposition 4.1]{diaconis1991growth}. In this paper, Diaconis and Fulton defined a growth model in which some particles, starting at some positions $x_1,x_2,\ldots $, stop at their first hitting time of a set $Y$ (as what we do here, when a particle at some vertex $y\in Y$ receives a particle, it loses its capacity to absorb another particle, and behaves as a standard non-absorbing vertex). For this general type of model, they show that the distribution of the eventually occupied vertices of $Y$ does not depend on the initial order of the balls.

The internal diffusion limited aggregation model is an instance of this type of growth model, where the graph is the lattice $\Z^d$ and all the balls start at the same position. We refer to Lawler  et al.\ \cite{lawler1992internal} for the study of the asymptotic shape. \\

Several authors \cite{Damron_Parking_transitive_unimodular_graphs_2019, przykucki2019parking} study a parking process on $\Z$, with an initial configuration similar to our golf process in the case $\db + \dt = 1$: at the beginning, each vertex contains either a car (a ball) with probability $\db$ or a parking space (a hole, with probability $\dt$) independently of the other vertices. But the dynamics is different: it is a discrete-time process, at time $k$ (for every $k\geq 0$), all the balls that have not yet reached a free hole do a random step (with random variables to break ties if several balls arrive at a free parking space at the same time). The process is clearly well-defined for all $k$, and in particular is well-defined regardless of the value of $\db$. The question that Przykucki et al.\ study in \cite{przykucki2019parking} is different from ours: their main results concern the average time that balls take to reach a hole. In fact, if $\db \geq 1/2$, the expected time for a car to reach a parking space is infinite. When $\db < 1/2$, every car reaches a parking space in almost surely finite time, but in fact no final configuration is reached in a finite time. Actually, we prove in Theorem \ref{cor:Zgolf=parking} that for $\db \leq 1/2$, the final configuration in our model and in theirs (i.e.\ in their case the limit process as time goes to infinity) have the same distribution.

In the Activated Random Walk model, starting with infinitely many particles on $\Z^d$, the particles perform independent continuous-time random walks, falling asleep at rate $\lambda$ where they are alone at their location (see for example \cite{Cabezas_2014} or this survey by Rolla \cite{Rolla_2020}). For $\lambda = \infty$, this model is a continuous-time version of the parking model. This model has also been studied on $\znz$ by Basu et al.\ who focus on the time taken by the process to stabilize when $\lambda$ is finite \cite{Basu_ARW_cycle_2019}. Again, the models are close but have different time dynamics, since in the golf model the balls move only once; the authors are particularly interested in the average behavior of their models (typically, how many times the origin is visited on average), not in computing the distribution of the process (as we are, here), so that their results are of a different nature from those of the present paper. 

We can also mention diffusion-limited annihilating systems (see for example \cite{Johnson_density_in_diff_limited_annih_systems_2023}), where particles of two types are randomly placed on $\Z$, also perform continuous-time random walks (at a rate depending on their type) and annihilate when hitting a particle of the opposite type. When one type of particle does not move, the model amounts to the parking model of Przykucki et al.\ we described above.

Finally, in Nadeau \& Tewari \cite[Section 4]{Nadeau_2021} (see also Petrov \cite{PETROV2018336}), while studying some algebraic questions related to the basis of Schubert classes, the authors discussed combinatorial formulas arising when one computes the probability of some events related to the final configuration in a $p$-parking in which the initial configuration is fixed.

\subsection{Contents of the paper}

The paper is organized as follows. 

\textbf{Section \ref{sect:znz} contains the proofs of all the theorems concerning the golf model and the $p$-parking on $\znz$.} In Section \ref{subsect:znzloiTL}, we prove Theorem \ref{thm:loitrousresiduels} (for the distribution of $\TL$ in the golf model on $\znz$), and we discuss the proof of Theorem \ref{thm:parking_loi_TL} in Section \ref{subsect:parking_loi_TL} (for the distribution of $\TL$ in the $p$-parking model). Section \ref{subsect:cercle_asymptotiques} is devoted to the proof of the asymptotic results for the block-size process $\Deltan\TL$ (Theorem \ref{thm:phaseT}), with a discussion on the combinatorial links between our block-size process and some combinatorial models of forests and paths in the previous section (Section \ref{subsect:cercle_asymptotiques_preliminaries}). Then, in Section \ref{subsect:connections_with_CL} connections are made between the asymptotics given in Theorems \ref{thm:phaseT} and \ref{cor:phaseTparking}. Finally, Section \ref{subsect:sparse_case} focuses on the block-size process with a fixed number of balls and holes, with the proof of Theorem \ref{thm:blocs_cercle_nb_nt_fixes} in Section \ref{subsect:asympt_nb_et_nt_fixes}.

\textbf{Section \ref{sect:z} contains the proofs of all the results concerning the golf model on $\Z$.} We prove that the golf model is valid (Theorem \ref{thm:defZ}) in Section \ref{subsect:defZ}. In Section \ref{subsect:discuss_def_Z}, we discuss the difficulty that arise when an infinite number of balls are involved. Theorem \ref{thm:loiblocsZ}, giving the distribution of the block-size process when $\db+ \dt = 1$, is proved in Section \ref{subsect:ZloiTL}. In the next Section (Section \ref{subsect:blocks_Z_cascomplexe}), we study the same process when $\db + \dt <1$. Finally, we prove Theorem \ref{prop:Z_same_density_bh_empty_holeset} in Section \ref{subsect:Z_pas_de_trous_si_db=dt} and couple the golf process with the parking process on $\Z$ in Section \ref{subsect:golf=parking_onZ}.

\section{Golf model on $\znz$} \label{sect:znz}

\subsection{Proof of Theorem \ref{thm:loitrousresiduels}}\label{subsect:znzloiTL}

Before proving Theorem \ref{thm:loitrousresiduels}, we give two lemmas. The first one focuses on the case when $\nholes = \nballs + 1$, in which at the end there is exactly one remaining hole, and which we call the \textit{mini-golf} case. This case in much simpler, and will appear to be crucial in the proof of Theorem \ref{thm:loitrousresiduels}.

\begin{lemma}{(Distribution of remaining holes - mini-golf case)} \label{lemma:casinitloitrousresiduels} 
We assume that $\nholes = \nballs + 1$.
For every $x\in\znz$, \begin{equation} \label{eqn:casinitloitrousresiduels1}
	\P^{n, \nballs, \nballs + 1,p}\left(\TL = \{x\}\right) = \frac{1}{n}
\end{equation} and 
\begin{equation} \label{eqn:casinitloitrousresiduels2}
	\P^{n, \nballs, \nballs + 1,p}\left(\TL = \{x\} \middle| x \in \Tinit \right) = \frac{1}{\nballs + 1}. 
\end{equation}
\end{lemma}

Observe that $\P^{n, \nballs, \nballs + 1,p}\left(\TL = \{x\} \middle|\Tinit = H\right) $ (for some set $H\subseteq \znz$) is different from (\ref{eqn:casinitloitrousresiduels2}), and strongly depends on the geometry of $H$.

\begin{proof}[Proof of Lemma \ref{lemma:casinitloitrousresiduels}]
Equation (\ref{eqn:casinitloitrousresiduels1}) follows immediately from the fact that both the distribution of the initial configuration and $P$ are rotationally invariant, which implies that the final configuration is also {rotationally} invariant, hence the result. Equation (\ref{eqn:casinitloitrousresiduels2}) follows immediately from Equation (\ref{eqn:casinitloitrousresiduels1}).

\end{proof}

The second lemma is a crucial decomposition lemma, which is the key point of all the results concerning golf processes. At time 1, when the golf process achieves its final configuration $\eta^1$, the intervals (blocks) between the remaining empty holes of $\TL$ are equilibrated in the following sense: each of them must have contained as many balls than holes at time 0. The following lemma allows to decompose the probability that $\TL = X$, conditional on the number of balls on each block at time 0, as the product of simpler probabilities on mini-golfs (on the blocks defined by $X$). It relies on the following \textbf{block decomposition principle:} the trajectory of any ball does not intersect with the set of final free holes, so the free holes of $\TL$ allow to divide $\znz$ into disjoint intervals, in the sense that no ball starting in one interval visit any vertex of another interval before finding a hole. \addition{Hence, when one considers an interval $]X_i,X_{i+1}[$ between two holes $X_i$ and $X_{i+1}$, if one activates only the balls in this interval, the probability that no ball gets out of it does not depend on the rest of the configuration, and better than that, this probability would be the same if the size $n$ of the circle $Z/n\Z$ were different: we can then reduce this size as much as it is useful for us, and take n so small that $X_i=X_{i+1}\mod n$. So this probability is the same as the probability that balls do not reach a specified hole on a circle (corresponding to the identification of the two holes on the extremities of the block, see Figure \ref{fig:illu_lemme_decomposition}). }

\begin{lemma}\label{lemma:simpl_thm_loi_trous_residuels}
We use the same notation as in Theorem \ref{thm:loitrousresiduels}: we consider $X \subseteq {\znz}$ of size $\card{X} = \nl$, and for every $i \in \znlz$, we set $\ell_i \coloneqq \Delta^{(n)}_{i}X$. We also let $x$ be the nicely ordered sequence containing all the elements of $X$ (recall that it is such that $\pi(x_1) < \ldots < \pi(x_{k-1}) < \pi(x_0)$, where $\pi$ is the canonical projection $\znz \to \llbracket 1,n \rrbracket$).
For every finite sequence $(b_i)_{1\leq i \leq \nl}$ such that $\sum_{i}b_i = \nballs$ and $\forall i, 0 \leq b_i \leq \ell_i/2$, 
\begin{align}
	\pcercle\left(\TL = X ~\middle|~ \grosevtnom \right) &= \prod_{i \in \znlz} \pcerclei \left( \TL = \{0\} ~\middle|~ 0 \in \Tinit \right) \label{eqn:decomposition}
	\\ 	& = \prod_{i \in\znlz} \frac{1}{b_i + 1},
\end{align}
where we define, for any $H, B, X$ and $(b_i)_{1 \leq i \leq \nl}$ the event $\grosevtnomdet$ corresponding to the conjunction of the following events : \grosevtdefdet

\end{lemma}

Figure \ref{fig:illu_lemme_decomposition} illustrates this lemma. We prove it in Section \ref{subsect:lemme_loi_cercle}.

\begin{figure}[t]\centering
{\begin{tikzpicture}[line cap=round,line join=round,>=triangle 45,scale = 0.8]
		\draw(0,0) circle (2cm);
		\begin{scriptsize}
			\draw [color=black,fill = black] (-1.41,1.41) circle (1.5pt);
			\draw [color=black,fill = black] (-0.77,1.85) circle (1.5pt);
			\draw [color=black,fill = black] (-1.85,0.77) circle (1.5pt);
			\draw [color=black,fill = black] (-2,0) circle (1.5pt);
			\draw [color=black,fill = black] (-1.85,-0.77) circle (1.5pt);
			\draw [color=black,fill = black] (-1.41,-1.41) circle (1.5pt);
			\draw [color=black,fill = black] (-0.77,-1.85) circle (1.5pt);
			\draw [color=black,fill = black] (0,-2) circle (1.5pt);
			\draw [color=black,fill = black] (0.77,-1.85) circle (1.5pt);
			\draw [color=black,fill = black] (1.41,-1.41) circle (1.5pt);
			\draw [color=black,fill = black] (1.85,-0.77) circle (1.5pt);
			\draw [color=black,fill = black] (1.85,0.77) circle (1.5pt);
			\draw [color=black,fill = black] (2,0) circle (1.5pt);
			\draw [color=black,fill = black] (1.41,1.41) circle (1.5pt);
			\draw [color=black,fill = black] (0.77,1.85) circle (1.5pt);
			\draw [color=black,fill = black] (0,2) circle (1.5pt);
			\node[above] at (0,-1.8) {\redcolor{$0$}};
			\node [black,right, xshift = 7pt] at (-1.85,-0.77) {$x_0$};
			\node [black,above left, xshift = -5pt] at (1.41,-1.41) {$x_{1}$};
			\node [black, xshift = 11pt, yshift = -9pt] at (112.5:2) {$x_{2}$};
			
			\draw [color=black,fill = white] (112.5:2) circle (3.5pt);
			\draw [color=black,fill = white] (-1.85,-0.77) circle (3.5pt);
			\draw [color=black,fill = white] (1.41,-1.41) circle (3.5pt);
			\node [align=left] at (2.65,1.8) {  $ b_1 \times \hspace{-0.1cm}\vcenter{\hbox{ \begin{tikzpicture}[scale = 0.7]
							\draw [color=black,fill = white] (1.85,-0.77) circle (3.5pt);
				\end{tikzpicture} }}\hspace{-0.1cm} + b_1 \times \hspace{-0.1cm}\vcenter{\hbox{ \begin{tikzpicture}[scale = 0.7]
							\draw [color=black,fill = black] (1.85,-0.77) circle (3.5pt);
				\end{tikzpicture} }}\hspace{-0.1cm} $};
			\node [align = right, yshift = 5] at (-2.6,1.7) { 
				$  b_2 \times \hspace{-0.1cm}\vcenter{\hbox{ \begin{tikzpicture}[scale = 0.7]
							\draw [color=black,fill = white] (1.85,-0.77) circle (3.5pt);
				\end{tikzpicture} }}\hspace{-0.1cm} + b_2 \times\hspace{-0.1cm}\vcenter{\hbox{ \begin{tikzpicture}[scale = 0.7]
							\draw [color=black,fill = black] (1.85,-0.77) circle (3.5pt);
				\end{tikzpicture} }}\hspace{-0.1cm} $};
			\node [] at (-0.2,-2.5) { {
					$b_0 \times \hspace{-0.1cm}\vcenter{\hbox{ \begin{tikzpicture}[scale = 0.7]
								\draw [color=black,fill = white] (1.85,-0.77) circle (3.5pt);
					\end{tikzpicture} }}\hspace{-0.1cm} + b_0 \times \hspace{-0.1cm}\vcenter{\hbox{ \begin{tikzpicture}[scale = 0.7]
								\draw [color=black,fill = black] (1.85,-0.77) circle (3.5pt);
					\end{tikzpicture} }}\hspace{-0.1cm} $}};
			\fill[blue!60!cyan, opacity = 0.3] (0,0) ++(101.25:2.2) 
			arc(101.25:-30:2.2)
			-- ++(-210:0.4)
			arc(-30:101.25:1.8)
			-- cycle;
			\fill[blue, opacity = 0.3] (0,0) ++(187.5:2.2)
			arc(187.5:123.75:2.2)
			-- ++(-56.25:0.4)
			arc(123.75:187.5:1.8)
			-- cycle;
			
			\fill[cyan, opacity = 0.3] (0,0) ++(-60:2.2)
			arc(-60:-142.5:2.2)
			-- ++(37.5:0.4)
			arc(-142.5:-60:1.8)
			-- cycle;
			
			\draw [color=orange] (112.5:2) circle (7pt);
			\draw [color=orange] (-1.85,-0.77) circle (7pt);
			\draw [color=orange] (1.41,-1.41) circle (7pt);
			
		\end{scriptsize}
		\node [] at (3.2,-0) { {\Large $ = $}};

		\begin{scope}[shift = {(5.3,0)}]
			\def\ray{1}
			\draw(0,0) circle (\ray);
			\begin{scriptsize}
				\draw [color=black,fill = black] (54:\ray) circle (1.5pt);
				\draw [color=black,fill = black] (198:\ray) circle (1.5pt);
				\draw [color=black,fill = black] (126:\ray) circle (1.5pt);
				\draw [color=black,fill = black] (-18:\ray) circle (1.5pt);
				\draw [color=black,fill = black] (-90:\ray) circle (1.5pt);
				\node[above] at (0,-0.8*\ray) {\redcolor{$0$}};
				\draw [color=black,fill = white](-90:\ray) circle (3.5pt);
				
				\node [] at (-0,1.5) { {$b_0 \times \hspace{-0.1cm}\vcenter{\hbox{ \begin{tikzpicture}[scale = 0.7]
									\draw [color=black,fill = white] (1.85,-0.77) circle (3.5pt);
						\end{tikzpicture} }}\hspace{-0.1cm} + b_0 \times \hspace{-0.1cm}\vcenter{\hbox{ \begin{tikzpicture}[scale = 0.7]
									\draw [color=black,fill = black] (1.85,-0.77) circle (3.5pt);
						\end{tikzpicture} }}\hspace{-0.1cm} $}};
				
				\fill[cyan, opacity = 0.3] (0,0) ++(-54:\ray+0.2)
				arc(-54:234:\ray+0.2)
				-- ++(54:0.4)
				arc(234:-54:\ray-0.2)
				-- cycle;
				
				\draw [color=orange] (-90:\ray) circle (7pt);

			\end{scriptsize}
		\end{scope}
		
		\node [] at (7,0) {\Large {$\times$}};
		
		\begin{scope}[shift = {(8.9,0)}]
			\def\ray{1.2}
			\draw(0,0) circle (\ray);
			\begin{scriptsize}
				\draw [color=black,fill = black] (218.57:\ray) circle (1.5pt);
				\draw [color=black,fill = black] (167.14:\ray) circle (1.5pt);
				\draw [color=black,fill = black] (115.71:\ray) circle (1.5pt);
				\draw [color=black,fill = black] (64.28:\ray) circle (1.5pt);
				\draw [color=black,fill = black] (12.85:\ray) circle (1.5pt);
				\draw [color=black,fill = black] (-38.57:\ray) circle (1.5pt);
				\draw [color=black,fill = black] (-90:\ray) circle (1.5pt);
				\node[above] at (0,-0.8*\ray) {\redcolor{$0$}};
				
				\draw [color=black,fill = white](-90:\ray) circle (3.5pt);
				\node [] at (-0,1.7) { {$b_1\times \hspace{-0.1cm}\vcenter{\hbox{ \begin{tikzpicture}[scale = 0.7]
									\draw [color=black,fill = white] (1.85,-0.77) circle (3.5pt);
						\end{tikzpicture} }}\hspace{-0.1cm} + b_1\times \hspace{-0.1cm}\vcenter{\hbox{ \begin{tikzpicture}[scale = 0.7]
									\draw [color=black,fill = black] (1.85,-0.77) circle (3.5pt);
						\end{tikzpicture} }}\hspace{-0.1cm} $}};
				
				\fill[blue!60!cyan, opacity = 0.3] (0,0) ++(-64.28:\ray+0.2)
				arc(-64.28:244.28:\ray+0.2)
				-- ++(64.28:0.4)
				arc(244.28:-64.28:\ray-0.2)
				-- cycle;
				\draw [color=orange] (-90:\ray) circle (7pt);	
			\end{scriptsize}
		\end{scope}
		
		\node [] at (10.8,0) {\Large {$\times$}};
		
		\begin{scope}[shift = {(12.3,0)}]
			\def\ray{0.8}
			\draw(0,0) circle (\ray);
			\begin{scriptsize}
				\draw [color=black,fill = black] (180:\ray) circle (1.5pt);
				\draw [color=black,fill = black] (90:\ray) circle (1.5pt);
				\draw [color=black,fill = black] (0:\ray) circle (1.5pt);
				\draw [color=black,fill = black] (-90:\ray) circle (1.5pt);
				\node[above] at (0,-0.8*\ray) {\redcolor{$0$}};
				
				\draw [color=black,fill = white](-90:\ray) circle (3.5pt);
				\node [] at (-0,1.4) { {$b_2 \times \hspace{-0.1cm}\vcenter{\hbox{ \begin{tikzpicture}[scale = 0.7]
									\draw [color=black,fill = white] (1.85,-0.77) circle (3.5pt);
						\end{tikzpicture} }}\hspace{-0.1cm} + b_2 \times \hspace{-0.1cm}\vcenter{\hbox{ \begin{tikzpicture}[scale = 0.7]
									\draw [color=black,fill = black] (1.85,-0.77) circle (3.5pt);
						\end{tikzpicture} }}\hspace{-0.1cm} $}};
				
				\fill[blue, opacity = 0.3] (0,0) ++(-45:\ray+0.2)
				arc(-45:225:\ray+0.2)
				-- ++(45:0.4)
				arc(225:-45:\ray-0.2)
				-- cycle;
				\draw [color=orange] (-90:\ray) circle (7pt);	
			\end{scriptsize}
		\end{scope}
		
\end{tikzpicture}}
\hfill
\caption[]{Illustration of the block decomposition principle, and its use to compute $\Prob{\TL = X \middle| X \subseteq \Tinit \et \forall i, \cardBi = b_i}$ given by Lemma \ref{lemma:simpl_thm_loi_trous_residuels}. 
	Here, we give an example for $n = 16$, $\nballs = 5$ and $\nholes =8$. Moreover, we consider $X=\{x_0,x_1,x_2\}$ such that $\ell_0 = 4$, $\ell_1 = 6$ and $\ell_2 = 3$. 
	The orange circled vertices correspond to the elements of $X$. Recall that $(b_0,b_1,b_2)\in B^{(\nballs, X)} = \left\{(b_i)_{1\leq i \leq \nl} :\sum_{i}b_i = \nballs  \et \forall i, 0 \leq b_i \leq \ell_i/2 \right\}$. 
	The notation $b_i \times \hspace{-0.1cm}\vcenter{\hbox{ \begin{tikzpicture}[scale = 0.7]
				\draw [color=black,fill = white] (1.85,-0.77) circle (3.5pt);
	\end{tikzpicture} }}\hspace{-0.1cm} + b_i \times \hspace{-0.1cm}\vcenter{\hbox{ \begin{tikzpicture}[scale = 0.7]
				\draw [color=black,fill = black] (1.85,-0.77) circle (3.5pt);
	\end{tikzpicture} }}\hspace{-0.1cm} $
	means that the corresponding block (in blue, between two consecutive elements of $X$) contains exactly $b_i$ balls and $b_i$ holes (assigned uniformly on the vertices of the block).
	Since no ball ever visits a hole of $\TL$, the elements of $\TL$ divide $\znz$ into ``independent'' blocks that behave as mini-golfs, and this allows to decompose the probability that $\TL = X$ into a product of probabilities that mini-golfs leave 0 free.
}
\label{fig:illu_lemme_decomposition}
\end{figure}

In what follows we identify $\Z/\nl\Z$ with $\llbracket 1, \nl \rrbracket$, and use the two notation equivalently. We now use these two lemmas to prove Theorem \ref{thm:loitrousresiduels}.

\begin{proof}[Proof of Theorem \ref{thm:loitrousresiduels}]
{We work under $\pcercle$, but we write $\P$ for short, since there is no ambiguity in this proof.}

Let $X$ be given, and recall that the set \[B^{(\nballs, X)} = \left\{(b_i)_{1\leq i \leq \nl} :\sum_{i}b_i = \nballs  \et \forall i, 0 \leq b_i \leq \ell_i/2 \right\}\] corresponds to the set of possible allocations of $\nballs$ balls into $\nl$ intervals, without having more than $\ell_i/2$ balls in the $i$th interval. 

The event $\TL = X$ can occur only if $X \subseteq \Tinit$ and there are exactly as many holes as balls in every block. It gives equations (\ref{eqn:thm_loi_trous_residuels_cercle_1}) and (\ref{eqn:thm_loi_trous_residuels_cercle_2}) in the computation below.

\begin{align}\pcercle&\left(\TL =X \right) =\sum_{(b_i)_{1\leq i \leq \nl} \in B^{(\nballs, X)}
	} \P\left(\TL = X \cap \grosevtnom \right) \label{eqn:thm_loi_trous_residuels_cercle_1} \\
	&\hspace{-3mm}= \sum_{(b_i)_{1\leq i \leq \nl} \in B^{(\nballs, X)}
	} \P\left(\TL = X\middle| \grosevtnom \right)\P\left(\grosevtnom\right)	\label{eqn:thm_loi_trous_residuels_cercle_2}
\end{align}

A counting argument gives, for any $b_i$ in $B^{(\nballs, X)}$,
\begin{equation}
	\P\left(\grosevtnom\right) = \frac{\prod_{i = 1}^{\nl} \binom{\ell_i}{b_i, b_i, \ell_i - 2 b_i}}{\binom{n}{\nholes, \nballs, n - \nholes - \nballs}}. \label{eqn:loiTL_cercle_aux_comptage}
\end{equation}

Finally, applying Lemma \ref{lemma:simpl_thm_loi_trous_residuels} and Equation (\ref{eqn:loiTL_cercle_aux_comptage}) to Equation (\ref{eqn:thm_loi_trous_residuels_cercle_2}) gives the right member of Equation (\ref{eqn:loi_TL_cercle}) and thus concludes the proof of Theorem \ref{thm:loitrousresiduels}.

\end{proof}

\subsection{Proof of Lemma \ref{lemma:simpl_thm_loi_trous_residuels}}\label{subsect:lemme_loi_cercle}

This section is devoted to the proof of Lemma \ref{lemma:simpl_thm_loi_trous_residuels}. 
We say that an interval $\llbracket a,b \rrbracket$ is \textit{full} if every ball initially in the interval $\llbracket a,b \rrbracket$ fills a hole inside $\llbracket a,b \rrbracket$, and no hole inside $\llbracket a,b \rrbracket$ is free at the end of the process. 
The proof principle is rather simple: since no ball ever visit any remaining hole, the probability to reach a given final configuration is the product of the probabilities that each interval between two consecutive final holes is full. 

To prove it formally, we introduce and study \textit{golf sequences}, which, as we mentioned in the introduction, were introduced by Fredes and Marckert \cite{fredes2021aldousbroder}. We borrow some of their notation, but for completeness we explain everything here. Golf sequences are informally the sequences of trajectories of the balls that respect the rules we gave (i.e.\ there are $\nballs$ trajectories, each trajectory starting at a vertex initially containing a ball and ending at the first hole that have not been reached by a previous trajectory).\\

\textbf{Golf sequences} can be defined on any connected graph $G = (V,E)$. As before, $\PMC$ is the transition matrix of an irreducible Markov chain on $G$.

Assume that $\Ballset^0$ and the activation clocks $\left(\Clock(v), v\in \Ballset^0\right)$ are given. Build the sequence of balls $S = (S_1, \ldots S_{\nballs})$ containing all the elements of $\Ballset^0$, ordered by activation time. In what follows we use this notation.

For a set of holes $\Holesetdet$ (such that, as usual, $\card{\Holesetdet}\geq \card{\Ballset^0}$) and any ordering of the balls $\Sources = (S_1, \ldots, S_{\nballs})$, we define the set of \textit{golf sequences}, $\golfsequences$, as follows. A sequence of paths $(w^{(1)}, \ldots, w^{(\nballs)})$ of respective lengths $(L_1, \ldots, L_{\nballs})$ belongs to $\golfsequences$ if and only if, for every $i\leq \nballs$ :
\begin{itemize}
\item the $i$th ball starts from $S_i$, i.e. $w^{(i)}_0 = S_i$, and follows the path $w_0^{(i)},\ldots,w^{(i)}_{L_i}$, where each step $(w^{(i)}_k,w_{k+1}^{(i)})$ is an edge of $G$,
\item the $i$th ball stops at the first hole that has not been filled yet, i.e. $\forall k< L_i, w^{(i)}_k \notin \Holesetdet\backslash\{w^{(j)}_{L_j}, j < i\}$ and $w^{(i)}_{L_i} \in \Holesetdet\backslash\{w^{(j)}_{L_j}, j < i\}$. 

\end{itemize}
Moreover, for a sequence $w\in \golfsequences$, we define its \textit{weight} as follows: \[\weight(w) \coloneqq \prod_{i = 1}^{\nballs} \prod_{j = 0}^{L_i - 1} \PMC
_{w^{(i)}_j, w^{(i)}_{j+1}}.\]
It corresponds to the probability that all the balls starting successively from $S_1$ to $S_{\nballs}$, and doing one by one a random walk stopped when hitting a still available hole, perform exactly the sequence of paths $w$. We can then extend the definition of the weight to a set of sequences, setting, for every $W \subseteq \golfsequences$, \[\weight(W) \coloneqq \sum_{w\in W} \weight(w).\]
In other words, $\weight$ can be seen as a probability measure on the set of golf sequences, and in particular, since our process is valid, $\weight(\golfsequences) = 1$ (thanks to Proposition \ref{prop:bonne_def_graphe_fini_ET_commutation}, \addition{and because the weight of $\golfsequences$ is exactly the probability that every ball finds a hole in finite time in the corresponding golf process}
).

We define 
\begin{equation}
\golfsequencesrestricted \coloneqq \left\{ w \in \golfsequences ~:~ \Holesetdet \backslash \{w^{(i)}_{L_i}, 1\leq i \leq \nballs\} = \TLdet\right\} \label{eqn:golf_sequences}
\end{equation}
the subset of $\golfsequences$ that leave free exactly the subset $\TLdet \subseteq \Holesetdet$ (thus with cardinality $\nl = \card{H} - \nballs$).

Then, \begin{equation}
\weight(\golfsequencesrestricted) = \P(\TL = \TLdet) , \label{eqn:truc_utile_etoui}
\end{equation}
where the probability is for our golf model on the same graph $G$, the same transition matrix $P$, with deterministic initial condition given by $H$ and $S$.

\paragraph{From golf to mini-golf: Block decomposition of the golf sequences on $\znz$.}
We now focus again on the golf on $\znz$.
Until the end of this section, we use the same notation as in Theorem \ref{thm:loitrousresiduels} and Lemma \ref{lemma:simpl_thm_loi_trous_residuels}: we consider $X \subseteq {\znz}$ of size $\card{X} = \nl$, and for every $i \in \znlz$, $\ell_i = \Delta^{(n)}_{i}X$. We also let $x$ be the nicely ordered sequence containing all the elements of $X$.

Lemma \ref{lemma:decompositionnongeneraleenhistoires} decomposes golf sequences on the cycle $\znz$ into a family of golf sequences on intervals between two consecutive holes, which are also golf sequences on smaller cycles (and are actually ``mini-golf sequences'', since they leave only one remaining hole, see Figure \ref{fig:decomposition} that illustrates this last point).

\begin{lemma}[Block decomposition of the golf sequences on $\znz$] \label{lemma:decompositionnongeneraleenhistoires}
We let $(b_i)_{1\leq i \leq \nl}$ be such that $\sum_{i}b_i = \nballs$ and $\forall i, 0 \leq b_i \leq \ell_i/2$, and we consider $\Ballsetdet, \Holesetdet \subseteq \znz$ such that $X \subseteq \Holesetdet$ and $\forall i, \cardBidet = \cardTidet = b_i$. We let $S$ be an ordering of $\Ballsetdet$.

Then, there exists a weight-preserving bijection
\[\fonction{f}{\golfsequencesrestrictedX}{\prod_{i \in \znlz} \golfsequencesrestrictedO}{W}{(W^1,\dots,W^{\nl})}\]
where:
\begin{itemize}
	\item for every $i \in \znlz$, $\Holesetdet_{i}$ is the restriction of $\Holesetdet$ to elements of $\Ii$, shifted by $-x_i$ (formally, $\Holesetdet_{i} = \{h - x_i : h \in \Holesetdet \cap \Ii \}$ seen as a subset of $\Z/{(\ell_i+1)}\Z$, that contains 0),
	\item  for every $i \in \znlz$,  $\Sources_{i}$ is the restriction of $\Sources$ to elements of $\Ii$, also shifted by $-x_i$.
	\item The weight of $(W^1,\dots,W^{\nl})$ is $\prod_{i = 1}^{\nl} \weight(W^i)$.
\end{itemize}
Moreover, the map $(H,S) \mapsto \prod_{i \in \znlz} (H_i,S_i)$  is also a bijection.

\end{lemma}

\begin{proof}[Proof of Lemma \ref{lemma:decompositionnongeneraleenhistoires}]
Showing that $(H,S) \mapsto \prod_{i \in \znlz} (H_i,S_i)$ is a bijection is immediate. We now focus on $f$.

This lemma is a deterministic version of the block decomposition principle (see Section \ref{subsect:znzloiTL}), and allows to prove this principle: we decompose golf sequences on $\znz$ leaving $\TLdet$ free into golf sequences on intervals between elements of $\TLdet$. Then, golf sequences on an interval with two free holes on its border are trivially in bijection with golf sequences on a cycle with vertex 0 free (see Figure \ref{fig:decomposition}).

\begin{figure} \centering
	\begin{tikzpicture}[line cap=round,line join=round, x=1.0cm,y=1.0cm,scale = 0.24]
		\usetikzlibrary{arrows.meta}
		\tikzset{>={Latex[width=1.7mm,length=1.7mm]}}
		\clip(-29,-8.5) rectangle (29,14);
		\draw(0,5.49) circle (5.85cm);
		\draw(0,25.01) circle (30cm);
		\begin{scriptsize}
			\draw [color=black,fill = lightgray] (2,0) circle (17pt);
			\draw [color=black,fill = black] (-2,0) circle (17pt) node[yshift = -3.5mm] {$s_1$};
			\draw [color=black,fill = black] (5.06,2.57) circle (17pt) node[yshift = -3mm, xshift = 2mm] {$s_4$};
			\draw [color=black,fill = lightgray] (5.76,6.51) circle (17pt);
			\draw [color=black,fill = black] (3.76,9.97) circle (17pt)node[yshift = 0mm, xshift = 3.5mm] {$s_2$};
			\draw [color=black,fill = white] (0,11.34) circle (17pt);
			\draw [color=orange] (0,11.34) circle (35pt);
			\draw [color=black,fill = lightgray] (-3.76,9.97) circle (17pt);
			\draw [color=black,fill = lightgray] (-5.76,6.51) circle (17pt);
			\draw [color=black,fill = black] (-5.06,2.57) circle (17pt) node[yshift = -3mm, xshift = -2mm] {$s_3$};
			\draw [color=black,fill = black] (-2.86,-4.86) circle (17pt) node[yshift = -3.5mm, xshift = -0mm] {$s_1$};
			\draw [color=black,fill = lightgray] (2.86,-4.86) circle (17pt);
			\draw [color=black,fill = black] (8.46,-3.77) circle (17pt) node[yshift = -3.5mm, xshift = 2mm] {$s_{12}$};
			\draw [color=black,fill = lightgray] (13.77,-1.65) circle (17pt);
			\draw [color=black,fill = black] (18.57,1.45) circle (17pt) node[yshift = -3mm, xshift = 2mm] {$s_8$};
			\draw [color=black,fill = black] (-8.46,-3.77) circle (17pt) node[yshift = -3mm, xshift = -2mm] {$s_7$};
			\draw [color=black,fill = lightgray] (-13.77,-1.65) circle (17pt);
			\draw [color=black,fill = lightgray] (-18.57,1.45) circle (17pt);
			\draw [color=black,fill = white] (-22.7,5.39) circle (17pt);
			\draw [color=black,fill = white] (22.7,5.39) circle (17pt);
			\draw [color=orange] (-22.7,5.39) circle (35pt) node[yshift = -3mm, xshift = -4.5mm, color = black] {$x_i$};
			\draw [color=orange] (22.7,5.39) circle (35pt) node[yshift = -3mm, xshift = 5.5mm, color = black] {$x_{i+1}$};
			\draw [line width = 0.7pt,->, red, shorten <= 9pt ,shorten >=9pt] (22.7,5.39) to[bend right] (0,11.34);
			\draw [line width = 0.7pt,->, red, shorten <= 9pt ,shorten >=9pt] (-22.7,5.39) to[bend left] (0,11.34);
			\draw[line width = 0.7pt,cyan,dotted,->,shorten <= 4pt + \pgflinewidth,shorten >=4pt + \pgflinewidth] (-2,0) to[bend right] (-5.06,2.57);
			\draw[line width = 0.7pt,cyan,dotted,->,shorten <= 4pt + \pgflinewidth,shorten >=4pt + \pgflinewidth] (-2.86,-4.86) to[bend right] (-8.46,-3.77);
			\draw[line width = 0.7pt,cyan,dotted,->,shorten <= 4pt + \pgflinewidth,shorten >=4pt + \pgflinewidth] (-5.06,2.57) to[bend right=15] (-5.76,6.51);
			\draw[line width = 0.7pt,cyan,dotted,->,shorten <= 4pt + \pgflinewidth,shorten >=4pt + \pgflinewidth] (-8.46,-3.77) to[bend right=15] (-13.77,-1.65);
			\draw[line width = 0.7pt,green,dotted,->,shorten <= 4pt + \pgflinewidth,shorten >=4pt + \pgflinewidth] (-5.06,2.57) to[bend right=35] (-5.76,6.51);
			\draw[line width = 0.7pt,green,dotted,->,shorten <= 4pt + \pgflinewidth,shorten >=4pt + \pgflinewidth] (-8.46,-3.77) to[bend right=35] (-13.77,-1.65);
			\draw[line width = 0.7pt,green,dotted,->,shorten <= 4pt + \pgflinewidth,shorten >=4pt + \pgflinewidth] (-5.76,6.51) to[bend right](-3.76,9.97);
			\draw[line width = 0.7pt,green,dotted,->,shorten <= 4pt + \pgflinewidth,shorten >=4pt + \pgflinewidth] (-13.77,-1.65) to[bend right] (-18.57,1.45);
			\draw[line width = 0.7pt,Emerald,dotted,->,shorten <= 4pt + \pgflinewidth,shorten >=4pt + \pgflinewidth] (3.76,9.97) to[bend right](5.76,6.51);
			\draw[line width = 0.7pt,Emerald,dotted,->,shorten <= 4pt + \pgflinewidth,shorten >=4pt + \pgflinewidth] (18.57,1.45) to[bend right] (13.77,-1.65); 
			\draw[line width = 0.7pt,green!20!cyan,dotted,->,shorten <= 4pt + \pgflinewidth,shorten >=4pt + \pgflinewidth] (5.06,2.57) to[bend left](5.76,6.51);
			\draw[line width = 0.7pt,green!20!cyan,dotted,->,shorten <= 4pt + \pgflinewidth,shorten >=4pt + \pgflinewidth] (8.46,-3.77) to[bend left] (13.77,-1.65);
			\draw[line width = 0.7pt,green!20!cyan,dotted,->,shorten <= 4pt + \pgflinewidth,shorten >=4pt + \pgflinewidth] (5.76,6.51) to[bend right=50](5.06,2.57);
			\draw[line width = 0.7pt,green!20!cyan,dotted,->,shorten <= 4pt + \pgflinewidth,shorten >=4pt + \pgflinewidth] (13.77,-1.65) to[bend right=50] (8.46,-3.77);
			\draw[line width = 0.7pt,green!20!cyan,dotted,->,shorten <= 4pt + \pgflinewidth,shorten >=4pt + \pgflinewidth] (5.06,2.57) to[bend right](2,0);
			\draw[line width = 0.7pt,green!20!cyan,dotted,->,shorten <= 4pt + \pgflinewidth,shorten >=4pt + \pgflinewidth] (8.46,-3.77) to[bend right] (2.86,-4.86);
		\end{scriptsize}
	\end{tikzpicture}
	\caption{From golf to mini-golf: Correspondence between the restriction of golf sequences between two holes (at distance $\ell+1$) of a large cycle and golf sequences on a cycle of size $\ell+1$ with one remaining hole.\\ Here $\ell = 8$. A black vertex corresponds to the initial position of a ball, a white vertex to a free hole, and a grey vertex to a hole that is occupied by a ball at the end. A dotted arrow associates a ball to the hole it filled. Each dotted arrow on the large circle corresponds to a dotted arrow on the small circle, illustrating the bijection defined in Lemma \ref{lemma:decompositionnongeneraleenhistoires}.\\ The restriction of $S = (s_1,\ldots, s_7,s_8,\ldots,s_{12},\ldots)$ to the interval between $x_i$ and $x_{i+1}$ gives $S_i = (s_1,s_2,s_3,s_4)$.
	}
	\label{fig:decomposition}
\end{figure}

Recall that the $i$th block is, here, the interval $\llbracket x_i+1, x_{i+1}-1\rrbracket$. We let $W = (w^{(1)}, \ldots, w^{(\nballs)}) \in  \golfsequencesrestrictedX$. For every $i$, we build the $i$th golf sequence on $\Z/(\ell_i+1)\Z$ as the restriction of $W$ to the $i$th block. More precisely,
we let $W^i$ be the golf sequence that contains the paths of the balls that belong to the $i$th block, shifted so that they now belong to $\Z/\ell_i\Z$: 
\[ W^i = \left( (w^{(j)}_k - x_i)_{0 \leq k \leq L_j} \right)_{j\in \alpha^i}\]
where $\alpha^i$ is the sorted sequence of the indices of the balls that belong to the $i$th block: $\exists j : \alpha^i = j \iff S_j \in \llbracket x_i, x_{i+1}\rrbracket$.

We can easily verify that $(W^1,\dots,W^{\nl}) \in \prod_{i \in \znlz}\golfsequencesrestrictedO  $. In fact, the important point to note is that on any interval $\llbracket x_i+1, x_{i+1}-1 \rrbracket$, the order of the balls is the one induced by the order of the balls on $\znz$, so in particular the defined trajectories still end on the first free hole they reach. 

The application $f$ is clearly bijective, and it is weight-preserving, since the weight of a trajectory only depends on the total number of steps to the right and to the left, so simple computation gives $\weight(W) = \prod_{k = 1}^{\nl} \weight(W^k)$.

\end{proof}

Now we can conclude by proving Lemma \ref{lemma:simpl_thm_loi_trous_residuels}. 

\begin{proof}[Proof of Lemma \ref{lemma:simpl_thm_loi_trous_residuels}]

Recall that we simply write $\P$ for $\pcercle$, and otherwise we use the notation $\pcerclei$.

We fix an arbitrary permutation $\sigma$ on $\znz$. 	
To any set of balls $B \subseteq \znz$, we associate $S(B) = S(B,\sigma)$ the sequence of the elements of $B$ arranged in increasing order (with respect to $\sigma$). It corresponds to the order according to which the balls will be activated: during the whole computation below, we assume that the order induced by the clocks is compatible with $\sigma$, and we write the assumption ${\sf Order}(\Clock) = \sigma$. Then, $\sigma_{[x_i,x_{i+1}]}$ is the permutation of $\Z/\ell_i\Z$ induced by the restriction of $\sigma$ to $\llbracket x_i, x_{i+1}\rrbracket$ (so that $S(B_{i},\sigma_{[x_i,x_{i+1}]})$ is the sequence of elements of $b_i = B \cap \llbracket x_i, x_{i+1}\rrbracket$ arranged in increasing order with respect to $\sigma$, and then shifted by $-x_i$).

We give a sequence of identities now, and we will provide some explanations after. Fix some $(b_i)_{1\leq i \leq \nl}$ such that $\sum_{i}b_i = \nballs$ and $\forall i, 0 \leq b_i \leq \ell_i/2$. We have
\begin{align}
	\P&\left(\TL = X\middle| \grosevtnom{ \et {\sf Order}(\Clock) = \sigma} \right)\label{eqn:ici0}\\
	&= \frac{1}{\prod_{i=1}^{\nl} \binom{\ell_i}{b_i, b_i, \ell_i-2b_i}} \sum_{H,B:\ \grosevtnomdet} \P \left(\TL = X \middle| \substack{\Tinit = H, \Binit = B \\{ \et {\sf Order}(\Clock) = \sigma}}  \right) \label{eqn:ici1}\\
	&= \frac{1}{\prod_{i=1}^{\nl} \binom{\ell_i}{b_i, b_i, \ell_i-2b_i}} \sum_{H,B:\ \grosevtnomdet}  \sum_{w \in {\sf GolfSequences}^{\znz}\left[ H, S(B,\sigma), X \right]} \weight(w) \label{eqn:preuve_cercle_avant_utilisation_lemme}\\
	&= \frac{1}{\prod_{i=1}^{\nl} \binom{\ell_i}{b_i, b_i, \ell_i-2b_i}} \prod_{i=1}^{\nl} \sum_{\substack{H_i,B_i \subseteq \Z/\ell_i \Z : \\0 \in H_i \et \card{H_i} = \card{B_i} = b_i}} \sum_{w_i \in {\sf GolfSequences}^{\Z/(\ell_i+1) \Z}\left[ H_{i}, S(B_{i},\sigma), \{0\} \right]} \hspace{-0.1mm} \weight(w_i)  \label{eqn:preuve_cercle_utilisation_lemme}\\
	&= \prod_{i=1}^{\nl} \frac{1}{\binom{\ell_i}{b_i, b_i, \ell_i-2b_i}} \sum_{\substack{H_i,B_i \subseteq \Z/\ell_i \Z : \\0 \in H_i \et \card{H_i} = \card{B_i} = b_i}} \pcerclei \left(\TL = \{0\} \middle|  \substack{\Tinit = H_i, \Binit = B_i \\{ \et {\sf Order}(\Clock) = \sigma_{[x_i,x_{i+1}]}}} \right)\label{eqn:ici2-}\\
	&=\prod_{i = 1}^{\nl} \pcerclei \left(\TL = \{0\} ~\middle| ~0 \in \Tinit{ \et {\sf Order}(\Clock) = \sigma_{[x_i,x_{i+1}]}}\right) \label{eqn:ici2}.
\end{align}

\paragraph{Details on the computation.} We obtain (\ref{eqn:ici1}) from (\ref{eqn:ici0}) by noting that every initial configuration (corresponding to the event $\{\Tinit = H, \Binit = B\}$, for every $H,B$ appearing in the sum in (\ref{eqn:ici1})) has the same probability, and using the fact that if $A,B,(B_i)_i$ are events such that $B = \sqcup_{i=1}^K B_i$ (it is a disjoint union) and $\forall i,j, \P(B_i) = \P(B_j)$, then $  \frac{1}{K} \sum \P(A | B_i) = \frac{1}{K \P(B_1)} \sum \P(A \cap B_i) = \frac{1}{K} \frac{\P(A|B) \P(B)}{\P(B_1)}= \P(A|B) $. 

To go from (\ref{eqn:ici1}) to (\ref{eqn:preuve_cercle_avant_utilisation_lemme}), we use Equation (\ref{eqn:truc_utile_etoui}) that gives the correspondence between weights and probabilities.

The passage from (\ref{eqn:preuve_cercle_avant_utilisation_lemme}) to (\ref{eqn:preuve_cercle_utilisation_lemme}) uses Lemma \ref{lemma:decompositionnongeneraleenhistoires}.

We then recognize mini-golfs and do the same in the other way round.

\paragraph{Conclusion.}We can then use the Commutation property (see Remark \ref{remark:prop_commutation}) to conclude directly that
\begin{align}
	\P & \left(\TL = X  \middle| \grosevtnom\right)
	=\prod_{i = 1}^{\nl} \pcerclei \left(\TL = \{0\} \middle| 0 \in \Tinit\right)
	=\prod_{i = 1}^{\nl} \frac{1}{b_i + 1}. \label{eqn:conclu_loi_TL_lemme_simpl_thm_loi_trous_residuels}
\end{align}

\end{proof}

\subsection{Proof of Theorem \ref{thm:parking_loi_TL}}\label{subsect:parking_loi_TL}

The proof of Theorem \ref{thm:parking_loi_TL} follows exactly the same reasoning as the proof of Theorem \ref{thm:loitrousresiduels}; we give the main arguments and computation here.

\begin{proof}
As in the statement of Theorem \ref{thm:parking_loi_TL}, we consider the $p$-parking model with parameters $n$, $\nballs$ and $p$. Recall that here, the set of balls $\Binit$ is a multiset of vertices of $\znz$. When $\nballs = n-1$, i.e.\ there is only one remaining hole at the end, then we have, as an analogue of Lemma \ref{lemma:casinitloitrousresiduels} for the mini-golf case, $\forall x \in \znz$,
\begin{equation}
	\pcercleparking\left(\TL = \{x\}\middle| x \notin \Binit \right) = \frac{n^{n-2}}{(n-1)^{n-1}}. \label{eqn:mini_parking2}
\end{equation}

The proof is similar to the proof of Lemma \ref{lemma:casinitloitrousresiduels}: invariance by rotation of the process gives $\pcercleparking\left(\TL = \{x\}\right) = \frac{1}{n}$, and Equation (\ref{eqn:mini_parking2}) then comes from the fact that \linebreak $\pcercleparking\left( x \notin \Binit \right) = \left(\frac{n-1}{n}\right)^{n-1}$.

For general $n$, $\nballs$ and $X \subseteq {\znz}$ of size $\card{X} = \nl = n - \nballs$, recall that we set $\ell_i \coloneqq \Delta^{(n)}_{i}X$ for every $i \in \Z/\nl \Z$. Then, we have:
\begin{align}
	\pcercleparking\left(\TL = X\right) &= \P\left(\TL = X \bigcap  \substack{X \cap \Binit = \emptyset \\ \et \forall i, \cardBi = b_i} \right)\\
	&= \P\left(\TL = X \middle|  \substack{X \cap \Binit = \emptyset \\ \et \forall i, \cardBi = b_i} \right)  \P\left( \substack{X \cap \Binit = \emptyset \\ \et \forall i, \cardBi = b_i} \right),
\end{align}
(where $\P$ refers to $\pcercleparking$).

Then, we can do the same decomposition as we did in the proof of Lemma \ref{lemma:simpl_thm_loi_trous_residuels} (we do not prove it again, but the same proof with golf sequences works perfectly here too):
\begin{align}
	\P\left(\TL = X \middle|  \substack{X \cap \Binit = \emptyset \\ \et \forall i, \cardBi = b_i} \right)& = \prod_{i\in \Z/\nl\Z} \pcercleparkingi\left(\TL = \{0\}\middle| 0 \notin \Binit \right)\\
	&= \prod_{i\in \Z/\nl\Z}\frac{(\ell_i+1)^{\ell_i - 1}}{{\ell_i}^{\ell_i}}.
\end{align}
Recall that $\Binit$ is characterized by the vector $W = (W_1,\ldots, W_n)$, such that $W_k$ counts the multiplicity of $k$ in $\Binit$ ($W$ has the multinomial distribution with parameter \linebreak $(\nballs; 1/n, \ldots, 1/n)$). Using Equation (\ref{eqn:mini_parking2}) and the fact that $\P\left( \substack{X \cap \Binit = \emptyset \\ \et \forall i, \cardBi = b_i} \right) = \binom{\nballs}{\ell_1,\ldots, \ell_{\nl}}\prod_{i \in \Z/\nl\Z} \left(\frac{\ell_i}{n}\right)^{\ell_i}$, we finally obtain:

\begin{align}
	\pcercleparking\left(\TL = X\right) = \frac{1}{n^{\nballs}}
	\binom{\nballs}{\ell_1,\ldots,\ell_{\nl}} 
	\prod_{i \in \Z/\nl\Z}^{} \left(\ell_i + 1\right)^{\ell_i-1}.
\end{align}
\end{proof}

We can notice that when $n = \nballs + 1$ and $p=1$, Equation (\ref{eqn:mini_parking2}) can be seen as a reformulation of the fact that the number of parking functions on the line of length $k$ is {$(k+1)^{k-1}$} (see \cite{Konheim1966AnOD}, or \cite{RIORDAN1969408} for a simpler proof, sometimes known as Pollak's proof, which we summarize here): on $\znz$, the probability that vertex 0 is free at the end is equal to the number of parking functions on $\{1,\ldots,n-1\}$ (leaving $n$ free, thus equal to $n^{n-2}$) divided by the total number of assignments of the balls (which is equal to $n^{n-1}$, each of the $n-1$ balls choosing a starting vertex among the $n$ vertices).

\subsection{Critical window: combinatorial preliminaries for the proof of Theorem \ref{thm:phaseT}}\label{subsect:cercle_asymptotiques_preliminaries}

Before giving the proof of Theorem \ref{thm:phaseT}, we discuss some useful combinatorial facts, that will lead us to deduce block-size statistics in the golf process, from the study of forests and of finite paths. 
\paragraph{From block sizes to sizes of trees in a marked forest.}
Catalan numbers $C_m=\binom{2m}m/(m+1)$ are known to enumerate various combinatorial objects, including rooted planar trees with $m$ edges, Dyck paths with $2m$ steps, as well as rooted complete binary trees with $m$ internal nodes (i.e. trees with a total of $2m+1$ nodes, such that every internal node has exactly 2 children).

We let $\forests{n}{k}$ be the set of forests $(f_0, \ldots, f_{k-1})$ with $k$ rooted complete binary trees and a total of $n$ nodes. We call \textit{tree-size vector} of a forest $f = (f_0, \ldots, f_{k-1})\in \forestsnr$ the sequence
\[\card{f} = (\card{f_0},\ldots, \card{f_{k-1}}),\]
where $\card{f_i}$ is the number of nodes of the tree $f_i$. We can note that $\sum_{i=0}^{k-1} \card{f_i} = n$.

In all what follows, we will have $k=\nl$ and $n = 2\nballs +\nl$, so that the parameters correspond to those of the golf model.
In particular, if $\sum_j (2b_j + 1) = n$, then there are $\prod_{j=0}^{k-1} C_{b_j}$ forests of $\forests{n}{k}$ such that the $j$th tree has $(2b_j + 1)$ nodes, for every $j$.

It is known that $\card{\forestsnr} = \frac{k}{n} \binom{n}{\frac{n-k}{2}}$ (this formula is folklore in combinatorics, and can be computed using exactly the same reasoning as Pitman did in the case of plane forests without binarity condition \cite[section 6.3]{pitman2006combinatorial}).

We also define $\markedforests{n}{k}$ as the set of forests $(f^{\bullet}_0, f_1, \ldots, f_{k-1})$ such that $f^{\bullet}_0 = (f_0, v)$ is a tree $f_0$ marked at one of its nodes $v$, and $(f_0, \ldots, f_{k-1})$ is a forest of $\forests{n}{k}$. For every family $b_0,\ldots, b_{k-1}$ such that $\sum_{j}(2b_j +1) = n$, there are 
\begin{equation}
(2b_0+1)\prod_{j=0}^{k-1} C_{b_j} \label{eqn:nb_de_forets_avec_cette_taille}
\end{equation} forests of $\markedforests{n}{k}$ such that the $i$th tree has $b_i$ internal nodes (hence $2b_i+1$ nodes). The additional factor $2b_0+1$ is the number of ways to choose a vertex in $f_0$. 

In the sequel, we let $\randforestsnr^{\bullet}$ be a random forest taken uniformly in $\markedforests{n}{k}$. 

We now go back to our {golf model} on the circle. We have the following results, which is a corollary of Theorem \ref{thm:loitrousresiduels}, and justifies our focus on forests: 

\begin{corollary}\label{cor:blocks_and_forests}
We consider a golf process under the distribution $\pcercle$, when $\nballs + \nholes = n$ and $\nl = \nholes - \nballs$. The distribution of the block sizes vector in this model corresponds to the tree sizes in the uniform random forests $\randforests{n}{\nl}^{\bullet}$:
\begin{equation*}\left({\Deltan_{i} \TL+1 }\right)_{0 \leq i < \nl} \overset{d}{=} \card{\randforests{n}{\nl}^{\bullet}} .\end{equation*}
\end{corollary}

Notice here that the the block containing 0 is size-biased.
\begin{proof}
When $n = \nballs + \nholes$, Theorem \ref{thm:loitrousresiduels} rewrites, for any $X$ such that $\card{X} = \nholes - \nballs$,
\[ \pcercle\left(\TL = X\right) = \frac{1}{{\binom{n}{\nballs}}} \prod_{i} \frac{1}{b_i + 1} \binom{2b_i}{b_i},\]
where for every $i$, $b_i = \ell_i/2 = (\Deltan_i X )/2$. We can notice here that this probability is proportional to $\prod_{i} C_{b_i}$. 

Now we set $b_0, \ldots,b_{\nl-1}$ such that $\sum b_i = \nballs$. There are exactly $2b_0 +1$ sets $X$ such that $\forall i, b_i = (\Deltan_i X) /2 $ (this is due to the fact that we indexed the process $\Deltan X$ so that $\Deltan_0 X$ is the number of points in the block that ``contains 0'').
We thus finally obtain
\begin{equation}
	\pcercle\left(\forall i, \Deltan_i \TL = 2b_i\right) = \frac{2b_0 + 1}{\binom{n}{\nballs}} \prod_{i} \Catalan{b_i}. \label{eqn:jolie_formule_deltaTL_znz}
\end{equation}	 

Moreover, the probability $\Prob{\card{\randforests{n}{\nl}^{\bullet}} = (2b_0+1,\ldots, 2b_{\nl-1}+1)}$ is proportional to $(2b_0+1)\prod_{j=0}^{\nl-1} C_{b_j}$ (see Equation (\ref{eqn:nb_de_forets_avec_cette_taille})), so finally, since proportional probability distributions are equal\footnote{since each of them has total mass equal to 1},
\[ \Prob{\card{\randforests{n}{\nl}^{\bullet}} = (2b_0+1,\ldots, 2b_{\nl-1}+1)}=\pcercle\left(\forall i, \Deltan_i \TL = 2b_i\right) .\]

\end{proof}

\paragraph{From marked forests to forests.}
Corollary \ref{cor:blocks_and_forests} links block sizes and sizes of trees in a marked forest (which thus exhibit a size bias). It can be useful here to understand further the role played by the mark, which is a bit subtle.
In the next lemma, we assert that the sorted tree sizes of $\randforestsnr$ and of $\randforestsnr^{\bullet}$ have the same distribution: the reason is that it is equivalent to draw a marked forest uniformly at random in $\markedforests{n}{k}$, and to draw a forest uniformly at random in $\forests{n}{k}$, to mark one of its nodes uniformly in the whole forest and shift the tree indices in $\Z/ k\Z$ so that this node belongs to the first tree. Formally, we let $\randforestsnr = ({\bf f}_0,\cdots,{\bf f}_{k-1})$ be a forest taken uniformly in $\forests{n}{k}$, $\bf u$ be a node chosen uniformly in $\randforestsnr$, and $\bf i$ be the index of the tree ${\bf f_i}$ that contains $\bf u$. The forest $R_{\bf u}({\randforestsnr})=(({\bf f}_{\bf i},{\bf u}),{\bf f}_{{\bf i}+1 \mod k},\cdots, {\bf f}_{{\bf i}+(k-1) \mod k})$ is the forest obtained by shifting the indices of ${\randforestsnr}$, so that ${\bf f}_{\bf i}$ is now the first tree.

\begin{lemma}\label{lemma:forests_and_marks} $R_{\bf u}({\randforestsnr})$ and $\randforestsnr^{\bullet}$ have the same distribution.
It implies that
\begin{align}
	\sorted{\card{\randforestsnr}} \overset{(d)}= \sorted{\card{\randforestsnr^{\bullet}}}.
\end{align} 
\end{lemma}
\begin{proof}Only the first statement needs to be proved, since the tree sizes of $R_{\bf u}({\randforestsnr})$ and ${\randforestsnr}$ coincide (up to a rotation of the indices). For this, it suffices to observe that, for every $f^\bullet=(f_1^\bullet,f_2,\cdots,f_k)\in \markedforests{n}{k}$, 
\begin{equation}
	\P\left({R_{\bf u}({\randforestsnr})} = f^\bullet \right)= \frac{k}{n} \frac{1}{\card{\forests{n}{k}}}. \label{eqn:rqtaille}
\end{equation}
Since this does not depend on $f^\bullet$, this is the uniform distribution on $\markedforests{n}{k}$.

\end{proof} 

From Equation (\ref{eqn:rqtaille}), we deduce the following formula, which will be useful later:
\begin{equation}
\card{\markedforests{n}{k}} = \binom{n}{\frac{n-k}{2}}. \label{eqn:nbdeforetsmarquees}
\end{equation}

\paragraph{From forests to paths, from sizes of trees in a forest to excursion lengths above the current minimum in a path.}
It is well-known that forests are in bijection with some paths, and can be studied with them. 

For a path $w=(w_i, 0\leq i \leq M)$ starting at $w_0=0$, ending at time $M\geq 0$ (for some $M$), and with increments $w_i-w_{i-1}\in\{+1,-1\}$ for all $i$, denote by $\tau_k(w)=\inf\{n: w_n=k\}$ the hitting time of $k$ by $w$ (set $+\infty$, if $k$ is not reached). We simply write $\tau_k$ when $w$ is implicit.

For every $n\geq0, k\geq 1$, there is a bijection between $\forests{n}{k}$ and the set of paths $P(n,k)$ with $n$ steps such that $\tau_{-k}=n$.
This can be viewed as a simple consequence of the bijection that encodes a rooted complete binary tree by a \Lukasie walk: for a rooted complete binary tree $t$, we consider the order induced by a depth-first exploration of its nodes, and we let $n_i$ be the number of children of the $i$th node. Then we define the path $w$ such that for every $i$, $w_i =\sum_{j=1}^{i} (n_i - 1) $. This path $w$ is a Dyck path augmented by an additional $-1$ step. Notice that a tree reduced to its root is sent on a single down step. We can then extend the definition of the \Lukasie walk to a forest, by simply concatenating the paths associated to each tree of this forest. It is also a bijection, since we can uniquely decompose $w \in P(n, k)$ as the concatenation of $k$ augmented Dyck paths, the $i$th path being the restriction of $w$ between times $\tau_{-(i-1)}$ and $\tau_{-i}$. This bijection is illustrated on Subfigure \ref{subfig:luka_foret}.

The uniform distribution on $\mathcal{F}(n,k)$ is then sent by this bijection onto the uniform distribution on $P(n,k)$, and this latter can be viewed as the distribution of a simple random walk $(W_i,0\leq i\leq n)$ (with i.i.d.\ increments, and distribution $\P(W_i-W_{i-1}=+1)=p$ and $\P(W_i-W_{i-1}=-1)=1-p$, for some $p\in(0,1)$), conditioned\footnote{This remarkable property, that any $p$ does the job, is a key point, since we can in particular choose the ``best'' $p$, according to the context, in order to bias the unconditioned walk so that is has a behavior similar to the conditioned walk (see for example Equation (\ref{eqn:eq_pour_tau})in the proof of \ref{thm:phaseT_first_item}).} by $\tau_{-k} = \tau_{-k}(W)=n$. We denote by $\P_{p}$ the associated probability measure on infinite random walks.

The tree sizes, in this representation, correspond to the $(\tau_{-j}-\tau_{-(j-1)})$, and often, in the literature, they are viewed as ``excursion lengths above the current minimum process''. 
Namely, an excursion interval above the current minimum process of a function $f$ is defined as a maximal non empty interval $e = (\ell,r) $ such that $f(\ell) = f(r) = \min_{s\leq \ell} f(s)$ (the corresponding excursion is then the restriction of $f$ to this interval). The length of an excursion interval $e = (\ell, r)$ is $\card{e} \coloneqq r - \ell$. Excursions intervals are characterized by their left bound and their length, so we define the set of excursion intervals $\Excursions(f)$ as the set of these $(\ell, r-\ell)$. 
Then, for $W$ of length $n$ such that $\tau_{-k}(W)=n$, we have $\Excursions(W) = \{(\tau_{-j}, \tau_{-j} - \tau_{-j-1} - 1), 0 \leq j < k\}$.

We can finally define $\sorted{\card{\Excursions(W)}}$ as the sequence of sorted lengths of excursion intervals. 
We then have the following corollary of Corollary \ref{cor:blocks_and_forests}, Lemma \ref{lemma:forests_and_marks}, and of what we have just said above:
\begin{corollary} \label{cor:tailles_excursions_ch_sans_marquage} Let $\TL$ be the set of remaining holes under $\pcercle$, with $\nballs + \nholes = n$ and $\nl = \nholes - \nballs$, and let $W$ be a uniform path in $P(n,\nl)$. Then, the sorted block-size process have the same distribution as the sorted lengths of excursion intervals:
\begin{align}
	\sorted{{\Deltan \TL }} &\overset{d}{=} \sorted{\card{\Excursions(W)}}\\
	& \overset{(d)}= \sorted{\left((\tau_{-j-1 }- \tau_{-j} - 1)\right)_{0 \leq j < \nl}}.
\end{align}
\end{corollary}

\begin{figure}[t]\centering
\begin{subfigure}{\textwidth} \centering
	\includegraphics[scale=0.69,page=3]{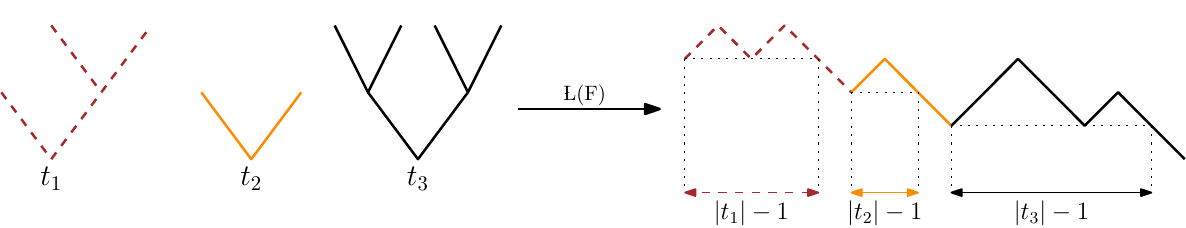}
	\caption{Illustration of the bijection mapping forests of $\forests{15}{3}$ to paths of $P(15,3)$ (such that $\tau_{-3} = 15$), using the \Lukasie walk. The $i$th excursion above the current minimum has size $|t_i|-1$.}
	\label{subfig:luka_foret}
\end{subfigure}
\hfill
\vspace{2mm}
\begin{subfigure}{\textwidth}\centering
	\includegraphics[scale=0.69,page=4]{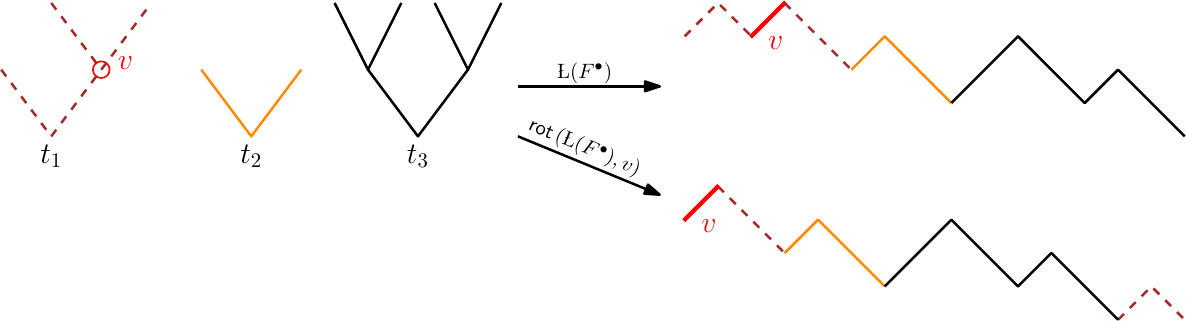}
	\caption{Illustration of the bijection mapping forests of $\markedforests{15}{3}$ to paths $W$ of length $15$ ending at $-3$. The first path is again obtained doing a \Lukasie walk, then the path is rotated so that the first step of the path is the one associated to the marked vertex. {Notice that the mark on the second path can be deleted, since it is necessarily on the first step of the path.} Another important remark is that the excursions of the second path do not directly correspond to the sizes of trees of $F$, since the excursion corresponding to the first tree has basically been cut into two pieces.} 
	\label{subfig:luka_foret_marquee}
\end{subfigure}
\hfill
\caption{Illustration of the bijections between forests and paths.}
\label{fig:illustration_lukasiewicz}
\end{figure}

\paragraph{Discussion on the paths associated to marked forests.} This paragraph is not required for the understanding of the proof of Theorem \ref{thm:phaseT}, except for the proof of \ref{thm:phaseT_third_item}. It also strengthens our study of excursions and allows to understand more deeply the connections made with the asymptotics of Chassaing-Louchard in Section \ref{subsect:connections_with_CL}.

We define the (discrete) rotation of a function as: 
\begin{equation}
\rotn(f,r) = t \mapsto  \begin{cases}
	f(t+r) - f(r)  & \si t \leq n - r, \\
	f(n) - f(r) + f(t-(n-r)) & \si t > n-r .
\end{cases} \label{eqn:def_de_rot}
\end{equation}
We will also use its continuous analogue (sometimes known as the Vervaat transform, see Section \ref{subsect:connections_with_CL} for a discussion on Vervaat's theorem):
\begin{equation}
\rot(f,r) = t \mapsto  \begin{cases}
	f(t+r) - f(r)  & \si t \leq 1 - r, \\
	f(1) - f(r) + f(t-(1-r)) & \si t > 1-r.
\end{cases} \label{eqn:def_de_rot_continue}
\end{equation}

We have proven in Equation (\ref{eqn:nbdeforetsmarquees}) that there are $\binom{n}{(n-k)/2}$ marked forests with $(n-k)/2$ trees and $n$ nodes. It is also the cardinality of $B(n,k)$, the set of paths of length $n$ ending at $-k$, i.e.\ bridges. There is a bijection between $\markedforests{n}{k}$ and $B(n,k)$: To a marked forest $F^\bullet = (F,v)\in \markedforests{n}{k}$ we can associate a path $\Lpol(F^\bullet)=\Lpol(F)$ of $P(n,k)$ (with the \Lukasie walk, as for the bijection between $\forests{n}{k}$ and $P(n,k)$). Then if $v$ corresponds to the $i$th vertex in the first tree (in the bread-first search order, which is the order taken in the \Lukasie walk), one can consider the rotation of $\Lpol(F^\bullet)$, $p \coloneqq \rotn(\Lpol(F^\bullet),i)$. This is a path ending at $-k$ at time $n$ (the condition $\tau_{-k} = n$ can be satisfied or not). 
Reciprocally, if $i$ is such that $n-i = \minargmin(p)$, then $\rotn(p,-i) \in P(n,k)$ (here and in what follows, if the minimum of a path $p$ is reached several times, then $\minargmin(p)$ is the smallest $k$ such that $p_k = \min(p)$). As $i$ encodes the marked vertex $v$, $\rotn(p,-i)$ encodes a \LL walk, hence a forest $F$ (since the \LL walk gives a bijection), thus $(F,v) \mapsto \rotn(\Lpol(F^\bullet),i)$ is a bijection. It is illustrated on Subfigure \ref{subfig:luka_foret_marquee}.

Thus, if $R^{(n,\nl)}$ is the \LL walk associated to a uniform marked forest $\randforestsmarked{n}{\nl}$ and $B^{(n,\nl)}\sim \mathcal{U}\left(B(n,\nl)\right)$ (in the sequel, we will often denote the uniform distribution over a set $S$ by $\mathcal{U}(S)$), then
\begin{align}
\rotn\left(B^{(n,\nl)},\minargmin(B^{(n,\nl)})\right) &\sim R^{(n,\nl)}\\
\intertext{and}
\rotn\left(R^{(n,\nl)},{\bf U}(\llbracket 0, \tau_{-1}(R^{(n,\nl)})\rrbracket)\right) &\sim B^{(n,\nl)} \label{eqn:marche_Luka_and_rotated_bridge_discreteVersion} 
\end{align}
(here, the variable ${\bf U}(\llbracket 0, \tau_{-1}(R^{(n,\nl)})\rrbracket)$ represents the uniform choice of a vertex in the first tree of the forest associated to $R^{(n,\nl)}$).

Together with Corollary \ref{cor:blocks_and_forests}, it implies:
\begin{corollary}\label{cor:tailles_excursions_ch_avec_marquage}We consider $\TL$ under $\pcercle$, with $\nballs + \nholes = n$, $\nl = \nholes - \nballs$. We let $b \sim \mathcal{U}\left(B(n,\nl)\right)$ and $p = \rotn (b, \minargmin(b))$, and then
\begin{align}
	\left({\Deltan_{i} \TL }\right)_{0 \leq i < \nl} &\overset{d}{=} \card{\Excursions( p)} \label{eqn:utile_pour_lien_entre_blocs_et_excursions_marquees} \\
	& \overset{(d)}= \left((\tau_{-j-1 }(p)- \tau_{-j}(p) - 1)\right)_{0 \leq j < \nl}.
\end{align}
\end{corollary}
It improves Corollary \ref{cor:tailles_excursions_ch_sans_marquage}, since here the equality in distribution holds not only for the sorted sequences, but for the whole sequences.

\addition{\begin{remark}\label{rem:tree_sizes} As illustrated on Figure \ref{subfig:luka_foret_marquee}, the tree sizes of a marked forest $F^\bullet = (F,v) \in \markedforests{n}{k}$ cannot be directly obtained from excursions of the associated bridge $b = \rotn(\Lpol(F), v)$, because the sub-path corresponding to the first tree may be cut into two pieces (at the beginning and the end of $b$). Yet, the size of the first tree is equal to $\tau_{-1}(b) + n - \tau_{-k}(b)$, while the size of the $i$th tree (for every $i \neq 1$) is still equal to $\tau_{-i-1 }(b)- \tau_{-i}(b) $. 
\end{remark}}

\paragraph{Maximum size of a block as a non-increasing function of $\nl$.}
The following corollary enables to compare (for the stochastic order) the maximum size of a block in the golf process $\Max{\Deltannl\TL}$, for different values of the parameters $n$ and $\nl$ (where we have adapted the notation, setting $\Deltannl\TL = {\Deltan\TL}$, in order to make explicit the dependence on $\nl$).

\begin{corollary} \label{cor:monotonie_taille_max_bloc}We again assume that $n = \nballs + \nholes$, and write $\nl = \nholes - \nballs$.
Then, for any fixed $n$, $\Max{\Deltannl\TL}$ is a non-increasing function of $\nl$ for the stochastic order. 

\end{corollary}

\addition{\begin{proof}
	Thanks to Corollary \ref{cor:blocks_and_forests}, it suffices to prove the same result for $\Max{\card{\randforestsmarked{n}{\nl}}}$, the maximum tree size in a random marked forest with $\nl$ (complete binary) trees and $n$ nodes. Recall that the set of marked forests $\markedforests{n}{\nl}$ is in bijection with the set of bridges $B(n,\nl)$. The main idea is then to couple uniform bridges of lengths $n$ ending in $-\nl$ with uniform bridges of lengths $n$ ending in $-\nl+2$, such that the maximum size of a tree in the corresponding forests can be compared.
	
	We consider the random function $\phi_{n,\nl} : B(n,\nl) \times \{1,\ldots, {(n+\nl)}/{2}\} \to B(n, \nl - 2)$ that changes some down step of a bridge ending in $-\nl$ to an up step, giving a bridge ending at $-\nl +2$. Formally, if $b\in B(n,\nl) $ and $i \in  \{1,\ldots, {(n+\nl)}/{2}\}$, and if $k_i$ is the index of the $i$th down step, then $\phi_{n,\nl}(b,i)$ is the bridge of $B(n, \nl - 2)$ obtained by replacing the down step at index $k_i$ by an up step. This bijection is illustrated on Figure \ref{fig:incr}.
	
	Finally, it is straightforward to check that if $k_i$ belongs to the $j$th tree, the forest corresponding to $\phi_{n,\nl}(b,i)$ is the forest where the $j$th tree (modulo $\nl -2$) has size equal to the sum of the sizes of the trees $j$, $j+1$ and $j+2$ (also modulo $\nl -2$) in the forest coded by $b$, while the other trees remain unchanged, and in particular have the same size. (It can be done formally using Remark \ref{rem:tree_sizes}). It implies that the maximum size of a tree in $b$ is not greater than the maximum size of a tree in $\phi_{n,\nl}(b,i)$.
	
	Moreover, if $b \sim \mathcal{U}(B(n,\nl))$ and $i \sim \mathcal{U}( \{1,\ldots, {(n+\nl)}/{2}\})$, then $$\phi_{n,\nl}(b,i) \sim \mathcal{U}(B(n,\nl-2)).$$ Combining this with the previous paragraph gives $\Max{\card{\randforestsmarked{n}{\nl}}} \leq \Max{\card{\randforestsmarked{n}{\nl-2}}}$ for the stochastic order.

	\begin{figure}[t]\centering
		\begin{subfigure}{0.45\textwidth} \centering
			\includegraphics[scale=0.6,page=4]{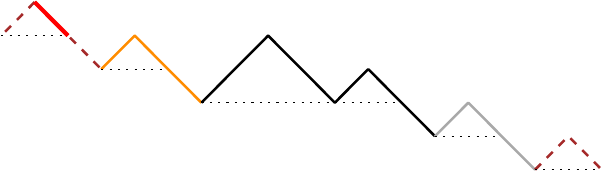}
			\caption{A bridge $b$ encoding a marked forest with 4 trees.}
			\label{subfig:incr1}
		\end{subfigure}
		\hfill
		\begin{subfigure}{0.47\textwidth}\centering
			\includegraphics[scale=0.6,page=5]{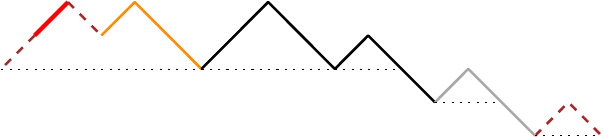}
			\caption{The bridge $\phi_{n,\nl}(b,2)$ encoding a marked forest with 2 trees, obtained by changing the second down step of $b$ (in red) into an up step. The first tree has the size of the first three trees of $b$, and the other trees are unchanged.} 
			\label{subfig:incr2}
		\end{subfigure}
		\hfill
		\caption{Illustration of the bijection $\phi_{n,\nl}$.}
		\label{fig:incr}
	\end{figure}
\end{proof}}

\bigremoval{
\begin{proof}
	Thanks to Corollary \ref{cor:blocks_and_forests} and Lemma \ref{lemma:forests_and_marks}, it suffices to prove the same result for $\Max{\card{\randforests{n}{\nl}}}$, the maximum tree size in a random forest with $\nl$ (complete binary) trees and $n$ nodes. 
	
	We actually prove it for the maximum tree size in a random unmarked forest with $p$ planar trees and $n$ edges. In fact, it is well-known that the set of planar trees with $x$ edges and the set of binary trees with $2x+3$ nodes are in bijection. We can thus define $\phi$ a bijection mapping the set of planar trees $\mathcal{T}^{\sf planar}$ to the set of binary trees $\mathcal{T}$ (with no size constraint on the trees), such that for every $t\in\mathcal{T}, \card{\phi(t)} = 2 \card{t} + 3$. We can then define, for any $n,p$, $\mathcal{F}^{\sf planar}(n,p)$ the set of forests with $p$ planar trees, and $\phi$ induces a bijection \begin{equation*}
		\fonction{\Phi}{\mathcal{F}^{\sf planar}(n,p)}{\mathcal{F}(2n+3p,p)}{(f_1,\ldots,f_p)}{(\phi(f_1),\ldots, \phi(f_p))}
	\end{equation*}
	which is such that, for any $f\in\mathcal{F}^{\sf planar}(n,p)$,  $\Max{\card{\Phi(f)}} = 2 \Max{\card{f}}  +3$. 
	
	It then suffices to prove that for any fixed $n$, $M(n,p)$, the maximum tree size in a uniform forest of $\mathcal{F}^{\sf planar}(n,p)$, is a non-increasing function of $p$ for the stochastic order.
	
	We use a bijection given by Bettinelli in \cite[Section 3.2, illustrated on Figure 4]{bettinelli}. 	
	It is a bijection between the set $\mathcal{F}^{\sf planar,ci}(n,p)$ of forests of $\mathcal{F}^{\sf planar}(n,p)$ with one distinguished corner and an integer in $\{1,\ldots,p+1\}$ and $\mathcal{F}^{\sf planar, vi}(n,p+1)$ the set forests of $\mathcal{F}^{\sf planar}(n,p+1)$ with a marked vertex and an integer in $\{1,\ldots,p\}$ (corners are a way to mark forests, they are formally defined in his paper, but we don't really need his definition here). A short study of this bijection enables to prove that, up to some re-rooting, it has split one tree in two, and kept the other trees unchanged. In particular, 
	if we consider $(f;c,i)$ a uniform element of $\mathcal{F}^{\sf planar,ci}(n,p)$, and $(f';c,i)$ the uniform element obtained by applying Bettinelli's bijection to $(f;c,i)$, then $f$ and $f'$ are uniform forests of $\mathcal{F}^{\sf planar}(n,p)$ and $\mathcal{F}^{\sf planar}(n,p+1)$, and $ \Max{\card{f}}  \geq \Max{\card{f'}} $. So, for the stochastic order, $M(n,p) \geq M(n,p+1)$.
\end{proof}
}

We now have all the combinatorial tools we need to prove Theorem \ref{thm:phaseT}. We end these preliminaries by stating the following proposition, which gives the convergence of a random walk conditioned by first hitting $-\lambda \sqrt{n}$ at time $n$, up to normalization, to a Brownian process $W^{(\lambda)}$ such that $\tau_{-\lambda}(W^{(\lambda)}) = 1$.

\begin{proposition}\label{prop:cv_marchecondit_vers_browniencondit} We consider $k=k(n)$ such that $k(n)/\sqrt{n}\to \lambda \in (0,+\infty)$ (and $n$ and $k(n)$ have the same parity). We let $(W_i,i\geq 0)$ be a simple random walk. In what follows, we view $W = (W_t, t\geq 0)$ as a continuous process, obtained by linear interpolation of the sequence $(W_k,k\geq 0)$. For every $n$, we let $ W^{(n, \lambda)} $ be a process distributed as $\left(\frac{W_{ nt}}{\sqrt{n}}, 0\leq t\leq 1\right)$ conditioned by $\tau_{-k(n)}(W)=n$. 
As $n$ goes to infinity, $  W^{(n, \lambda)}$ converges in distribution in $\Czo$ (equipped with the uniform topology), to $W^{(\lambda)}$ a Brownian motion {$B$} conditioned by $\tau_{-\lambda}(B) =1$, where $\tau_{-\lambda}(B)\coloneqq \inf\{t>0: B_t = -\lambda\}$.
\end{proposition}
The process $W^{(\lambda)}$ is called a first passage bridge from $0$ to $\lambda$. In Appendix \ref{annex:cv_marchecondit_vers_browniencondit}, we prove this proposition and we give a formal definition of a Brownian motion conditioned by $\tau_{-\lambda}(B) =1$. Chassaing and Janson \cite{Bertoin_Chaumont_Pitman_First_passage_bridges} give some similar results, for random walks with different distribution, in Theorems 4.1, 4.2 and 4.3.

\subsection{Critical window: proof of Theorem \ref{thm:phaseT}}\label{subsect:cercle_asymptotiques}
\begin{proof}[Proof of Theorem \ref{thm:phaseT}]

In all the proof, we let $k = k(n) = \nl(n)$. Notice that $n$ and $k(n)$ have the same parity (since this is always true for $n = \nballs + \nholes$ and $\nl = \nholes - \nballs$).

\textbf{\textbullet Proof of \ref{thm:phaseT_first_item}.}
We assume that $k(n) = an + o(\sqrt{n})$ for some constant $a>0$.
We will first show that
\begin{equation}
	\exists \alpha_a,\beta_a >0, \ \Prob{ \alpha_a \log n \leq \Max{\Deltan \TL} \leq  \beta_a \log n} \to 1.\label{eqn:undesobjectifs1}
\end{equation}    
Then, to conclude when $k(n) = an + o(n)$, we consider some $\eps >0$, and we can conclude using (\ref{eqn:undesobjectifs1}) for $k^-(n) \coloneqq a (1-\eps) n $ and $k^+(n) \coloneqq a (1+\eps) n $, and the monotonicity of $\Max{\Delta^{(n,k(n))}\TL}$ given in \ref{cor:monotonie_taille_max_bloc}:
\begin{equation}
	\Prob{ \alpha_{a(1+\eps)} \log n  \leq  \Max{\Deltan \TL} \leq  \beta_{a(1-\eps)} \log n } \to 1.
\end{equation}
\medskip
We consider ${\bf f}=({\bf f}_1,\cdots,{\bf f}_k)$ uniformly chosen in $\forests{n}{k(n)}$ (we previously denoted this random variable as $\randforests{n}{k(n)}$, but we simplify the notation in the proof). Lemma \ref{lemma:forests_and_marks} implies that $\Max{\bff}$ has the same distribution as ${\Max{\randforests{n}{k(n)}^{\bullet}}}$, which has the same distribution as $\Max{\Deltan \TL} +1$ (by Corollary \ref{cor:blocks_and_forests}). Thus, to prove (\ref{eqn:undesobjectifs1}) it suffices to prove that 
\begin{equation}
	\exists \alpha_a, \beta_a \in \R_+, \ \Prob{ \alpha_a \log n  \leq \smallcard{\Max{\bf f}} \leq  \beta_a \log n } \to 1. \label{eqn:trucici}
\end{equation} 
The remaining of this proof focuses on proving (\ref{eqn:trucici}).

\textbf{Computation of the upper bound:} We first focus on giving $\beta_a$ such that \linebreak $\P( \smallcard{\Max{\bff}}\geq \beta_a \log n)\underset{n}\to 0$. 
We have, for any $C >0$ and any $p_a \in (0,1)$ (we fix $C$ and $p_a$ later in the proof), given the correspondence forest-path discussed previously,

\begin{align}
	\P( \card{\bff_1} \geq C\log n)
	& = \sum_{m=\ceil{C\log n}, m \text{ is odd}}^{n} \P \left( \card{\bff_1} = m \right) \label{eq:f0}  \\
	& = \sum_{m=\ceil{C\log n}, m \text{ is odd}}^{n} \P_{p_a} \left(\tau_{-1} = m \middle| \tau_{-k(n)} = n \right) \label{eq:f1}  \\
	& = \sum_{m_0=\frac{\ceil{C\log n}-1}{2}}^{\frac{n-1}{2}} \frac{\P_{p_a}\left(\tau_{-1} = 2m_0+1 \right) \P_{p_a}\left(\tau_{k(n)-1} = n-(2m_0+1) \right)}{\P_{p_a} \left(\tau_{-k(n)} = n\right)} \label{eq:f2}  \\
	& \leq \sum_{m_0=\frac{\ceil{C\log n}-1}{2}}^{\frac{n-1}{2}} \frac{\P_{p_a}\left(\tau_{-1} = 2m_0+1 \right)}{\P_{p_a} \left(\tau_{-k(n)} = n\right)}.\label{eq:f3}
\end{align}

We know, thanks to the famous Rotation principle (sometimes also called Kemperman's formula, see e.g.\ \cite[equation (6.3) in Section 6.1]{pitman2006combinatorial}), that the denominator of (\ref{eq:f3}) is such that
\begin{align}
	\P_{p_a} \left(\tau_{-k(n)} = n\right) &= \frac{k(n)}{n} \P_{p_a} \left(W_{n} = -k(n)\right) \geq \frac{c_a}{2\sqrt{n}} \label{eqn:principe_de_rotation},
\end{align} 
for $n$ large enough and $c_a= \frac{1}{\sqrt{2 \pi p_a (1-p_a)}}$. Indeed, fix {$p_a = \frac{a+1}{2}$} (until the end of the proof of \ref{thm:phaseT_first_item}). Then, $\E_{p_a}\left[W_{n}\right] = (2p_a-1)n = -an = -k(n)+o(\sqrt{n})$ and we can use the {central local limit theorem} (see (\ref{eqn:TCLL3}) in Annex \ref{annex:TCLL} for more details) to prove that 
\begin{equation}
	\P_{p_a} \left(W_{n} = -k(n)\right) \underset{n\to\infty}\sim \frac{1}{\sqrt{2 \pi n p_a (1-p_a)}}
\end{equation} 
and thus that 
\begin{align}
	\P_{p_a} \left(\tau_{-k(n)} = n\right) \underset{n\to\infty}\sim  \frac{c_a}{\sqrt{n}}. \label{eqn:eq_pour_tau}
\end{align}

Moreover, we can compute, for every $m_0$, $\P_{p_a}\left(\tau_{-1} = 2m_0+1 \right) = C_{m_0} {p_a}^{m_0} {(1-p_a)}^{m_0+1} $ and use the fact that $ C_{m_0} \underset{m_0\to\infty}\sim \frac{4^{m_0}}{{m_0}^{3/2} \sqrt{\pi}}$ (obtained by Stirling formula), which implies that $C_{m_0} \leq 2\frac{4^{m_0}}{{m_0}^{3/2} \sqrt{\pi}}$ for ${m_0}$ large enough. So for $n$ large enough:
\begin{align}
	\P( \card{\bff_1} \geq C\log n) 
	& \leq \frac{4\sqrt{n}}{c_a}\sum_{m_0=\frac{\ceil{C\log n}-1}{2}}^{\frac{n-1}{2}} \frac{4^{m_0}}{{m_0}^{3/2} \sqrt{\pi}} {p_a}^{m_0} {(1-p_a)}^{{m_0}+1}\\
	& \leq \frac{4\sqrt{n}}{c_a} \frac{2^{3/2}(1-p_a) }{{(C \log n - 1)}^{3/2} \sqrt{\pi} } \sum_{m_0=\frac{\ceil{C\log n}-1}{2}}^{\frac{n-1}{2}} \left(4 p_a (1-p_a)\right)^{m_0}\\
	\intertext{and we conclude by noting that since $p_a = \frac{a+1}{2} \neq \frac12$, $ q_a \coloneqq 4 p_a (1-p_a) < 1$ hence this partial geometric sum can be easily bounded:}
	\P( \card{\bff_1} \geq C\log n) 
	& \leq \frac{4\sqrt{n}}{c_a} \frac{2^{3/2}(1-p_a)}{{(C \log n-1)}^{3/2} \sqrt{\pi} } \frac{{q_a}^{(C \log n-1)/2}}{1 - q_a}.
\end{align}
Finally, when $\beta_a >0$ is such that $(\beta_a/2) \log q_a < -1 - 1/2$ (for example with $\beta_a = -(3+\eps)/ ( \log q_a)$, for some $\eps >0$), noting that ${q_a}^{(\beta_a/2) \log n} = n^{(\beta_a/2) \log q_a }$, we get
\begin{align}
	\P( \card{\bff_1} \geq \beta_a\log n) 
	&= o \left(1/n\right) \label{eqn:encoreuneequnutile} .
\end{align}

Since the $f_i$ have same law, a union bound argument gives:
\begin{align}
	\P( \smallcard{\Max{\bff}} \geq \beta_a\log n) 
	&\leq k(n) \P( \card{\bff_1} \geq \beta_a\log n)
\end{align}
and this goes to 0 thanks to (\ref{eqn:encoreuneequnutile}).

\textbf{Computation of the lower bound:} To prove that $\P(\alpha_a \log n \leq \smallcard{ \Max{\bff}})\underset{n}\to 1$ for some $\alpha_a >0$, we will prove something much stronger:
\begin{equation}
	\exists \alpha_a >0, \	\Prob{\exists i : \card{{\bf f}_i} = 2\lfloor (\alpha_a/2) \log n \rfloor+1 } \to 1\label{eqn:plus_fort_que_la_lowerbound}.
\end{equation}

For simplicity of notation, in all what follows, we write $\bar{\alpha_a \log n}$ instead of \linebreak$2\lfloor (\alpha_a/2) \log n \rfloor+1 $, the closest odd integer, hence assuming that $\bar{\alpha_a \log n}$ is an odd integer in the computation. We do the same for $\bar{c\log n}$.

For every integer $i$ and every constant $c>0$, we introduce the corresponding counting variable $\aci \coloneqq \ind{\card{{\bf f}_i} =  \bar{c \log n}}$ and $\nbc = \sum_{i=1}^{k(n)} \aci$, and we let $\fnc \coloneqq \Prob{\acone = 1}$. (We will fix later the value of $c$.)

As before, it is straightforward that
\begin{align}
	\fnc &= \Prob{\card{{\bf f}_1} =\bar{c \log n} } = \frac{\P_{p_a}\left(\tau_{-1} = \bar{c\log n}\right) \P_{p_a}\left(\tau_{-k(n)+1} = n - \bar{c\log n}\right)}{\P_{p_a}\left(\tau_{-k(n)} = n\right)}. \label{eqn:trucla}
\end{align}

For any $g_n = o(n)$, $ \P_{p_a}\left(\tau_{-k(n)} = n + g_n\right) \sim \frac{c_a}{\sqrt{n}}$ (see (\ref{eqn:eq_pour_tau}) which is still valid in this case) and thus (\ref{eqn:trucla}) can be simplified (recalling that, for every $m$, $\P_{p_a}\left(\tau_{-1} = 2m+1 \right) = C_{m} {p_a}^{m} {(1-p_a)}^{m+1} $ and using Stirling formula):
\begin{align}
	\fnc &\underset{n\to\infty}\sim \P_{p_a}\left(\tau_{-1} = \bar{c \log n}\right)
	\underset{n\to\infty}\sim \frac{(1-p_a)\left(4p_a (1-p_a)\right)^{ (c/2)\log n }}{\sqrt{\pi}{ ((c/2)\log n )}^{3/2}}.
\end{align}

We can do a similar computation to prove that 
\begin{align}
	\Prob{\acone = \actwo = 1} \sim {(\fnc)}^2.
\end{align}

Thus, $\Cov{\acone, \actwo} = \Prob{\acone = \actwo = 1} -  \Prob{\acone = 1}^2 = o(({\fnc})^2)$, so that
\begin{align*}
	\Esp{\nbc} &= k(n) \Esp{\acone} \underset{n\to\infty}\sim a n \fnc\\
	\Var{\nbc} &= k(n) \Var{\acone} + k(n) (k(n)-1) \Cov{\acone, \actwo}\\
	&= k(n) \fnc + o(n^2 ({\fnc})^2) .
\end{align*}
\addition{Recall that $q_a = 4 p_a (1-p_a) <1$. We now replace $c$ by $\alpha_a = -(2-\eps)/\log q_a$, for some $\eps>0$. Then, ${q_a}^{c/2 \log n} =
	n^{-1+\eps/2}$ and $\fna \sim x n^{-1+\eps/2}(\log n )^{3/2}$ (where $x = (1-p_a)/(\sqrt{\pi} (c/2)^{3/2})$). It implies that $n \fna = xn^{\eps/2}/ (\log n )^{3/2} = o((n\fna)^2)$.} Then, $\Var{\nbca} = o( (n {\fna})^2)$. We can the use Chebyshev's inequality to get
\begin{equation}
	\Prob{\card{\nbca-\Esp{\nbca}} \geq an \fna/2} \leq \frac{\Var{\nbca}}{\left(an \fna/2\right)^2} \underset{n\to\infty}\to 0.
\end{equation}
Since $\fna \gg 1/n$, we get $\Esp{\nbca} \gg 1$ and immediately deduce that $\Prob{\nbca= 0} = o(1)$: Equation \eqref{eqn:plus_fort_que_la_lowerbound} holds.
\medskip

{\noindent \textbullet \bf Proof of \ref{thm:phaseT_second_item}. } We consider $k(n)$ such that $k(n)/\sqrt{n}\to \lambda \in [0,+\infty)$.

Thanks to Corollary \ref{cor:tailles_excursions_ch_sans_marquage}, it suffices to focus on the sizes of excursions of a uniform path of $P(n,k(n))$, and this latter has the same distribution as a simple random walk $W$ of length $n$ with i.i.d.\ increments $\P(W_i-W_{i-1}=\pm1)=1/2$, conditioned by $\tau_{-k(n)} = n$.

We use Proposition \ref{prop:cv_marchecondit_vers_browniencondit} to prove \ref{thm:phaseT_second_item}. The key idea is that $ W^{(n, \lambda)} \overset{(d)}\to W^{( \lambda)} $ in $\Czo)$, which implies the convergence of the sorted lengths of excursions of $ W^{(n, \lambda)} $ above the current minimum, i.e.\ of the block sizes (in distribution, for the product topology) to the lengths of excursions of $ W^{( \lambda)} $ above the current minimum. Indeed, Lemma 7 of \cite{Aldous1997} applied to $W^{(\lambda)}$ and $\left(W^{(n,\lambda)}\right)$ gives the desired result. It is stated for functions $[0,\infty) \to \R$, but can be adapted to functions $[0,1] \to \R$. It is then quite straightforward to check that the hypotheses are verified. We then have: $\Excursions \left(W^{(n,\lambda)}\right)$ converges to $\Excursions \left(W^{(\lambda)}\right)$ as $n$ goes to infinity in distribution (as in Aldous \cite{Aldous1997}, we regard $\Excursions \left(W^{(\lambda)}\right)$ and $\Excursions \left(W^{(n,\lambda)}\right)$ as point processes on $[0,1] \times (0,1]$, and convergence holds for the vague convergence of counting measures on $[0,1] \times (0,1]$).

Thus, for every $k$, the lengths of the sorted $k$ largest excursions (they correspond to the sorted $k$ highest projections of $\Excursions \left(W^{(n,\lambda)}\right)$ onto the second coordinate) converge to their continuous counterpart.
The result given in Theorem \ref{thm:phaseT}.\ref{thm:phaseT_second_item} is an immediate consequence of this.

\noindent{\textbullet \bf Proof of \ref{thm:phaseT_third_item}. } 
We let $\eps >0, \delta >0$. We are going to show that 
\begin{align}
	\pcercle \left( \frac{\Max{\Deltan\TL}}{n} \geq \eps\right) <\delta. \label{eqn:obj_cv_proba_iii}
\end{align}

First, observe that for any continuous map $f : [0,1] \to \R$ such that $f(0) = f(1) = 0$, and for any $\lambda >0$, the largest excursion of $f^\lambda$ defined by $f^\lambda(s) = f(s) - \lambda s$ for $ s \in \intzo$ is smaller than $\frac{\max f - \min f}{\lambda}$ (Figure \ref{fig:illustration_rotation_borne_excursion} illustrates this, and we use this property to obtain (\ref{eqn:truc}) below).

This therefore applies to $b^\lambda$, a Brownian bridge ending at $-\lambda$, using that $b^\lambda \overset{(d)}= (b^0_s - \lambda s, s\in \intzo)$, where $b^0$ is the standard Brownian bridge ending at 0 (see \cite[page 37]{revuzcontinuous}). In fact, there exist $m$ and $M$ such that $\Prob{\forall t, m \leq b^0_t \leq M} \geq 1-\delta$ (because the law of the range $[\min X, \max X]$ of a process $X$ taking its values in $\Czo$ is a law in $\R^2$, and then is tight). We consider $\lambda>0$ such that $\frac{M-m}{\lambda}\leq \eps$. Then, with probability at least $1-\delta$, $b^\lambda$ has excursions smaller than $\frac{M-m}{\lambda}$, hence smaller than $\eps$. 

We consider, as in the previous point, the bridges $X^{(n,\lambda)}$ defined as $\left(\frac{W_{nt}}{\sqrt{n}}, 0\leq t\leq 1\right)$ where $W$ is a simple random walk conditioned by $W_n = k(n) \sim \lambda \sqrt{n}$.
We know that $X^{(n,\lambda)} \overset{(d)}\to b^\lambda$ as $n$ goes to infinity, in distribution in $\Czo$ (this is well-known, and easy to prove by proof of convergence of the FDD and tightness).

As a consequence,
\begin{equation*}
	\card{\Excursions( \rot (X^{(n,\lambda)}, \argmin(X^{(n,\lambda)})))} \xrightarrow[n\to\infty]{{(d)}} \card{\Excursions( \rot (b^\lambda, \argmin(b^\lambda)))} 
\end{equation*}
for the vague convergence of counting measures on $\intzo \times (0,1]$. The reasoning is close to what is done in the proof of \ref{thm:phaseT_second_item}. It is possible to conclude with a continuity argument, using that $\argmin b^\lambda$ is unique almost surely, and Skorokhod representation theorem, to couple $X^{(n,\lambda)}$ and $b^\lambda$. Vervaat uses exactly this kind of reasoning in \cite{Vervaat1979}.

Using Corollary \ref{cor:tailles_excursions_ch_avec_marquage}, this implies the existence of $n_0$ such that, for every $n\geq n_0$,
\begin{align}
	\pcerclenllambda\left(\frac{\Max{\Deltan \TL}}{n} > 2\eps\right) &= 
	\P\left(\Max{\card{\Excursions( \rot (X^{(n,\lambda)}, \argmin(X^{(n,\lambda)})))} } > 2\eps\right) \\
	&\leq \Prob{\Max{\card{\Excursions( \rot (b^\lambda, \argmin(b^\lambda)))}} > \eps} \label{eqn:cetrucla} \\ 
	&\leq \Prob{\frac{\max b^0 - \min b^0}{\lambda} \geq \eps} \label{eqn:truc} \\ 
	& \leq \delta.
\end{align}

To conclude, since $\nl(n) / \sqrt{n} \to +\infty$, then for any $n$ large enough, $\nl(n) \geq \lambda \sqrt{n}$ and Corollary \ref{cor:monotonie_taille_max_bloc} and the previous discussion imply that 
\begin{align}
	\pcerclenl\left(\Max{\Deltan\TL}/n >2\eps\right) &\leq \pcerclenllambda\left(\Max{\Deltan\TL}/n >2\eps\right)\\
	&\leq \delta.
\end{align} 

We have the desired result: $\Max{\Deltan\TL}/n \to 0$ in probability, as $n$ goes to infinity.

\begin{figure}[ht]\centering
	\includegraphics[scale=0.55]{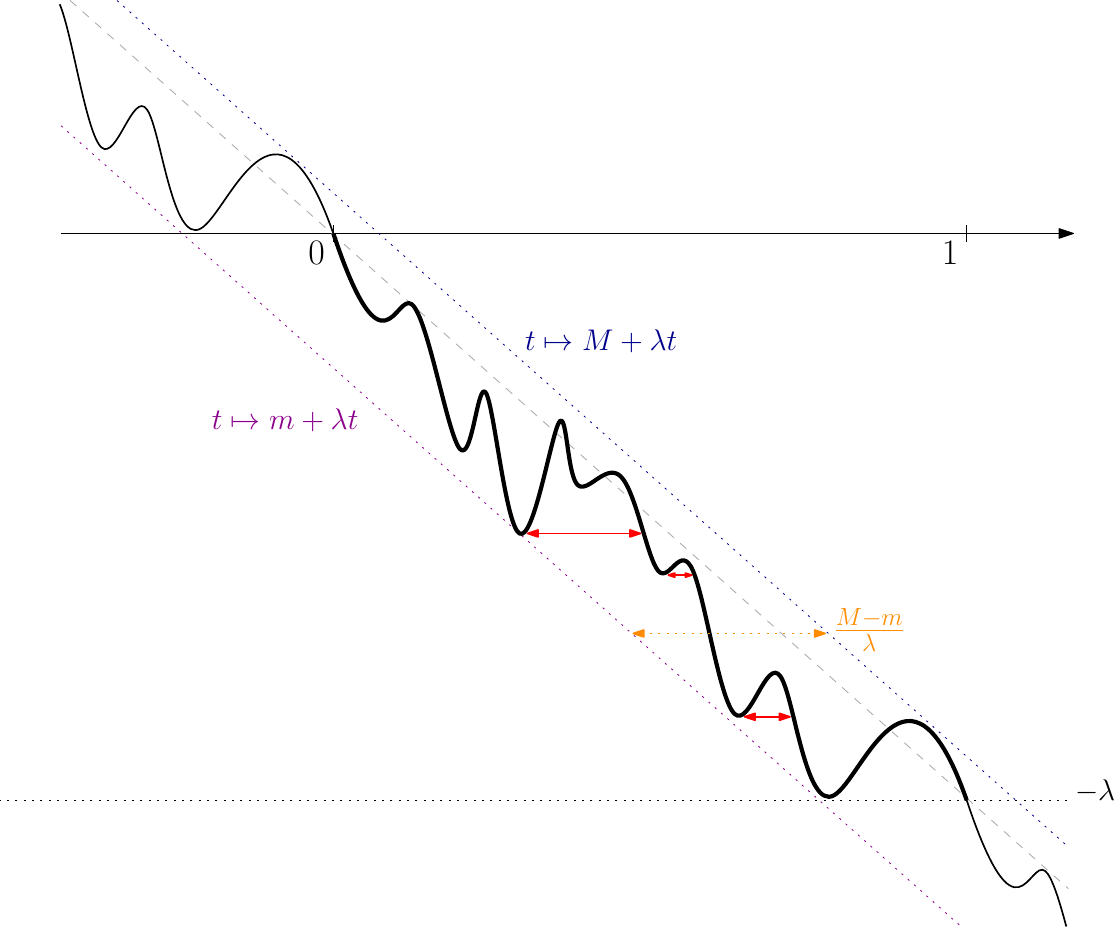}
	\hfill
	\caption{Illustration of the proof that sizes of excursions of a Brownian bridge ending at $-\lambda$ (and its rotation) go to 0 when $\lambda$ goes to infinity. The function in bold represent a continuous map $f^\lambda$ defined as $f^\lambda(s) = f(s) - \lambda s$, for every $s\in \intzo$ and for $\lambda >0$ ($f^\lambda$ can be seen as the Brownian bridge $b^\lambda = \left(b^0 - \lambda t, 0 \leq t \leq 1\right)$ ending at $-\lambda$, where $b^0$ is a standard Brownian bridge). It has been extended periodically, so that every rotation of $f^\lambda$ has the same graph up to some translation, and in particular the same excursion lengths. We assume that $\forall t, m \leq f(t) \leq M$. Then, $f^\lambda (t) =  f(t) - \lambda t$ is always above the line $t \mapsto m -  \lambda t$, and below the line $t \mapsto M - \lambda t$ (illustrated by the colored doted lines on the figure). The excursions lengths of $f^\lambda$ are then always smaller by $\frac{M-m}{\lambda}$ (on the drawing, some excursions are represented by red arrows, while the orange dotted arrow, of length $\frac{M-m}{\lambda}$, illustrates what is the maximum size of an excursion).}
	\label{fig:illustration_rotation_borne_excursion}
\end{figure}

\medskip
\noindent{\textbullet \bf Proof of \ref{thm:phaseT_fourth_item}. }

We prove that, when $\nl \sim \lambda \sqrt{n}$, $\Delta_0\TL /n$ the normalized size of the first block (marked, since it contains 0) converges in distribution to$\frac{N^2}{N^2 + \lambda^2 }$ (where $N$ is a standard Gaussian random variable with mean 0 and variance 1). We borrow arguments from Chassaing-Louchard \cite{chassaing2002phase}, and use them to conclude that when $\nl = o(\sqrt{n})$, $\Delta_0\TL /n \to 1$ (in probability, as $n$ goes to infinity). We now detail the main steps.

First, as a consequence of Corollary \ref{cor:blocks_and_forests}, for every $k$,
\begin{align}
	&	\pcercle \left( \Delta_0 \TL = 2k \right) 
	= \frac{(2k+1)C_k \ \card{\forests{n-(2k+1)}{\nl(n) - 1}}}{\card{\markedforests{n}{\nl(n) }}}\\
	&= { \frac{\sqrt{n}}{\sqrt{n-(2k+1)}} \frac{(2k+1)C_k 2^{-(2k+1)} \sqrt{n-(2k+1)} \card{\forests{n-(2k+1)}{\nl(n) - 1}} 2^{- n + (2k+1)} }{\sqrt{n}\card{\markedforests{n}{\nl(n) }}2^{-n}}}.
\end{align}
One can use this to compute, for any $x \in ([0,1])$ and $\lambda$, with $k = \lfloor nx/2 \rfloor$ and $\nl(n) \sim \lambda \sqrt{n}$:
\begin{align}
	\frac{n}{2}\pcercle \left( \Delta_0 \TL = 2k \right) &\sim n \frac{1}{\sqrt{1-x}} \frac{nx \frac{1}{2 \sqrt{\pi} (nx/2)^{3/2} } \ \frac{\lambda \sqrt{n}}{n(1-x)}  \frac{2 \exp\left(\frac{-\lambda^2}{2(1-x)}\right)}{\sqrt{2\pi}}}{\frac{2 \exp{\frac{-\lambda^2}{2}}}{\sqrt{2\pi}}}\\
	&\sim \frac{1}{\sqrt{2\pi}} \frac{\lambda}{\sqrt{x}(1-x)^{3/2}} \exp \left( \frac{-\lambda^2 x}{2(1-x)} \right) \eqqcolon f(\lambda, x) \label{eqn:densiteflambdax}
\end{align}
where we used that for any $n,r$, $ \card{\forests{n}{r}} = \frac{r}{n} \card{\markedforests{n}{r}}$ and $2^{-n}\card{\markedforests{n}{r}} $ is the probability that a simple random walk of length $n$
ends at height $-r$. We then used the central local limit theorem to compute the asymptotic terms.

The limit $f(\lambda, .)$ is in fact the density of $\frac{N^2}{\lambda^2 + N^2}$ (see the details in the proof of Theorem 2.1 of \cite{chassaing2002phase}). We then use a discrete version of Scheffé's lemma (see Corollary \ref{lemma:discreteScheffe} in Annex \ref{annex:scheffe}, in which we take $a_n = n/2$), to get
\[\frac{\Delta_0\TL}{n}  \overset{(d)}{\underset{n\to\infty}\longrightarrow} \frac{N^2}{\lambda^2 + N^2}.\]

We can then conclude. We let $\eps>0$ and $\delta>0$. For $\lambda>0$ large enough, 
\begin{equation*}
	\Prob{N^2/(N^2 + \lambda^2) < 1-\eps/2} \leq \delta.
\end{equation*}
Thanks to the convergence in distribution we proved above, there exists $n_0$ such that, for all $n\geq n_0$, if $\nl(n) \sim \lambda \sqrt{n}$, then
\begin{align}
	\pcerclenllambda\left(\Deltan_0\TL/n < 1-\eps\right) &\leq \Prob{N^2/(N^2 + \lambda^2) < 1-\eps/2} \leq \delta.
\end{align}
Finally, since $\nl(n) / \sqrt{n} \to 0$, for any $n$ large enough, $\nl(n) \leq \lambda \sqrt{n}$, and by Corollary \ref{cor:monotonie_taille_max_bloc}:
\begin{align}
	\pcerclenl\left(\Max{\Deltan\TL}/n < 1-\eps\right) &\leq \pcerclenllambda\left(\Max{\Deltan\TL}/n < 1-\eps\right)\\
	&\leq \pcerclenllambda\left(\Deltan_0\TL/n < 1-\eps\right)\\
	&\leq \delta.
\end{align} 

So finally, $\Max{\Deltan\TL}/n \to 1$ in probability, as $n$ goes to infinity.

\end{proof}

\subsection{Comparison between golf and parking asymptotic behavior}\label{subsect:connections_with_CL}

As we mentioned in Remark \ref{rem:remarque_parking}, Theorem \ref{thm:phaseT} and Corollary \ref{cor:phaseTparking} highlight very similar behaviors for the asymptotics of the golf process and the $p$-parking process, studied by Chassaing and Louchard for $p=1$ \cite{chassaing2002phase}. Let us shed more light on this connection.

In our study, the block sizes have the same distribution as sizes of trees in a uniform (marked) forest of unlabeled binary trees (see Corollary \ref{cor:blocks_and_forests} and Lemma \ref{lemma:forests_and_marks}), while in parking process the block sizes have the same distribution as the sizes of trees in forests of Cayley trees (i.e.\ trees with vertices labeled from $1$ to $n$), see \cite[Ex 6.4-31]{knuth1973art}. Hence, the combinatorial objects are different, but the reader familiar with this research domain could guess that the block sizes asymptotic behavior in both models are quite similar.

\addition{In addition, another small difference can be noted. The block-sizes in our process is biased by the size of the block containing 0, whereas parking blocks are also biased but in a subtle way: the parking process can be seen as a temporal process, and in \cite{chassaing2002phase} they consider the parking process rooted at the last empty place in the process. }

Actually, we can prove that the limit laws of the block-sizes are the same for both models up to some random rotation, using the following arguments.

We let $e = \left(e_t\right)_{t\in\intzo}$ be the standard Brownian excursion and $b = \left(b_t\right)_{t\in\intzo}$ be the standard Brownian bridge (ending at 0).
Then, for every $\lambda >0$, we set again $b^\lambda = \left(b_t-\lambda t \right)_{t\in\intzo}$, the Brownian bridge ending at $-\lambda$, and we define similarly $e^\lambda = \left(e_t-\lambda t \right)_{t\in\intzo}$, the Brownian excursion above the line $x \mapsto -\lambda x$.

Vervaat's theorem states that, if $a = \argmin_t b_t$, then $e \overset{(d)} = \left( b_{t+a \mod 1} - b_a \right)_{t\in\intzo}$
\cite{Vervaat1979}. Conversely, if $u \sim \mathcal{U}(\intzo)$, then 
$b \overset{(d)} = \left( e_{t+u \mod 1} - e_u \right)_{t\in\intzo}$ \cite{Biane1986}. We can restate this using the rotation function $\rot$ (defined in Equation (\ref{eqn:def_de_rot_continue})):
$b \overset{(d)}= \rot(e,u) \et e \overset{(d)}= \rot(b, \argmin_t b_t)$. 

This result can be generalized for excursions and bridges with a linear drift: $b^\lambda \overset{(d)}= \rot(e^\lambda,u)$
and 
$e^\lambda \overset{(d)}= \rot\left(b^\lambda,\argmin_t( b^\lambda_t + \lambda t)\right)$ (this is a direct consequence of Vervaat's theorem and of the fact that the rotation and the addition of a linear drift commute: $(\rot (e,u)(t) - \lambda t)_t = \rot(e^\lambda, u) $).

We now consider two other processes. First, let $W^\lambda$ be a Brownian motion conditioned by $\tau_{-\lambda} = 1$ (see Proposition \ref{prop:cv_marchecondit_vers_browniencondit} for its definition). Second, we define $R^\lambda$ as the rescaled limit of $R^{(n,\lambda\sqrt{n})}$ (the \LL walk associated to $\randforestsmarked{n}{\lambda\sqrt{n}}$, see Equation (\ref{eqn:marche_Luka_and_rotated_bridge_discreteVersion})). 

The following lemma implies that the processes $W^\lambda$, $e^\lambda$ and $R^\lambda$, up to some translation, have the same set of excursions, and therefore that the processes $\Deltan \TL$ in the golf model and in the $p$-parking model have asymptotically the same limit in law (up to a rotation).
\begin{lemma} \label{lemma:generalisation_Vervaat} For any process $\xi \in \{W^\lambda, R^\lambda\}$,
\begin{itemize}
	\item $e^\lambda \overset{(d)}= \rot\left(\xi,\argmin_t( \xi_t + \lambda t)\right)$
	\item $b^\lambda \overset{(d)}= \rot\left(\xi,u\right)$, where $u\sim\mathcal{U}(\intzo)$.
\end{itemize}
Moreover, 
\begin{equation}
	R^\lambda \overset{(d)}= \rot\left(b^\lambda, \argmin(b^\lambda)\right) \label{eqn:marche_Luka_and_rotated_bridge_limit}.
\end{equation}
Therefore, the processes $W^\lambda$, $R^\lambda$ and $e^\lambda$ have the same excursion lengths (above the current minimum process) up to a rotation. This is also true for the process $b^\lambda$ extended periodically.
\end{lemma}
Figure \ref{fig:illustration_rotation} illustrates this lemma.

The interested reader can find results with the same flavor in \cite{Bertoin_Chaumont_Pitman_First_passage_bridges} and \cite{chassaingjanson_Vervaatlike}. Nevertheless, we think that our lemma is not proved in the literature, so we give the sketch of the proof.

\begin{figure}[t]\centering
\includegraphics[scale=0.5]{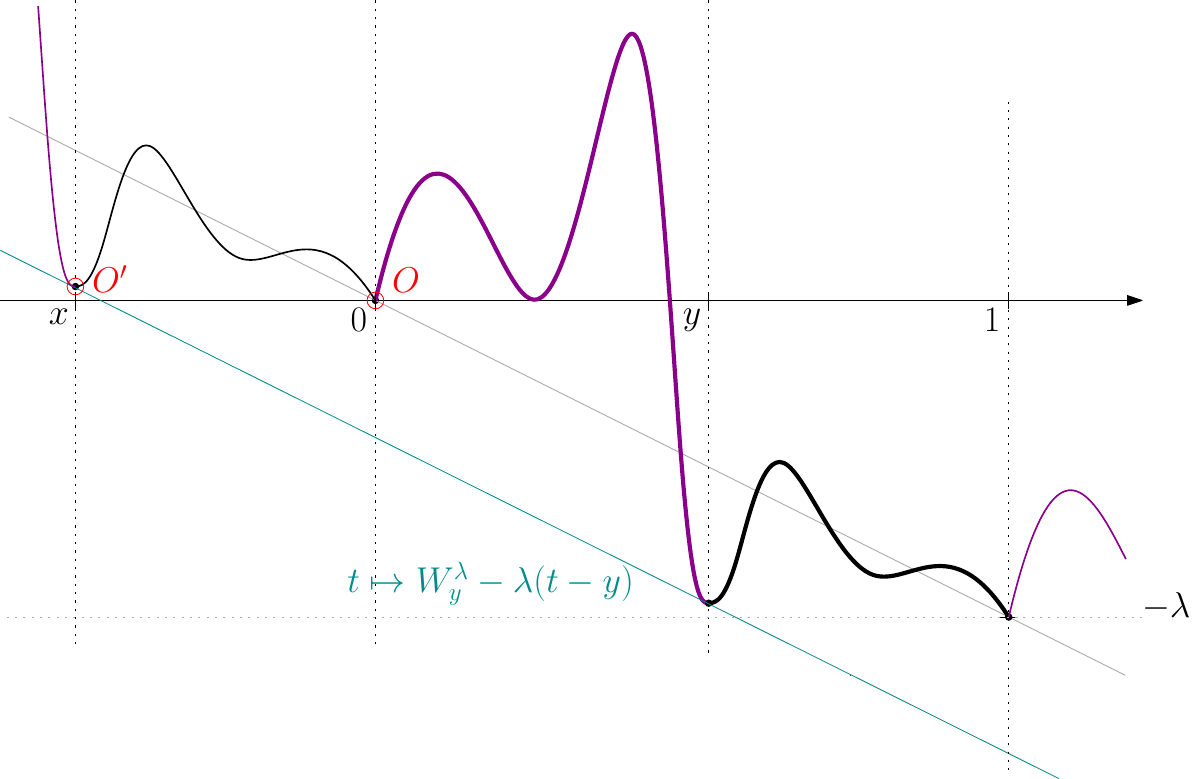}
\hfill
\caption{Illustration of Lemma \ref{lemma:generalisation_Vervaat}. The graph in bold, from the origin $\color{red} O$ to the point with coordinates $(1, -\lambda)$ represents $W^\lambda$, a Brownian process such that $\tau_{-\lambda}=1$. It has been extended periodically, so that the parts in black (or in purple) are exactly the same up to a translation. Then, $y$ is such that $y = \argmin_{t\in\intzo} W^\lambda_t + \lambda t$, and geometrically it means that the process $W^\lambda$ is always above the line of slope $-\lambda$ passing through $(y,W^\lambda_y)$ (drawn in blue on the figure). Therefore, the curve of $\rot\left(W^\lambda,\argmin_t( W^\lambda_t + \lambda t)\right)$ is the same as that of $W^\lambda$ with a change of origin (the graph from ${\color{red} O'}=(x,W^\lambda_x)$ (with $x = y - 1$) to $(y,W^\lambda_y)$). We can see on the figure that it is an excursion above the blue line. Lemma \ref{lemma:generalisation_Vervaat} states that it has the same distribution as $e^\lambda$.
}
\label{fig:illustration_rotation}
\end{figure}

\begin{proof}[Sketch of the proof]
We do not give a full proof here, but we give the main idea behind it, that follows from the proof of Vervaat's theorem \cite{Vervaat1979}. 
We focus on (\ref{eqn:marche_Luka_and_rotated_bridge_limit}) (the other statements have similar proofs). Equation (\ref{eqn:marche_Luka_and_rotated_bridge_discreteVersion}) is a discrete version of this. Moreover, if $\nl \sim \lambda \sqrt{n}$, $R^{(n,\nl)}$ converges to $R^\lambda$ (after a suitable rescaling), and similarly $B^{(n,\nl)}$ (the uniform bridge of length $n$ ending at $-\nl$, see (\ref{eqn:marche_Luka_and_rotated_bridge_discreteVersion})) converges to $b^\lambda$ (also after a suitable rescaling). One can conclude using a continuous mapping theorem (\cite[Theorem 2.7]{Billingsley1999}). The difficulty is that $\rot (., \argmin(.))$ is only continuous at functions $f$ that reach their minimum value only once. Vervaat explains how to deal with this difficulty, and the same reasoning would work here.
\end{proof}

\subsection{Sparse Case}\label{subsect:sparse_case}
\subsubsection{Proof of Theorem \ref{thm:blocs_cercle_nb_nt_fixes}}\label{subsect:asympt_nb_et_nt_fixes}

\begin{proof} We give some notation that enables us to deal properly with the fact that we have $\nl$ variables with a density function (with respect to the Lebesgue measure) on $\R^{\nl -1}$, corresponding to the restriction to $\nl-1$ variables (the last one is a deterministic function of the others, since the sum of the block sizes is equal to $n-\nl$). We define, for any $x=  (x_0,\ldots,x_{\nl-1})\in{\sf Simplex}(\nl)$, $\bar{x} = (x_0, \ldots, x_{\nl-2})$, and $\bar{{\sf Simplex}(\nl)} = \{\bar{x} : x \in {\sf Simplex}(\nl) \}$, the restriction of ${\sf Simplex}(\nl)$ to its first $\nl - 1$ coordinates. Notice that it is similar to consider $x\in{\sf Simplex}(\nl)$ or $x\in\bar{{\sf Simplex}(\nl)}$ with the additional value $x_{\nl-1} \coloneqq 1 - \sum_{i=0}^{\nl-2}x_i$ (which we thus always consider to be defined).

For any $x\in {\sf Simplex}(\nl) $ and for every $k$ such that $0 \leq k < \nl -1$, we set $\ell_k \coloneqq \lfloor x_k (n-\nl) \rfloor$, and we fix $\ell_{\nl - 1} \coloneqq n - \nl - \sum_{i=0}^{\nl - 2}\ell_i$. We will establish that, for any $(x_0,\ldots,x_{\nl-2})\in \bar{{\sf Simplex}(\nl)}$,
\begin{equation}
	f_n (x_0,\ldots,x_{\nl-2}) \coloneqq n^{\nl-1}\pcercle\left(\Delta_i \TL = \ell_i, \forall i \in \Z/\nl\Z \right) \underset{n\to\infty}{\to} f(x_0,\ldots,x_{\nl-2}) \label{eqn:Scheffe1}
\end{equation}
and that 
\begin{equation}
	\int_{\bar{{\sf Simplex}(\nl)}} f(x_0,\ldots,x_{\nl-2}) ~dx_0\ldots dx_{\nl-2} = 1. \label{eqn:Scheffe2}
\end{equation}

These two properties allow to conclude immediately: Corollary \ref{lemma:discreteScheffe} (which gives a discrete version of Scheffé's lemma) implies Theorem \ref{thm:blocs_cercle_nb_nt_fixes}.

It suffices to prove (\ref{eqn:Scheffe1}) and (\ref{eqn:Scheffe2}) to complete the proof of Theorem \ref{thm:blocs_cercle_nb_nt_fixes}.

\paragraph{Proof of (\ref{eqn:Scheffe1}).}
We let $c \coloneqq \card{\configinitcercle}$. We just need the following computation: for $x = (x_0,\ldots, x_{\nl-2}) \in \bar{{\sf Simplex}(\nl)}$,

\begin{align}
	f_n(x)&= \frac{n^{\nl-1}\left(\ell_0+1\right) }{c} \sum_{(b_i) : \substack{\sum_{i}b_i = \nballs \\\text{ et } \forall i, 2b_i \leq \ell_i}}\ \prod_{k \in \Z/\nl\Z} \ \frac{1}{b_k + 1}\binom{\ell_k}{b_k,b_k,\ell_k-2b_k}  \\
	&\underset{n \to\infty}{\sim}  \frac{n^{\nholes - \nballs}x_0 \nballs! \nholes ! }{n ^{\nballs + \nholes }} \sum_{(b_i) :\sum_{i}b_i = \nballs}\ \prod_{k \in \Z/\nl\Z} \ \frac{1}{b_k!(b_k + 1)!}\frac{\ell_k!}{(\ell_k-2b_k)!} \\
	&\underset{n \to\infty}{\sim}  n^{ - 2\nballs}x_0 \nballs! \nholes ! \sum_{(b_i) :\sum_{i}b_i = \nballs}\ \prod_{k \in \Z/\nl\Z} \ \frac{1}{b_k!(b_k + 1)!}\left(x_k n\right)^{2 b_k} \\
	&\underset{n \to\infty}{\to}  x_0 \nballs! \nholes ! \sum_{(b_i) :\sum_{i}b_i = \nballs}\ \prod_{k \in \Z/\nl\Z} \ \frac{1}{b_k!(b_k + 1)!}{x_k}^{2 b_k}. \label{eqn:lalimite}
\end{align}

\paragraph{Proof of (\ref{eqn:Scheffe2}).}
We now compute the integral below, using the formula given in Remark \ref{rem:loi_de_Dirichlet} and the fact that $2\nballs + \nl = \nballs +  \nholes$.
\begin{align}
	\int_{\bar{{\sf Simplex}(\nl)}} f(x) ~dx	&=\int_{\bar{{\sf Simplex}(\nl)}}  x_0 \nballs! \nholes ! \sum_{(b_i) :\sum_{i}b_i = \nballs}\ \prod_{k \in \Z/\nl\Z} \ \frac{1}{b_k!(b_k + 1)!}{x_k}^{2 b_k} ~dx \\
	&= \frac{\nballs! \nholes ! }{(\nholes + \nballs)!}\sum_{\sum_{i}b_i = \nballs}\ (2b_0+1) C_{b_0}\prod_{k \in \Z/\nl\Z, k\neq 0} C_{b_k}	\label{eqn:unjolicardinal}.
\end{align} 
We have seen (in Equation (\ref{eqn:nb_de_forets_avec_cette_taille})) that $(2b_0+1) C_{b_0}\prod_{k =1}^{\nl - 1} C_{b_k}$ is the number of forests of $\markedforests{ \nballs }{\nl}$ such that for every $i$ the $i$th tree has $b_i$ internal nodes. Hence, by (\ref{eqn:nbdeforetsmarquees}),
\begin{align}
	\int_{\bar{{\sf Simplex}(\nl)}} f(x) ~dx &= \frac{\nballs! \nholes ! }{(\nholes + \nballs)!} \card{\markedforests{\nballs}{\nl}}
	= 1.
\end{align}

\end{proof}

\subsubsection{Discussion on the continuous model that appears as the limit of the golf process in the sparse case}\label{subsect:remarque_limit_sparse_case}

In Theorem \ref{thm:blocs_cercle_nb_nt_fixes}, the normalized block sizes converge in distribution to a mixed Dirichlet distribution. Actually, it can be seen as the distribution of the final configuration of several continuous analogues of the golf process defined on the unit circle $C = \R\backslash\Z$. We consider a finite number of balls $\nballs$ and holes $\nholes$ (that are random real points) chosen independently and uniformly on $C$, the balls being equipped with random clocks as usual (call $\mu_\infty (\nballs, \nholes)$ this initial distribution). We consider the two following variants:
\begin{enumerate}[label=(\subscript{V}{{\arabic*}})]
\item \label{limit_sparse_case_variant1}upon activation, a ball fills the first hole to its right;
\item \label{limit_sparse_case_variant2}upon activation, a ball does a Brownian motion stopped at the first hitting time of a free hole, which it fills. 
\end{enumerate}

To show that the distribution of the block-size process in both variants (i.e. the distances between consecutive free holes in the final configuration) is given by $f$ (see Equation (\ref{eqn:densite_sparse_case})), a coupling argument suffices: first, concerning the initial configuration, the unit circle appears as a limit when normalizing $\znz$ by $n$, and by this scaling, the initial configuration converges to $\mu_\infty (\nballs, \nholes)$. 
Then there are two cases:
\begin{itemize}
\item When $p>1/2$ (the case $p<1/2$ can be dealt with similarly), the discrete trajectory of a generic ball is a random walk $(S_k)_{k\geq0}$ on $\znz$ of parameter $p$ starting at $S_0 = \lfloor n u \rfloor$. It is biased and has the following asymptotic behavior (it is driven by its mean): 
\begin{align}
	\left(\frac{S_{nt}}{n}, t\geq 0\right) \overset{(d)}\to \left( u+ (2p-1)t \mod 1, t\geq 0\right)
\end{align}
uniformly on each compact (this can be seen as a consequence of Donsker's theorem, which allows to control the error term).
In other words, in the continuous limit, the balls almost surely go to the right at constant speed (which means that in the continuous golf model the balls jump instantaneously to the first hole on the right). This corresponds to variant \ref{limit_sparse_case_variant1}. 

\item When $p=1/2$ (corresponding to variant \ref{limit_sparse_case_variant2}), the discrete trajectory of a generic ball is, similarly, a random walk $(S_k)_{k\geq0}$ on $\znz$ of parameter $p=1/2$ starting in $S_0 = \lfloor n u \rfloor$. This random walk converges to a Brownian motion:
\begin{align}
	\left(\frac{S_{n^2 t}}{n}, t\geq 0\right) \overset{(d)}\to \left( (u + B_t) \mod 1, t\geq 0\right)
\end{align}
in $C([0,M], C)$ the set of continuous functions from $[0,M]$ to $C$, equipped with the topology of uniform convergence, for any $M>0$.
Notice that we chose $n$ as space normalization because we want that, in the continuous limit, the Brownian travels across a macroscopic part of the circle, and this imposes the time normalization in $n^2$. 
\end{itemize}

This informal justification can be made rigorous by the Skorokhod representation theorem.

\section{Golf model on $\Z$} \label{sect:z} 

\subsection{The model on $\Z$ is valid - proof of Theorem \ref{thm:defZ} }\label{subsect:defZ}

\begin{proof}
We first prove the theorem when $\db < \dt$.

\paragraph{Case $\db < \dt$:}The key idea in this case is to prove that almost surely, for a given initial configuration $\eta^0$, there exists an infinite set of holes $\Sep$ such that $\Sep\subseteq \TL$, and $\Sep$ is independent of the balls trajectories. We call elements of $\Sep$ \textit{separators}, and show that they allow to {divide $\Z$ into finite disjoint intervals on which it is easy to define the golf process.}\\

We consider the initial configuration $\eta^0$ taken at random as in Section \ref{subsect:intro_Z}.

First, we say that a vertex in $n\in \Z$ is a \textit{separator} in $\eta^0$ if the following event $A_n$ is true : 
\begin{align*}
	A_n \coloneqq \left\{\forall k < n, \sum_{j = k}^{n-1} \eta^0_j\leq 0 \right\} \cap \left\{\eta^0_n = -1\right\} \cap \left\{\forall k > n, \sum_{j = n+1}^k \eta^0_j \leq 0 \right\}
\end{align*}
(it is a hole, and there is a surplus of holes on any interval of the form $\llbracket a, n-1\rrbracket$ or $\llbracket n+1, b\rrbracket$).
We let $\Sep = \{k : A_k\}$ be the set of separators.

\begin{lemma}\label{lemma:infinite_nb_separators}
	There is almost surely an infinite number of separators on the right and on the left of 0.
\end{lemma}

Before proving this lemma, we define a coding between initial configurations and bi-infinite paths with steps 0, -1 and +1, illustrated on Figure \ref{fig:defZ_modele_aux}.

To any $\eta^0\in\{-1, 0, +1\}^\Z$, we associate a bi-infinite path $S = S(\eta^0)= (S_{k-1/2})_{k\in\Z}$ such that $S_{-1/2} = 0$ and $\forall k \in \Z, S_{k+1/2} = S_{k-1/2} + \eta^0_k$.
We consider indices in the set of half-integers $\Z+1/2$, so that a vertex of the initial configuration corresponds to an edge in the associated bi-infinite path. For example, if some vertex $n$ contains a ball, $\eta_n = 1$, thus $S_{n-1/2} = S_{n+1/2} +1$, so the corresponding edge is a {``+1 edge''}.

We also define the \textit{height} of a vertex $n$ as the height of (the middle of) its corresponding edge in the bi-infinite path $(S_{k-1/2})_{k\in\Z}$: $\height(n) \coloneqq \frac{S_{n-1/2} + S_{n+1/2}}{2}$. Observe that when the golf process is valid at time $t$, we can extend the definition of the corresponding bi-infinite path and of the height function, so as to define $S^t$ and $h^t$.

Then, $A_n$ is true if and only if $n$ is a hole, and there exists no $k \neq n$ such that $\{S_{k-1/2}, S_{k+1/2}\}=\{S_{n-1/2}, S_{n+1/2}\}$, and equivalently, if there is no $k \neq n$ such that $h_k = h_n$. As an example, the vertex $n$ on Figure \ref{fig:defZ_modele_aux} is a separator (if we assume that the bi-infinite path never comes back at the height of $n$). Figure \ref{fig:defZ_modele_aux} illustrates the reason why we focus on these separators. 
\\

\begin{figure}\centering
	\begin{tikzpicture}[inner sep=0.7mm, scale = 0.5, line width=0.7]
		\draw[-,line width=1.5pt] (-13,1) -- (12,1);
		\foreach \x/\y/\z in {-4/-3/-4,-7/-6/-3,-8/-5/-4,-10/-9/-4}{
			\draw[green,->,shorten <= .2 + \pgflinewidth,shorten >=.12cm + \pgflinewidth] (\x,1) to[bend right] (\y,1);
			\draw[green,->,shorten <= .2 + \pgflinewidth,shorten >=.2cm + \pgflinewidth] (\x,\z+0.5) to (\y,\z+0.5);
		}
		\draw[green,->,shorten <= .2 + \pgflinewidth,shorten >=.21cm + \pgflinewidth] (-11,1) to[bend right] (-1,1);
		\draw[green,->,shorten <= .2 + \pgflinewidth,shorten >=.2cm + \pgflinewidth] (-11,-4.5) to (-1,-4.5);
		\foreach \x/\y/\z in {2/1/-6,6/5/-7,8/7/-7,9/4/-6}{
			\draw[cyan,->,shorten <= .2 + \pgflinewidth,shorten >=.12cm + \pgflinewidth] (\x,1) to[bend right] (\y,1);
			\draw[cyan,->,shorten <= .2 + \pgflinewidth,shorten >=.2cm + \pgflinewidth] (\x,\z-0.5) to (\y,\z-0.5);
		}
		\foreach \x/\y in {-11/0,-10/1,-8/1, -7/2, -4/1, 2/-2, 6/-3, 8/-3, 9/-2}{
			\draw[-,black] (\x-1/2,\y-5) -- (\x+1/2,\y-4);
			\draw[black, fill] (\x,1) circle (0.2);
		}
		\foreach \x/\y in {-12/1,-9/2,-6/3,-5/2,-3/2,-1/1,0/0,1/-1,4/-1,5/-2,7/-2,10/-1,11/-2}{
			\draw[-,black] (\x-1/2,\y-5) -- (\x+1/2,\y-6);
			\draw[black, fill=white] (\x,1) circle (0.2);
		}
		\foreach \x/\y in {3/-1,-2/1}{
			\draw[-,black] (\x-1/2,\y-5) -- (\x+1/2,\y-5);
			\draw[black, fill=black] (\x,1) circle (0.08);
		}
		\draw[red] (-1,1) circle (0.4);
		\draw[orange!80!yellow] (0,1) circle (0.4) node[xshift=10, yshift = 10] {$n$};
		\draw[orange!80!yellow, line width = 2.2pt] (-1/2,-5) -- (1/2,-6);
		\draw[red, line width = 2.2pt] (-1/2,-5) -- (-3/2,-4);
		\draw[dotted,orange!80!yellow] (-13,-5) -- (12,-5);
		\draw[dotted,orange!80!yellow] (-13,-6) -- (12,-6); 
		\node[left] at (-13,-5) {$S_n$};
		\node[left] at (-13,-6) {$S_{n+1}$};
		\draw[black, dotted] (11.5, -8) -- (12.5,-9);
		\draw[black, dotted] (-13.5, -3) -- (-12.5,-4);
	\end{tikzpicture}
	\caption{Illustration of the correspondence between initial configurations and bi-infinite paths with steps 0, -1 and +1. A ball ($\bullet $) corresponds to a step $+1$, a hole ($\circ$) to a step -1, and a neutral vertex to a step 0.\\
		If a vertex is a separator (for example, the vertex $n$ in orange) the corresponding edge is such that no other edge is at its height, and we will show that the corresponding hole is in $\TL$, whatever the ball trajectories {(for a first intuition, the green and blue arrows pointing towards $n$ illustrate a kind of worst case, in which all the ball trajectories are oriented towards $n$)}. The vertex $n-1$, circled in red, is not a separator: there are other edges at the height of the red one, and balls can reach the corresponding hole.
	}
	\label{fig:defZ_modele_aux}
\end{figure}

\begin{proof}[Proof of Lemma \ref{lemma:infinite_nb_separators}]
	
	We denote by $\mudbdt$ the product measure on $\mathcal{S}$ which is the law of $\eta^0 = (\eta^0_k)_{k\in\Z}$ (see (\ref{eqn:mudbdt})).
	The measure $\mudbdt$ is invariant under translation and ergodic. In particular, every event measurable with respect to $\sigma(\eta^0)$ and invariant under translation occurs with probability 0 or 1.
	
	We now prove that almost surely, $\inf \{n: A_n\} = -\infty$ and $\sup \{n: A_n \} = +\infty$.
	By symmetry, it suffices to study $\inf \{n: A_n\}$. Since $\inf \{n: A_n\}$ is a random variable taking value in $\Z \cup \{-\infty, +\infty\}$ and invariant under translation, almost surely $\inf \{n: A_n\} \in \{-\infty, +\infty\}$. 
	Moreover, the events $\inf \{n: A_n\} = -\infty$ and $\inf \{n: A_n\} = +\infty$ are invariant under translation, so by ergodicity they have probability 0 or 1. 
	We can compute the probability of $A_0$:
	\begin{align*}
		\mudbdt(A_0) &=\mudbdt\left(\eta^0_0 = -1\right) \times \mudbdt\left(\forall k \leq 0, \sum_{j = k}^{0}\eta^0_j \leq 0\right) \times \mudbdt\left(\forall k \geq 1, \sum_{j = 1}^k \eta^0_j \leq 0 \right) \\
		&= (\dt) \times (p_0)^2
	\end{align*} 
	(by independence of the $\eta_k$, and with $p_0 = \mudbdt\left(\forall k \geq 1, \sum_{j = 1}^k X_j \leq 0 \right) >0$). The number $p_0$ corresponds to the probability that a random walk with steps $(\eta^0_k)_{k\geq1}$ starting in $S_{1/2} = -1$ does not hit 0 after time 1/2, and this probability is positive, since this random walk is transient (because $\dt > \db$). It can actually be computed (we do it in the proof of Claim \ref{claim:c'est_bien_une_proba}), but we do not need it here, since knowing its positivity suffices.
	
	Finally, we know that $	\mudbdt\left(\inf \{n: A_n\} = -\infty\right) \geq 	\mudbdt(A_0) > 0$, thus $\mudbdt(\inf \{n: A_n\} = -\infty)= 1$. Similarly, $\mudbdt(\sup \{n: A_n\} = +\infty)= 1$.
	
\end{proof}

Until the end of the proof of the case $\db < \dt$, we fix the initial configuration $(\eta^0_k)$ and assume that $\inf \Sep = - \infty$ and $\sup \Sep  = +\infty$ (which is almost sure, thanks to Lemma \ref{lemma:infinite_nb_separators}).

For every $s\in \Sep$, we define $\suiv(s) = \inf\{s' \in \Sep : s' > s\}$ (it is well defined), and we then let $I_s \coloneqq \llbracket s, \suiv(s)\rrbracket$. For every separator $s\in \Sep$, we say that $I_s$ is an \textit{enclosed interval}. Enclosed intervals are such that, for every $x \in I_s$ such that $x\neq s$ and $x\neq \suiv(s)$, $\height(s) > \height(x) > \height(\suiv(s))$.

For every interval $I$, we define the golf process \textit{restricted to $I$} as the golf process on $I$, with initial configuration $(\eta^0_i)_{i\in I}$, clocks $(C(i))_{i\in I}$ and transition matrix $\PMC^p|_I$, which corresponds to $\PMC^p$ on $I$, but reflected on the boundaries (the choice of the simple reflection is arbitrary, since we will see that the balls actually never hit these boundaries). We denote it by $\eta[I]\coloneqq (\eta^t[I], t\in\intzo) $. Note that in general, this is not valid, as there can be more balls than holes in this initial configuration.

We now give the following lemma, that shows that if we consider the process restricted to an enclosed interval, no ball exits this interval.
\begin{lemma}\label{lemma:no_ball_exits_enclosed_intervals}
	Let $I = \llbracket s_1,s_2\rrbracket$ be an enclosed interval (that is, $s_1 \in \Sep$ and $s_2 = \suiv(s_1)$). The golf process $\eta[I]$ is valid, and at time 1, $s_1$ and $s_2$ are free holes, that is $\eta[I]^1_{s_1} = -1$ and $\eta[I]^1_{s_2} = -1$.
	
\end{lemma}

\begin{proof}
	
	The process $\eta[I]$ is valid, since an enclosed interval contains more holes than balls (since $\sum_{k = s_1+1}^{s_2-1} \eta^0_k \leq 0$) and this number of balls is finite (by Proposition \ref{prop:bonne_def_graphe_fini_ET_commutation})
	
	For the second statement, we can consider the associated path and height (restricted to $I$) $(S^t_{k-1/2})_{s_1 \leq k \leq s_2+1 }$ and $(\height^t(s))_{s_1 \leq s \leq s_2} $, defined for every $t\in\intzo$. We now show that every ball in $I$ finds a hole inside $I$, never hitting $s_1$ nor $s_2$.
	
	More precisely, we show the three following properties, for every $t$:
	\begin{enumerate}[label=(\subscript{P}{{\arabic*}})]
		\item \label{P1} if a ball starting at some vertex $b\in I$ is able to reach a free hole at vertex $h_1\in I$ at time $t$, with $h_1 < b$ (respectively $h_2\in I$ such that $h_2 > b$) 
		then $\height^{t^-}(h_1) \leq \height^{t^-}(b)$ (respectively $\height^{t^-}(h_2) \geq \height^{t^-}(b)$),
		\item \label{P2} for every $t' < t$, for every $x \in I$, $\height^{t'}(s_1) > \height^{t'}(x) > \height^{t'}(s_2)$,
		\item \label{P3} at time $t$, $s_1$ and $s_2$ are free holes.
		
	\end{enumerate}
	
	Since $s_1$ and $s_2$ are free at time 0, then for every $t$, \ref{P1} and \ref{P2} imply \ref{P3}, and \ref{P3} gives the desired result. So we only need to prove \ref{P1} and \ref{P2}, and to prove it for $t \in \Clock(\Ballsetdet\cap I)$ (the set of clocks of balls that belong to $I$). The proof is illustrated on Figure \ref{fig:defZ_fig2}.
	
	\begin{figure}
		\centering
		\begin{subfigure}[b]{0.92\textwidth}
			\centering
			\begin{tikzpicture}[inner sep=0.7mm, scale = 0.45]
				\clip(-6,-2.5) rectangle (15,4.5);
				\draw[-,line width=1pt, color = orange] (-5,4) -- (-4,3) node [midway, below left] (a) {$s_1$};
				\draw[-,line width=1pt, color=lightgray] (-4,3) -- (-2,1) -- (-1,1) -- (0,0);
				\draw[-,line width=1pt, color = red] (0,0) --++ (1,-1) node [midway] (hole1) {};
				\node [red,below left] at (hole1) {$h_1$};
				\draw[-,line width=1pt] (1,-1) --++ (1,1) --++ (1,0) --++ (1,1);
				\draw[-,line width=1pt, color = red] (4,1) --++ (1,1) node [midway] (ball) {};
				\node [red,below right] at (ball) {$b$};
				\draw[-,line width=1pt] (5,2) --++ (1,0) --++ (1,1) --++ (1,0);
				\draw[-,line width=1pt, color = red]  (8,3) --++ (1,-1) node [midway] (hole2) {}; 
				\node [red,below left] at (hole2) {$h_2$};
				\draw[-,line width=1pt, color=lightgray] (9,2) -- (10,1) -- (11,1) -- (13,-1);
				\draw[-,line width=1pt, color = orange] (13,-1) -- (14,-2) node [midway, below left] (b) {$s_2$};
				
				\draw[->, color = red] (ball) to[bend right] (hole1);
			\end{tikzpicture}
			\caption{$S^{t^-}$: a ball in \redcolor{$b$} is able to reach the holes \redcolor{$h_1$} and \redcolor{$h_2$}, and fills \redcolor{$h_1$}. It can't reach \orangecolor{$s_1$} and \orangecolor{$s_2$}}
			\label{fig:sep_stays_sep1}
		\end{subfigure}
		
		\begin{subfigure}[b]{0.92\textwidth}
			\centering
			\begin{tikzpicture}[inner sep=0.7mm, scale = 0.45]
				\clip(-6,-2.5) rectangle (15,4.5);
				\draw[-,line width=1pt, color = orange] (-5,4) -- (-4,3) node [midway, below left] (a) {$s_1$};
				\draw[-,line width=1pt, color=lightgray] (-4,3) -- (-2,1) -- (-1,1) -- (0,0);
				\draw[-,line width=1pt, color = red] (0,0) --++ (1,0) node [midway] (hole) {};
				\draw[-,line width=1pt] (1,0) --++ (1,1) --++ (1,0) --++ (1,1);
				\draw[-,line width=1pt, color = red] (4,2) --++ (1,0) node [midway] (ball) {};
				\draw[-,line width=0.5pt, color = red,dotted] (0,0) --++ (1,-1);%
				\draw[-,line width=0.5pt,dotted] (1,-1) --++ (1,1) --++ (1,0) --++ (1,1);%
				\draw[-,line width=0.5pt, color = red,dotted] (4,1) --++ (1,1) ;%
				\draw[-,line width=1pt, color=lightgray] (5,2) --++ (1,0) --++ (1,1) --++ (1,0)--++ (1,-1);
				\draw[-,line width=1pt, color=lightgray] (9,2) -- (10,1) -- (11,1) -- (13,-1);
				\draw[-,line width=1pt, color = orange] (13,-1) -- (14,-2) node [midway, below left] (b) {$s_2$};
				
			\end{tikzpicture}
			\caption{$S^{t}$: only the height of edges corresponding to vertices between \redcolor{$h_1$} and \redcolor{$b$} (included) have changed, and are still smaller then $\height(\orangecolor{s_1})$ and greater than $\height(\orangecolor{s_2})$. The previous height function appears in dots}
			\label{fig:sep_stays_sep2}
		\end{subfigure}	
		\caption{We recall that a downward edge corresponds to a free hole, an upward edge to a ball, and a horizontal edge to a neutral vertex (which can be an occupied hole or the vertex of a ball that has already been activated). This figure illustrates the proof that restricting to a closed interval, no ball can reach a separator. The steps $s_1$ and $s_2$ (in orange) correspond to two separators. 
		}
		\label{fig:defZ_fig2}
	\end{figure}
	
	Let $t \in \Clock(\Ballsetdet\cap I)$. First, \ref{P1} is a direct consequence of the fact that for each $t$, $S^t$ is non-decreasing on intervals containing no holes. 
	Let $h_1$ (resp. $h_2$) be the first hole on the right (resp. on the left) of $b$ (they exist since $s_1$ and $s_2$ are free). 
	Since there is no hole between $h_1$ and $b$, $S^{t^-}_{h_1+1/2} \leq S^{t^-}_{b-1/2}$. We immediately get $\height^{t^-}(h_1) = S^{t^-}_{h_1+1/2} +1/2 \leq S^{t^-}_{b-1/2} + 1/2 = \height^{t^-}(b)$. 
	The proof is similar for $h_2$. 
	
	We can now focus on \ref{P2} and on the modification of the height function $\height$ at time $t$, when a ball starting at $b$ fills the first hole $h_1$ at its right (the left case is exactly the same). Subfigure $\ref{fig:sep_stays_sep2}$ illustrates this modification. 
	
	If $b$ fills $h_1$ at time $t$: $\eta^t_{h_1} = 0 $, $\eta^t_b = 0$, and the other vertices' state is unchanged. In terms of height we get
	$\height^t(h_1) = \height^{t^-}(h_1) + 1/2$, $\height^t(b) = \height^{t^-}(b) + 1/2$ and $\forall x \in (h_1,b)$, $\height^t(x) = \height^{t^-}(x) + 1$ (and the other heights remain unchanged).
	So for every $x \in I$, $\height^t(x) \geq \height^{t^-}(x) > \height^{t^-}(s_2) = \height^t(s_2)$. Moreover, $\height(s_1), \height(b) \in \Z+1/2$, so $\height^{t^-}(s_1) > \height^{t^-}(b)$ implies $\height^{t^-}(s_1) > \height^{t^-}(b) + 1/2  =\height^t(b)$. In addition to that, for all $x\in[h_1, b)$, $\height^t(x) \leq \height^t(b) < \height^t(s_1)$, thus finally: for every $x \in I, \height^t(x) < \height^t(s_1) $
	
	We can do exactly the same reasoning if $b$ fills $h_2$, so \ref{P2} is true for every $t$.
\end{proof}

We now conclude the proof in the case $\db < \dt$ (of Theorem \ref{thm:defZ}).

Almost surely, the initial configuration $\eta^0$ on $\Z$ can be divided into an infinite number of finite enclosed intervals (between consecutive separators). \addition{Therefore, conditional on $\eta^0$, the golf process can be divided into independent blocks. To be able to apply Lemma \ref{lemma:no_ball_exits_enclosed_intervals} (in order to conclude that no ball exit any of these enclosed intervals), we need to ensure that no ball come from outside an enclosed interval. One way of doing this is to introduce the alternative process \textit{with sinks}, defined as the golf process with the subtlety that we allow the separators to be filled by an arbitrary number of balls (the separators are sinks). Here we can immediately apply Lemma \ref{lemma:no_ball_exits_enclosed_intervals}, and it implies that almost surely no ball fills any separator. In particular if we go back to the initial process where separators are empty holes that can be filled by at most one ball, then when no separator are filled in the golf process with sinks, these two models coincide, and this happens with probability 1.} It implies that our model is valid almost surely, and that the separators are almost surely not filled.

Concerning measurability, the separators enable to divide the golf process into finite golf processes, which are \textit{càdlàg}, hence the golf process on $\Z$ is almost surely \textit{càdlàg} (since we consider the product topology on $\mathcal{S}$).

\paragraph{Case $\db = \dt$:}
For any $t<1$, we consider the golf process until time $t$, which we denote $\eta[<t]$, that corresponds to the golf process with same initial configuration and dynamics as before, but in which we have activated only the balls with a clock smaller than $t$. It is straightforward that the golf process on $\Z$ is valid until time $t$, for any $t<1$ (because the density of balls is then $t\db < \dt$, so we can use the first part of the proof). If we consider a vertex $b$ containing a ball, its clock $\Clock(b)$ is such that $\Clock(b)<1$ almost surely and thus, in the process $\eta[<(\Clock(b) + 1)/2]$, this ball has been activated. Moreover, for any $s < t <1$, $(\eta[<t])[<s] = \eta[<s]$, therefore, we can define the golf process on $\Z$ as $\eta = \eta[<1] = \lim_{t\to 1, t < 1} \eta[<t]$: the limit exists and corresponds to the process we have described.

Finally, it is also almost surely \textit{càdlàg}: thanks to the previous paragraph it is straightforward that it is \textit{càdlàg} before any time $t<1$; the process is even continuous at $1^-$, since for the continuity on a product space it suffices to have the continuity for each projection, and each state is modified only once (at a time $<1$).
\end{proof}

From this proof we can deduce the two following corollaries:
\begin{corollary}\label{cor:Z_nb_infini_de_trous}
When $\db < \dt$, the set of separators $\Sep$ is included in the set of remaining holes $\TL$. In particular, there is an infinite number of holes on the left and on the right of 0. 
\end{corollary}

It is a consequence of Lemma \ref{lemma:infinite_nb_separators} and of the conclusion of the proof for the case $\db = \dt$.

\begin{corollary}\label{cor:Zdef_tant_que_pas_trop_de_balles}
If $\db > \dt$, the golf process on $\Z$ is valid until time $\dt/\db$ {(included)}. 
\end{corollary}
Indeed, the probability that a vertex contains a ball activated before time $\dt/\db$ is $\dt$. The golf process until time $\dt/\db$ is thus, up to a rescaling of time, similar to a golf process with ball density equal to $\dt$, which we have proven to be valid. In Section \ref{subsect:Z_pas_de_trous_si_db=dt}, Corollary \ref{cor:Zdef_tant_que_pas_trop_de_balles2} allows us to prove that actually, after time $\dt/\db$, the process is not valid anymore.
\\

\subsection{The golf process on $\Z$ is not always valid for different initial conditions}\label{subsect:discuss_def_Z}

\begin{figure}[]\centering
\begin{subfigure}[b]{0.9\textwidth}
	\centering
	\begin{tikzpicture}[inner sep=0.7mm, scale = 0.35]
		
		\draw[dotted, color=Black] (0,0) -- (0,1.5); \node[color = Black, above] at (0,1.5) {\footnotesize $0$};
		\draw[-,line width=1.5pt] (-13,1) -- (11,1);
		\draw[dotted, line width=1.5pt] (-14,1)--(-13,1); \draw[dotted, line width = 1.5pt] (11,1) -- (12,1);
		\foreach \x in {-12,-9,-6,-3,-10,-11,-7, -8,-4, -5,-2,-1}{
			\draw[black, fill] (\x,1) circle (0.2);
		}
		\foreach \x in {0,1,3,4,6,9,10,2, 5, 7, 8}{
			\draw[black, fill=white] (\x,1) circle (0.2);
		}
		\node[] at (-16,1) {\footnotesize ${\eta{(a)}}^0:$};
	\end{tikzpicture}
	\caption{${\eta{(a)}}^0$ such that $\forall x \geq 0,{\eta{(a)}}^0_x = \hole$ and $\forall x < 0, {\eta{(a)}}^0_x = \ball.$}
	\label{subfig:defZ_contreexs_1}
\end{subfigure}
\vspace{0.5cm}

\begin{subfigure}[b]{0.9\textwidth}
	\centering
	\begin{tikzpicture}[inner sep=0.7mm, scale = 0.35]
		
		\draw[dotted, color=Black] (0,0) -- (0,1.5); \node[color = Black, above] at (0,1.5) {\footnotesize $0$};
		\draw[-,line width=1.5pt] (-13,1) -- (11,1);
		\draw[dotted, line width=1.5pt] (-14,1)--(-13,1); \draw[dotted, line width = 1.5pt] (11,1) -- (12,1);
		\foreach \x in {-10,-11,-7, -8,-4, -5,-2,-1}{
			\draw[black, fill] (\x,1) circle (0.2);
		}
		\foreach \x in {-12,-9,-6,-3,0,1,3,4,6,9,10,2, 5, 7, 8}{
			\draw[black, fill=white] (\x,1) circle (0.2);
		}
		\node[] at (-16,1) {\footnotesize ${\eta{(b)}}^0:$};
	\end{tikzpicture}
	\caption{${\eta{(b)}}^0$ such that ${\eta{(b)}}^0_x = \hole$ if $x\geq 0$ or $x=0 \mod 3$ and ${\eta{(b)}}^0_x = \ball$ otherwise.}
	\label{subfig:defZ_contreexs_2}
\end{subfigure}
\vspace{0.5cm}

\begin{subfigure}[b]{0.9\textwidth}
	\centering
	\begin{tikzpicture}[inner sep=0.7mm, scale = 0.35]
		
		\draw[dotted, color=Black] (0,0) -- (0,1.5); \node[color = Black, above] at (0,1.5) {\footnotesize $0$};
		\draw[-,line width=1.5pt] (-13,1) -- (11,1);
		\draw[dotted, line width=1.5pt] (-14,1)--(-13,1); \draw[dotted, line width = 1.5pt] (11,1) -- (12,1);
		\foreach \x in {-11, -9,-7, -5,-3,-1}{
			\draw[black, fill] (\x,1) circle (0.2);
		}
		\foreach \x in {-12,-10,-6,-8,-4,-2,0,1,3,4,6,9,10,2, 5, 7, 8}{
			\draw[black, fill=white] (\x,1) circle (0.2);
		}
		\node[] at (-16,1) {\footnotesize ${\eta{(c)}}^0:$};
	\end{tikzpicture}
	\caption{${\eta{(c)}}^0$ such that ${\eta{(c)}}^0_x = \hole$ if $x\geq 0$ or $x=0 \mod 2$ and ${\eta{(c)}}^0_x = \ball$ otherwise.}
	\label{subfig:defZ_contreexs_3}
\end{subfigure}
		%
	%
\caption{Three possible initial configurations on $\Z$ that illustrate the complexity of defining a golf process on $\Z$. The golf process with initial configuration $\eta(a)^0$ is not valid. The process with initial configuration $\eta(c)^0$ is valid, {while the one with initial configuration $\eta(b)^0$ is only valid until time $1/2$.}  }
\label{fig:defZ_contrexs}
\end{figure}

{As we briefly mentioned earlier}, building a well-defined golf process on infinite graphs such as $\Z$ is not that simple. In this subsection, we discuss and illustrate this complexity with several examples (detailed on Figure \ref{fig:defZ_contrexs}). For these examples, we fix the initial condition, and we still assume that the clocks are independent and drawn uniformly at random on $[0,1]$.

We start with the first example, with initial condition defined on Subfigure \ref{subfig:defZ_contreexs_1}.
For any $x <0$, the ball in $x$ has clock $\Clock(x)$, and there almost surely exists an infinite number of other balls with clock smaller than $\Clock(x)$. It means that there cannot be any hole $y\geq 0$ free at time $\Clock(x)$, and thus the ball in $x$ will not be able to find a hole. Thus, the process would be frozen at any time $t>0$ (this prevents $(\eta^t, t\geq 0)$ to be \textit{càdlàg}).

This example also illustrates a limit of the Commutation property (which we had on finite graphs, see Remark \ref{remark:prop_commutation}. In fact, if the clocks were not random but were, for example, an increasing function of $\card{x}$ (for example, if $\Clockdet(x) = 1 - 1/(1+\card{x}), \forall x <0$), then it is quite straightforward that the golf process would be valid: for every $x <0$, the ball in $x$ would fill the hole in $\card{x}-1$ at time $\Clockdet(x)$.

Examples $b$ and $c$ (see Subfigures \ref{subfig:defZ_contreexs_2} and \ref{subfig:defZ_contreexs_3} for the definition of the initial configurations $\eta(b)^0$ and $\eta(c)^0$) are also very interesting. We can show (doing exactly as for Corollary \ref{cor:Zdef_tant_que_pas_trop_de_balles}, using the proof of Theorem \ref{thm:defZ}, the key tool being the existence of separators at any time $t<1$) that the golf process with initial configuration $\eta(c)^0$ is valid, whereas the one with initial configuration $\eta(b)^0$ is only valid until time $t =1/2$ (as for example $a$, it is after that immediately frozen, thus not \textit{càdlàg at $t=1/2$}).

For the golf process with initial configuration $\eta(c)^0$, it would actually be possible to prove that $\TL = \Z_{\geq 1}$.

\addition{All these examples illustrate the importance of separators, even when the initial configuration is not ergodic : when there are no separators anymore, the model is not valid. For the golf process with initial condition $\eta(c)^0$, there are no separators, but if we consider the process before time $t$ (with $t<1$), thus considering only balls with clock smaller than $t$, there is again an infinite number of separators. Similarly, on $\eta(b)^0$ there are separators until time $1/2$, and the process is valid until time 1/2. }

\addition{We can notice that the ``proportion'' of balls $\lim_{r\to\infty} \frac{\card{\{x \in \llbracket-r,r\rrbracket: \eta^0_i = 1\}}}{2r+1}$ is not invariant by translation. It does not enable to distinguish valid examples from non-valid ones.}

\subsection{Distribution of the block-size process - proof of Theorem \ref{thm:loiblocsZ}}\label{subsect:ZloiTL}

In this subsection, we prove Theorem \ref{thm:loiblocsZ}. We thus fix $\db$ and $\dt$, such that $\db < \dt$ and $\db+\dt = 1$.\\

{We first give the main idea of the proof.} It relies on a coupling between the {golf model} on $\Z$ and a golf model on $\znz$, with $n$ large enough, and containing $\nholes(n)$ holes and $\nballs(n)$ balls, such that $\nholes(n)/n\sim \dt$ and $\nballs(n)/n\sim \db$. We give and prove three key lemmas, and then we use them to conclude.

\textbf{Sketch of the proof.} We will see that we can fix some constant $M$ large enough such that, both on $\Z$ and on $\znz$, the probability that $\TL$ restricted to $\llbracket -M, M\rrbracket $ contains at least $R+1$ holes on the left and on the right of 0 is arbitrarily close to 1 (the same $M$ being valid for all $n$ large enough, see Step 2 below). This implies (morally) that the positions of the $R$ first holes on the left (and on the right) of 0 only depend on the golf process restricted to $\llbracket -M, M\rrbracket $. Moreover, when $\nholes(n)/n\sim \dt$ and $\nballs(n)/n\sim \db$, the initial configuration on $\znz$ restricted to $\llbracket -M, M\rrbracket $ is close (in distribution) to the initial configuration on $\Z$, also restricted to $\llbracket -M, M\rrbracket $ (see Step 3 below). It enables to construct a probability space on which the initial configurations coincide on $\llbracket -M, M\rrbracket $ with high probability, as well as the paths of the balls on this interval, that leads to a coupling of the final configurations (still restricted to $\llbracket -M, M\rrbracket $) that are then equal with high probability (Step 4 below gives the last lemma that enables to conclude properly). It then remains to study the distribution of the remaining holes on a large circle and show its convergence in distribution for the product topology. To do this, we show (in Step 1 below) that actually the process {$\Delta^{(n)} \TL$} has the same law as a family of independent variables conditioned by their sum, for which deriving the asymptotic distribution is routine.\\

Let us start the formal proof. 

\paragraph{Notation}For every $n$, we define $\nholes(n) \coloneqq \lfloor n \dholes \rfloor$ and $\nballs(n) \coloneqq n - \nholes(n)$, so that $\nballs(n) + \nholes(n) = n$, $\nballs(n) \sim \db n $ and $\nholes(n) \sim \dt n$. As usual, $\nl = \nholes(n)-\nballs(n)$.

\addition{As we deal with several families of random variables, we use the following notation to avoid ambiguity:

\textbullet We write, as before $\Deltan\TL$ for the block sizes process on $\znz$, and we write $\pcercle$ for probabilities concerning the golf process on $\znz$ when needed (sometimes, for compactness reasons, we omit it, but then the superscript ``$(n)$'' removes any ambiguity).

\textbullet Similarly, $\Delta\TL$ refers to the block sizes process on $\Z$, and most of the time we write $\pz$ for the associated probabilities.

\textbullet Finally, we introduce variables $(\zzlambda_i )_{i\in\Z}$, and we simply denote by $\P$ the associated probabilities.}

\paragraph{Step 1:} We let $\lambda = \db \dt$ (as stated in Theorem \ref{thm:loiblocsZ}). We begin by defining a family $(\zzlambda_i )_{i\in\Z}$ of independent random variables, taking non-negative even integer values, with
\begin{align*}
\forall k\geq 0,\ & \Prob{\zzlambda_0 = 2k} = \frac{(2k+1) C_k \lambda^k}{\mathcal{H}(\lambda)} \\
\et  \forall i \neq 0, \forall k\geq 0,\ & \Prob{\zzlambda_i = 2k} = \frac{ C_k \lambda^k}{\mathcal{G}(\lambda)} .
\end{align*}

\begin{lemma}	\label{lemma:varindepcondsomme}
We have the following equality for the golf process on $\znz$:
\begin{equation}
	{\cal L}\left(\left(\Deltan_i\TL\right)_{0 \leq i < \nl} \right)={\cal L}\left( \left(\zzlambda_i\right)_{0 \leq i < \nl}~\middle|~\sum \zzlambda_i= n-\nl\right). \label{eqn:varinderpcondsomme1}
\end{equation}

Moreover, the {conditioning} on the sum vanishes asymptotically: 
\begin{equation}
	\left(\Deltan_i\TL\right)_{-R \leq i < R} \overset{(d)}{\underset{n\to\infty}\longrightarrow} \left(\zzlambda_i\right)_{-R \leq i < R}. \label{eqn:varinderpcondsomme2}
\end{equation}
\end{lemma}

\begin{remark} The variables $\zzlambda_i$ ($i \neq 0$), conditioned by their sum have an expectation that tends to $\frac{2\db}{2\dt-1}$. In fact, 
\begin{equation*}
	(\nl - 1) \E\left[\zzlambda_1 \middle| \sum_{i } \zzlambda_i = n - \nl\right] + \E\left[\zzlambda_0 \middle| \sum_{i } \zzlambda_i = n - \nl\right] = n - \nl,
\end{equation*} and finally
\begin{align*}
	\E\left[\zzlambda_1 \middle|\sum_{i } \zzlambda_i = n - \nl\right] &\underset{n \to\infty}{\sim}  \frac{n - \nl}{\nl - 1} \\ &\underset{n \to\infty}{\to} \frac{1-\dt + \db}{\dt - \db} = \frac{2\db}{2\dt -1} = \E[\zzlambda_1].
\end{align*}

It is consistent with the value of $\lambda$ chosen in Theorem \ref{thm:loiblocsZ}, and it is the main reason why the conditioning on the sum vanishes asymptotically.
\end{remark}

\begin{proof} 
\textbf{Proof of (\ref{eqn:varinderpcondsomme1}).}
It suffices to show that, for every sequence $\left(b_i\right)_{0 \leq i < \nl}$ such that $\sum 2b_i = n - \nl$, then
\begin{equation}
	\P^{n,\nballs(n),\nholes(n),p}\left(\forall i, \Deltan_i\TL = 2b_i\right) = \P\left(\forall i , \zzlambda_i = 2b_i \middle| \sum_{i} \zzlambda_i = n - \nl\right), \label{eqn:varindepcondsomme}
\end{equation}  
where everywhere (and also until the end of the proof, unless specified otherwise) $i$ belongs to $\llbracket 0, \nl \llbracket$.
Equation (\ref{eqn:jolie_formule_deltaTL_znz}) in the proof of Corollary \ref{cor:blocks_and_forests} implies that the probability $\P^{n,\nballs(n),\nholes(n),p}\left(\forall i, \Deltan_i\TL = 2b_i\right)$ is proportional to $(2b_0+1)\prod_{j=0}^{\nl-1} C_{b_j}$. Moreover, it is immediate that:
\begin{equation*}
	\P\left(\forall i, \zzlambda_i = 2b_i \middle| \sum_{i} \zzlambda_i = n - \nl\right) = \frac{\lambda^{n-\nl} (2b_0+1) C_{b_0}\prod_{i}C_ {b_i}}{\Prob{\sum_{i } \zzlambda_i = n - \nl}\mathcal{H}(\lambda) {\mathcal{G}(\lambda)}^{\nl-1} }.
\end{equation*}
In particular, both left and right members of Equation (\ref{eqn:varindepcondsomme}) are proportional, thus as they both sum to 1 (when summing over all the sequences $\left(b_i\right)_{0 \leq i < \nl}$ such that $\sum 2b_i = n - \nl$) they are equal and finally Equation (\ref{eqn:varindepcondsomme}) is true.

\textbf{Proof of (\ref{eqn:varinderpcondsomme2}).}
We know thanks to (\ref{eqn:varinderpcondsomme1}) that, for every $\left(b_i\right)_{-R \leq i \leq R}$,

\begin{align}
	\P\left(\Deltan_i\TL = 2b_i, -R \leq i \leq R\right)&= \P\left( \zzlambda_i = 2b_i, \forall i \in \llbracket -R, R-1\rrbracket ~\middle|~ \sum_{i = -R}^{\nl-1-R} \zzlambda_i=n-\nl\right)\\
	&=  \prod_{i =-R}^{R-1} \P(\zzlambda_i = 2b_i)  \frac{\P\left(\sum_{i =R}^{\nl - 1 - R} \zzlambda_i = n - \nl - \sum 2b_i\right)}{\P\left(\sum_{i = -R}^{\nl - 1 - R} \zzlambda_i = n - \nl \right)}\label{eqn:limavecquotient}.
\end{align}
The central local limit theorem (see (\ref{eqn:TCLL3}) {in Appendix \ref{annex:TCLL}}) implies that 
\begin{equation}
	\frac{\sigma \sqrt{n}}{2} \P\left(\sum_{i =R}^{\nl - 1 - R} \zzlambda_i = n - \nl - \sum 2b_i\right) \underset{n\to\infty}{\to} \frac{1}{\sqrt{2\pi}}, \label{eqn:num_du_quotient}
\end{equation}
where $\sigma^2 = \Var{\zzlambda_j} \in (0,\infty)$, for every $j \neq 0$.

One could want to do the same for $\sum_{i = -R}^{\nl - 1 - R} \zzlambda_i$, but since $\zzlambda_0$ does not have the same distribution as the other $\zzlambda_i$, we first need to ``get rid of it'' before applying {the central local limit theorem}.

Markov's inequality implies that $\P (\zzlambda_0 \geq n^{1/3}) \leq c /n^{1/3}$ (with $c=\E [\zzlambda_0] $).
We can thus rewrite
\begin{align*}
	\P\left(\sum_{i = -R}^{\nl - 1 - R} \zzlambda_i = n - \nl \right)
	=&  \P\left(\sum_{i = -R, i \neq 0}^{\nl - 1 - R} \zzlambda_i = n - \nl - \zzlambda_0 \middle| \zzlambda_0 < n^{1/3} \right)(1-O (n^{-1/3})) \\
	&+ O (n^{-1/3}).
\end{align*}
Then, recall that $\lambda$ is such that $\E[\zzlambda_i] = \frac{2\db}{2\dt-1}$, for every $i\neq 0$ and that $\db + \dt = 1$. As a consequence, $\E[\sum_{i=-R}^{\nl - 1 - R} \zzlambda_i] = \frac{2\db}{2\dt-1} \nl \underset{n\to \infty}{\sim} n - \nl$. We can thus use Equality (\ref{eqn:TCLL2}) in Appendix \ref{annex:TCLL} (which is a consequence of the central local limit theorem) to prove that 
\begin{equation}
	\frac{\sigma \sqrt{n}}{2} \P\left(\sum_{i = - R, i\neq 0}^{\nl - 1 - R} \zzlambda_i = n - \nl - \zzlambda_0 \middle| \zzlambda_0 < n^{1/3} \right) \underset{n\to\infty}{\to} \frac{1}{\sqrt{2\pi}}\label{eqn:cv_unif_tcll}
\end{equation}
(if it is not clear enough, write the LHS of (\ref{eqn:cv_unif_tcll}) conditionally to $\{\zzlambda_0 = x\}$, for some $|x|<n^{1/3}$, and observe that the central local limit theorem gives the convergence to the limit uniformly for these $x$).	

This implies that 
\begin{equation}
	\frac{\sigma \sqrt{n}}{2} \P\left(\sum_{i = -R}^{\nl - 1 - R} \zzlambda_i = n - \nl \right) \underset{n\to\infty}{\to} \frac{1}{\sqrt{2\pi}}. \label{eqn:denom_du_quotient}
\end{equation}

Now, we can combine Equations (\ref{eqn:num_du_quotient}) and (\ref{eqn:denom_du_quotient}) to see that the quotient in Equation (\ref{eqn:limavecquotient}) converges to 1, and finally we conclude : 
\begin{align}
	\P\left(\Delta_i\TL^{(n)} = 2b_i, -R \leq i \leq R-1\right) \underset{n\to\infty}{\longrightarrow} \prod_{i =-R}^{R-1} \P(\zzlambda_i = 2b_i) \label{eqn:tropbien}.
\end{align}

\end{proof}

\paragraph{Step 2:} We show that with high probability (that is, up to a probability $\eps>0$), the $R$ first holes on the right and on the left of 0 belong to some deterministic finite (though large enough) interval around 0, which we will denote as $[M, M]$ (both for the golf model on $\Z$ and on $\znz$, for any $n$ large enough). It is proven in the following lemma, and we will see later that it implies that the process on $[M, M]$ suffices, with high probability, to determine the $R-1$ first free holes on the right and on the left of 0.

\begin{lemma}\label{lemma:yapleindetrous} Let $R > 0$ and $\eps >0$. There exists $M \in \N$ and $n_0 \in \N$ such that,
\begin{align}
	\pz\left( \card{\TL \cap [-M,0] } \geq R+2  \et \card{\TL \cap [1,M] } \geq R+2\right) > 1-\eps, \label{eqn:yapleindetrous_surZ}
\end{align}
and for every $n\geq n_0$,  
\begin{align}
	\P^{n,\nballs(n),\nholes(n),p}\left( \card{\TL \cap [-M,0] } \geq R+2 \et \card{\TL \cap [1,M] } \geq R+2 \right) > 1-\eps.\label{eqn:yapleindetrous_surZnZ}
\end{align}

\end{lemma} 

\begin{proof} On $\Z$, we can define an ordering of $\TL$: we define the random variables $\TL_i$, for every $i\in \Z$, corresponding to the $i$th hole with respect to 0 (the hole indexed by 0 is the first hole on the left of 0: $\TL_0 = \sup \{h\in \TL : h \leq 0\}$; the other holes are such that $\forall i\in\Z, \TL_i < \TL_{i+1}$). In this way, it is consistent with the definition of $\Delta\TL$, since $\Delta_0\TL = \TL_1 - \TL_0 -1$. Similarly, we define the variables $(\TL^{(n)}_i)_{i\in \Z/\nl\Z}$ corresponding to the ordering of the elements of $\TL^{(n)}$ on $\znz$. 

For the golf process on $\Z$, we have proven (see Corollary \ref{cor:Z_nb_infini_de_trous}) that there is an infinite number of holes, on the left and on the right of 0. Thus $\TL_i$, with $i\in\llbracket -R-1,R+2\rrbracket$ are almost surely finite random variables, and this directly implies (\ref{eqn:yapleindetrous_surZ}).

On $\znz$, we can use Lemma \ref{lemma:varindepcondsomme} to prove that $\left(\TL_i^{(n)}\right)_{-R-1 \leq i \leq R+2}$ converges in distribution (it is a consequence of (\ref{eqn:tropbien}) and of the fact that the point of index 0 is uniform in the block that contains 0, of size $\Deltan_0\TL$, that converges in distribution).
\end{proof}

\paragraph{Step 3:} In order to couple the golf model on $\Z$ with the golf model on $\znz$ (on a large neighborhood of $0$), we want to show that the distributions of the two initial configurations on this large neighborhood become arbitrarily close (when $n$ goes to infinity).

\begin{lemma}[{uniformity of the approximation}]\label{lemma:approxunif} Let $I$ be some finite interval of $\Z$.
\[ \sup_{\left(x_i\right)_{i\in I} \in \{\ball, \hole\}^I } \left| \frac{\P^{n,\nballs(n),\nholes(n),p}\left(\forall i \in I, \eta^0_i = x_i\right)}{\pz\left(\forall i \in I, \eta^0_i = x_i\right)} -1 \right| \limit{}{n \to \infty} 0.\]

(Here, on $\znz$, on the left of 0, the index $-i\in I$ with $-i \leq 0$ needs to be seen as $-i = n-i \in \znz$, so that on $\Z$ and $\znz$, around 0, the indices can be identified).
\end{lemma}

\begin{proof}
Let $\left(x_i\right)_{i\in I} \in \{\ball, \hole\}^I$. One can easily compute:
\begin{align}
	\P^{n,\nballs(n),\nholes(n),p}\left(\forall i \in I, \eta^0_i = x_i\right)  &= \frac{\binom{n-\card{I}}{\nholes(n) - \card{\{i\in I : x_i = \hole} }}{\binom{n}{\nholes(n)}} \\&
	\underset{n\to\infty}{\sim} \frac{{\nholes(n)}^{\card{\{i\in I : x_i = \hole\}}} {\nballs(n)}^{\card{\{i\in I : x_i = \ball\}}}}{n^{\card{I}}}\\
	&\underset{n \to \infty}{\longrightarrow} \dholes^{\card{\{i\in I : x_i = \hole\}}} \db^{\card{\{i\in I : x_i = \ball\}}}. \label{eqn:untruc}
\end{align}
And the right member of (\ref{eqn:untruc}) is equal to $\pz\left(\forall i \in I, \eta^0_i = x_i\right)$.

\end{proof}

\paragraph{Step 4:} We give and prove the last technical lemma we need to conclude. It is a deterministic lemma, in the sense that it involves considerations on a fixed final configuration and states that something cannot happen. It is illustrated on Figure \ref{fig:loi_Z_lemme_technique_memes_trous}.
\begin{lemma}\label{lemma:loi_Z_lemme_technique_memes_trous}
We assume that $\eta$ and $\eta'$ are two golf processes (one on $\znz$ and one on $\Z$) such that the initial configurations $\eta^0$ and $\eta'^0$ coincide on some interval $I = \llbracket x, y \rrbracket$. We also assume that the clocks are the same on $I$, and that the trajectories of the balls of $I$ are given by the same infinite paths, in the following way: there is a family of infinite paths $(w^{(i)})_{i\in I}$ such that for every $i$, $w^{(i)}$ is a path on $\Z$ starting at $i$, and a ball starting at vertex $b$ at some time $t = \Clockdet(b)$ does the walk given by $w^{(b)}$ stopped at the first empty hole it finds (and fills it, as usual).

If there is $z$ such that $x < z < y$, $x$ and $y$ are empty holes in $\eta'^1$ and $z$ is a empty hole in $\eta^1$, then $z$ is also empty in $\eta'^1$.
\end{lemma}

\begin{proof}
We prove this by contradiction. 

We assume that $h_0 \coloneqq z$ is not empty in $\eta'^1$. It means that there exists $b_1 \in I$ such that at time $t_1 \coloneqq \Clockdet(b_1)$, $h_0$ was the first hole of $w^{(b_1)}$ that was empty in ${\eta'}^{t_1}$. Notice that since $x$ and $y$ are empty, $b_1$ necessarily belongs to $I$. Since, in $\eta$, the ball in $b_1$ did not reach $h_0$, it means that there exists $h_1$ that was empty in $\eta^{t_1^-}$ but not empty in ${\eta'}^{t_1^-}$ (and such that $h_1$ appears before $h_0$ in $w^{(b_1)}$).

\addition{We apply this reasoning inductively : from a hole $h_i$ empty in $\eta^{t_i^-}$ but not empty in ${\eta'}^{t_i^-}$, we can obtain, with an identical reasoning, a ball $b_{i+1}\in I$ with clock $t_{i+1} \coloneqq \Clockdet(b_{i+1}) < t_i$, that fills $h_i$ in $\eta'$ at time $t_{i+1}$. It implies that there exists a hole $h_{i+1}$ which is empty in $\eta^{t_{i+1}^-}$ but not empty in ${\eta'}^{t_{i+1}^-}$.}
We thus build an infinite sequence $(b_k)_{k\geq 0}$ of balls such that for every $k$, $b_k\in I$ and $\Clockdet(b_{k+1}) < \Clockdet(b_k)$. As $I$ is finite, we have a contradiction, and deduce that $z$ is empty.
\end{proof}

\begin{figure}
\centering
\begin{tikzpicture}[inner sep=0.7mm, scale = 0.35]
	
	\draw[dotted, color=Black] (0,-3.5) -- (0,3.5); \node[color = Black, below] at (0,-3.5) {\footnotesize $z$};
	\draw[dotted, color=Black] (-9,-3.5) -- (-9,3.5); \node[color = Black, below] at (-9,-3.5) {\footnotesize $x$};
	\draw[dotted, color=Black] (6,-3.5) -- (6,3.5); \node[color = Black, below] at (6,-3.5) {\footnotesize $y$};
	\draw[-,line width=1.5pt] (-13,2.5) -- (11,2.5);
	\draw[dotted, line width=1.5pt] (-14.5,2.5)--(-13,2.5); \draw[dotted, line width = 1.5pt] (11,2.5) -- (12.5,2.5);
	\foreach \x in {-11, -9,-7, -5,-3,-1,-12,-10,-6,-8,-4,-2,0,1,3,4,6,9,10,2, 5, 7, 8}{
		\draw[black,line width = 0.8pt]  (\x,2.3) -- (\x,2.7);
	}
	\node[] at (-16.5,2.5) {\footnotesize ${\eta}^1:$};
	
	\draw[-,line width=1.5pt] (-13,-2.5) -- (11,-2.5);
	\draw[dotted, line width=1.5pt] (-14.5,-2.5)--(-13,-2.5); \draw[dotted, line width = 1.5pt] (11,-2.5) -- (12.5,-2.5);
	\foreach \x in {-11, -9,-7, -5,-3,-1,-12,-10,-6,-8,-4,-2,0,1,3,4,6,9,10,2, 5, 7, 8}{
		\draw[black,line width = 0.8pt]  (\x,-2.3) -- (\x,-2.7);
	}
	\draw[black, fill = orange] (0,-2.5) circle (0.3);
	\draw[black, fill = white] (-9,-2.5) circle (0.3);
	\draw[black, fill = white] (6,-2.5) circle (0.3);
	\draw[black, fill = white] (0,2.5) circle (0.3);
	\node[] at (-16.5,-2.5) {\footnotesize ${\eta'}^1:$};
	\draw [decorate, decoration = {calligraphic brace, raise = 2pt, amplitude = 4pt}] (-9,4) --node[yshift=15] {$I$}   (6,4);
\end{tikzpicture}
\caption{Illustration of Lemma \ref{lemma:loi_Z_lemme_technique_memes_trous}. We show two golf processes $\eta$ and $\eta'$ (typically, one of them is the golf process on $\Z$ and the other one is on $\znz$). A white circle corresponds to an empty hole in the final configuration, while the orange one represents a hole on which we do not do any assumption (but we will show that it is also free). The other vertices can contain holes, empty or not. We assume that the initial configurations restricted to I coincide ($\eta^0_i = \eta'^0_i, \forall i \in I$) and that the trajectories of the balls are given by the same infinite paths, as explained in Lemma \ref{lemma:loi_Z_lemme_technique_memes_trous}. Lemma \ref{lemma:loi_Z_lemme_technique_memes_trous} states that if $x$ and $y$ are free holes in $\eta'^1$ and $z \in \rrbracket x, y \llbracket $ is free in $\eta^1$, then $z$ is also free in $\eta'^1$. 
}
\label{fig:loi_Z_lemme_technique_memes_trous}
\end{figure}

\paragraph{{Conclusion :}}We can now prove Theorem \ref{thm:loiblocsZ}.

\begin{proof}[Proof of Theorem \ref{thm:loiblocsZ}]
We let $R > 0$ and $\eps > 0$ and take $M$ large enough so that Equations (\ref{eqn:yapleindetrous_surZ}) and (\ref{eqn:yapleindetrous_surZnZ}) hold.

\medskip
{As explained above, for $n$ large enough, it is possible to construct a coupling between the initial configurations on $\znz$ and $\Z$ so that the initial configurations are equal on $I = [-M,M]$ with high probability.}
We can then use Lemma \ref{lemma:approxunif}: there exists $n$ (large enough) such that
\begin{align}
	\sup_{\left(x_i\right)_{i\in I} \in \{\ball, \hole\}^I } \left| \P^{n,\nballs(n),\nholes(n),p}\left(\forall i \in I, {\eta^0}^{(n)}_i = x_i\right) - \pz\left(\forall i \in I, \eta^0_i  = x_i\right)\right| < \frac{\eps/2}{2^{2M+1}} \label{eqn:couplage2}
\end{align}
(here and in the sequel, to avoid ambiguity, we write ${\eta^0}^{(n)}$ for the initial configuration on $\znz$).

It implies that there exists a coupling of ${\eta^0}^{(n)}$ and $\eta^0$ such that\footnote{More precisely, there exists a probability space $(\Omega, \mathcal{A},\P)$ on which one can find copies $\bar{\eta}^0_i$ of ${\eta}^0_i$ under $\pz$ and ${\bar{\eta}^0_i}^{(n)}$ of ${{\eta}^0_i}^{(n)}$ under $\pcercle$ such that with probability at least $1-\eps$, $\forall i\in I,{\bar{\eta}^0_i} = {\bar{\eta}^0_i}^{(n)}$ }
\begin{align}
	\Prob{\forall i \in I, {\eta^0}^{(n)}_i = \eta^0_i} \geq 1-\eps.
\end{align}
We assume in what follows that ${\eta^0}^{(n)}$ and $\eta^0$ are coupled this way. We extend the coupling to the clocks, letting $\Clock(v)$ (with $v\in \Z$) be independent with distribution $\Uzo$, and then defining the clocks on $\znz$ by $\forall i \in I, \Clock^{(n)}(v) \coloneqq \Clock(v)$, and for any $v \notin I$, $\Clock^{(n)}(v)\sim \Uzo$ also independently of the other variables.

We then couple the trajectories of the balls. We define, for every $v\in \Z$, an infinite random walk $(w^{(v)}_k)_{k\geq 0}$ with transition matrix $\PMC$, starting at $w^{(v)}=0 = v$. We then define the golf processes on $\Z$ so that a ball starting at $u$ at time $\Clock(u)$ does a random walk with steps given by $w^{(u)}$ until it hits a free hole. Similarly, on $\znz$, a ball starting at $u\in \znz$ at time $\Clock^{(n)}(u)$ uses $w^{(i)}$ for its random walk, where
$i$ is the unique integer such that $-n/2 < i \leq n/2$ and $i \mod n = u$. We then have a coupling of the golf processes on $\Z$ and $\znz$.

\begin{figure}
	\centering
	\begin{tikzpicture}[inner sep=0.7mm, scale = 0.35]
		
		\draw[dotted, color=Black] (-11,-3.5) -- (-11,3.5); 
		\draw[dotted, color=Black] (-4,-3.5) -- (-4,3.5); 
		\draw[dotted, color=Black] (-8,-3.5) -- (-8,3.5); 
		\node[color = Black, below] at (-11.4,-1) {\footnotesize $k'$};
		\node[color = Black, above] at (-12,1) {\footnotesize $k$};
		\node[color = Black, below] at (-2,-1) {\footnotesize $0$};
		\node[color = Black, above] at (-2,1) {\footnotesize $0$};
		\node[color = Black, above] at (9,1) {\footnotesize $l$};
		\draw[-,line width=1.5pt] (-13,2.5) -- (11,2.5);
		\draw[dotted, line width=1.5pt] (-14.5,2.5)--(-13,2.5); \draw[dotted, line width = 1.5pt] (11,2.5) -- (12.5,2.5);
		\node[] at (-16.5,2.5) {\footnotesize ${\eta}^1:$};
		
		\draw[-,line width=1.5pt] (-13,-2.5) -- (11,-2.5);
		\draw[dotted, line width=1.5pt] (-14.5,-2.5)--(-13,-2.5); \draw[dotted, line width = 1.5pt] (11,-2.5) -- (12.5,-2.5);
		\draw[black,line width = 0.8pt]  (-2,-2.3) -- (-2,-2.7); \draw[black,line width = 0.8pt]  (-2,2.3) -- (-2,2.7);
		\draw[black, fill = white] (-11,-2.5) circle (0.3);
		\draw[black, fill = orange] (-11,2.5) circle (0.3);
		\draw[black, fill = white] (-4,-2.5) circle (0.3);
		\draw[black, fill = orange] (-4,2.5) circle (0.3);
		\draw[black, fill = white] (-8,-2.5) circle (0.3);
		\draw[black, fill = orange] (-8,2.5) circle (0.3);
		\draw[black, fill = white] (-12,2.5) circle (0.3);
		\draw[black, fill = white] (9,2.5) circle (0.3);
		\node[] at (-16.5,-2.5) {\footnotesize ${\eta'}^1:$};
		\draw [decorate, decoration = {calligraphic brace, raise = 2pt, amplitude = 4pt}] (-12.5,3.6) --node[yshift=15] {\footnotesize $R+2$ holes}   (-1.5,3.6);
		\draw [decorate, decoration = {calligraphic brace, raise = 2pt, amplitude = 4pt}] (-1.5,-3.6) --node[yshift=-15] {\footnotesize $R+2$ holes}   (-11.5,-3.6);
	\end{tikzpicture}
	\caption{ \addition{Illustration of the fact that the $R+2$ empty holes to the left of 0 (included) in $\eta^1$ coincide with the $R+2$ empty holes to the left of 0 in ${\eta}'^1$ : we let $k$ (resp.\ $k'$) be the position of the $R+2$th hole to the left of 0 (included) in $\eta^1$ (resp.\ ${\eta}'^1$). We assume that $k\leq k'$ (the case $k'\leq k$ is symmetric). There are $R+2$ empty holes in $[k',0]$ in ${\eta'}^1$, and Lemma \ref{lemma:loi_Z_lemme_technique_memes_trous} implies that all these holes are also empty holes in ${\eta}^1$ (since we know that there exists a free hole to the right of 0 in $\eta$, at position $l$ on the picture): all the orange vertices in $\eta^1$ are empty holes. As there are exactly $R+2$ empty holes in $[k,0]$ in ${\eta}^1$, it implies that $k' = k$ and that in the interval $[k,0]$, a hole is free in $\eta^1$ if and only if it is free in ${\eta'}^1$.}
	}
	\label{fig:loi_Z_memes_trous}
\end{figure}

Therefore, combining this with Lemma \ref{lemma:yapleindetrous}, we deduce that with probability at least $1-3\eps$, $\forall i \in I, {\eta^0}^{(n)}_i = \eta^0_i$, $\card{\TL \cap [-M,0] } \geq R+2$ and $\card{\TL \cap [1,M] } \geq R +2$. Lemma \ref{lemma:loi_Z_lemme_technique_memes_trous} enables to conclude that on this event, the positions of the remaining holes indexed from $-R$ to $R+1$ also coincide (see Figure \ref{fig:loi_Z_memes_trous}):  $(\Delta_i\TL)_{-R \leq i \leq R} = (\Deltan_i\TL)_{-R \leq i \leq R}$ with probability at least $1 - 3\eps$.

\medskip
We can conclude about the distribution of $(\Delta_i\TL)_{-R \leq i \leq R}$. We have seen (Lemma \ref{lemma:varindepcondsomme}) that $(\Deltan_i\TL)_{-R \leq i \leq R}$ converges in distribution to $(\zzlambda_i)_{-R \leq i \leq R}$.
In the previous paragraph, we showed that $(\Deltan_i\TL)_{-R \leq i \leq R}$ converges in probability to $(\Delta_i\TL)_{-R \leq i \leq R}$. It implies that $(\Delta_i\TL)_{-R \leq i \leq R}$ has the same distribution as $(\zzlambda_i)_{-R \leq i \leq R}$.

\end{proof}

\subsection{Distribution of the block-size process in general}\label{subsect:blocks_Z_cascomplexe}

We give the following theorem for the distribution of $\Delta\TL$. Notice that we do not assume that $\db + \dt = 1$ anymore.

\begin{theorem} \label{thm:blocsZ_general} We assume that $0 < \db < \dt$ and $\db + \dt \leq 1$. Then, under $\pdbdh$,
for every $\ell_0\geq0$, \begin{align}
	\Prob{\Delta_0\TL = \ell_0} = (\dt - \db) \sum_{b_0 : 2 b_0 \leq \ell_0} \frac{\ell_0 + 1}{b_0 + 1}\binom{\ell_0}{b_0, b_0, \ell_0 - 2 b_0} \db^{b_0} \dt^{b_0 + 1}(1 - \db - \dt)^{\ell_0 - 2 b_0}. \label{eqn:lalimite!}
\end{align} 
\end{theorem}
One could obtain, with the same reasoning, a theorem for the distribution of the whole process $\left(\Delta_i\TL\right)$ (similar to what we have in Theorem \ref{thm:loiblocsZ} when $\db+\dt = 1$), but it would involve heavy computation that we leave as an exercise to the interested reader.

When $\db + \dt = 1$, the theorem simplifies and gives: for every $k\geq0$,
\begin{align}
\Prob{\Delta_0\TL = 2 k } = (\dt - \db) \frac{2 k  + 1}{ k+ 1}\binom{2 k }{k} {\db\dt}^{k}. \label{eqn:lalimite_compat_thm_loi_blocks_db=dt}
\end{align} 
which is consistent with Theorem \ref{thm:loiblocsZ}.

\begin{proof}[Proof of Theorem \ref{thm:blocsZ_general}]
The idea of the proof is exactly the same as for the proof of Theorem \ref{thm:loiblocsZ}: we show that, when $\nballs = \lfloor \db n\rfloor$ and $\nholes = \lfloor \dt n\rfloor$, $\Deltan_0\TL \overset{(d)}\to L_0$ as $n$ goes to infinity, for some random variable $L_0$ whose law is given by the right-hand side of (\ref{eqn:lalimite!}); then, Lemmas \ref{lemma:yapleindetrous} and \ref{lemma:approxunif} still hold and we can again do a coupling between the initial configurations on $\znz$ (for some $n$ large enough) and $\Z$. 

We thus only need to prove that $\Deltan_0\TL$ converges in distribution as $n$ goes to infinity.

First, we can show the following claim:

\begin{claim}\label{claim:loi_bloc0}
	For any integer $\ell_{0}\geq 0$ such that $\ell_0 \leq n-2$,
	\begin{align}
		&\Prob{\Deltan_0\TL = \ell_0} \\
		&= \sum_{b_0 : b_0 \leq \nballs \et 2b_0 \leq \ell_0} \frac{\ell_0 + 1}{b_0 + 1} \binom{\ell_0}{b_0,b_0,\ell_0-2b_0} \  \frac{ \frac{\nl - 1}{\nholes - b_0 - 1} \binom{n - \ell_0 - 2}{\nballs - b_0, \nholes - b_0 - 2, n - \ell_0 - \nballs - \nholes + 2b_0}}{\binom{n}{\nballs, \nholes, n-\nballs-\nholes}}.
	\end{align}
\end{claim}
We postpone its proof to the end of the proof of Theorem \ref{thm:blocsZ_general}.

Now, for every $n$, we let $\nballs(n) \coloneqq \lfloor \db n \rfloor$ and $\nholes(n) \coloneqq \lfloor \dt n \rfloor$. Then, using the Claim and Stirling's formula, we can easily prove that for any $\ell_0$,
\begin{align}
	\Prob{\Deltan_0\TL = \ell_0} \underset{n \to \infty}\to f(\ell_0) \label{eqn:lalimitecoooool}
\end{align} 
where we have defined, for every $\ell_0$, $f(\ell_0)$ as the right-hand side of Equation (\ref{eqn:lalimite!}). 
\begin{claim} \label{claim:c'est_bien_une_proba}
	We have the following:
	\begin{equation}
		\sum_{\ell_0 \geq 0} f(\ell_0) = 1. \label{eqn:c'est_bien_une_proba}
	\end{equation}
\end{claim}
The proof of this claim is quite technical (yet very interesting, since it involves some nice combinatorial interpretation of this sum), so we postpone it to the end of the subsection.

Thanks to these previous claims, we can deduce that the law of $\Deltan_0 \TL$ converges to $f$.

\end{proof}

\begin{proof}[Proof of Claim \ref{claim:loi_bloc0}.] The computation follows exactly the same reasoning as the proof of Theorem \ref{thm:loitrousresiduels}, and we detail the main subtleties. 
There are $\ell_0 + 1$ pairs $(x_0,x_1)$ such that $x_0 \leq 0 < x_1$ and $x_1 - x_0 - 1 \mod n = \ell_0$. For any such $(x_0,x_1)$ we want to compute the probability of the event $A \coloneqq \left\{\TL \cap \llbracket x_0 , x_1\rrbracket = \{x_0, x_1\}\right\}$ (and this will provide $\Prob{\Deltan_0\TL = \ell_0} = (\ell_0 + 1) \Prob{A}$).

It is possible to obtain a decomposition result analogous to Equation (\ref{eqn:decomposition}) in Lemma \ref{lemma:simpl_thm_loi_trous_residuels}, stating that if we condition on the number of holes and balls between $x_0$ and $x_1$ being $b_0$,
\begin{align}
	\Prob{A\middle| x_0,x_1\in\Tinit \et \card{ \llbracket x_0 , x_1\rrbracket \cap \Binit} = b_0  } =& \ \P^{\ell_0, b_0, b_0+1,p}(\TL = \{0\} | 0\in \Tinit) \\& \times  \P^{n-\ell_0, \nballs - b_0, \nholes - b_0- 1,p}(\TL = \{0\} | 0\in \Tinit)
\end{align}
(this key  decomposition is illustrated on Figure \ref{fig:bloc_contenant0}).
Then, as before, $\P^{\ell_0, b_0, b_0+1,p}(\TL = \{0\} | 0\in \Tinit) = \frac{1}{b_0+1}$. The difference that appears with the proof of Theorem \ref{thm:loitrousresiduels} lies in the computation of $ \P^{n-\ell_0, \nballs - b_0, \nholes - b_0- 1,p}(\TL = \{0\} | 0\in \Tinit)$. Since there are $(\nholes - b_0 - 1) - (\nballs - b_0) = \nl - 1$ remaining holes in the associated golf model, it implies (by invariance by rotation, as in Lemma \ref{lemma:casinitloitrousresiduels}) that 
\begin{equation}
	\P^{n-\ell_0, \nballs - b_0, \nholes - b_0- 1,p}(\TL = \{0\} ) = \frac{\nl - 1}{n - \ell_0}
\end{equation}
and finally that  
\begin{equation}
	\P^{n-\ell_0, \nballs - b_0, \nholes - b_0- 1,p}(\TL = \{0\} | 0\in \Tinit) = \frac{\nl - 1}{\nholes - b_0  - 1}.
\end{equation}
The conclusion of the claim follows from a computation similar to Equations (\ref{eqn:thm_loi_trous_residuels_cercle_1}) and (\ref{eqn:thm_loi_trous_residuels_cercle_2})

\begin{figure}[t]\centering
	{\begin{tikzpicture}[line cap=round,line join=round,>=triangle 45,scale = 1.1]
			\draw(0,0) circle (2cm);
			\begin{scriptsize}
				\draw [color=black,fill = black] (-1.41,1.41) circle (1.5pt);
				\draw [color=black,fill = black] (-0.77,1.85) circle (1.5pt);
				\draw [color=black,fill = black] (-1.85,0.77) circle (1.5pt);
				\draw [color=black,fill = black] (-2,0) circle (1.5pt);
				\draw [color=black,fill = black] (-1.85,-0.77) circle (1.5pt);
				\draw [color=black,fill = black] (-1.41,-1.41) circle (1.5pt);
				\draw [color=black,fill = black] (-0.77,-1.85) circle (1.5pt);
				\draw [color=black,fill = black] (0,-2) circle (1.5pt);
				\draw [color=black,fill = black] (0.77,-1.85) circle (1.5pt);
				\draw [color=black,fill = black] (1.41,-1.41) circle (1.5pt);
				\draw [color=black,fill = black] (1.85,-0.77) circle (1.5pt);
				\draw [color=black,fill = black] (1.85,0.77) circle (1.5pt);
				\draw [color=black,fill = black] (2,0) circle (1.5pt);
				\draw [color=black,fill = black] (1.41,1.41) circle (1.5pt);
				\draw [color=black,fill = black] (0.77,1.85) circle (1.5pt);
				\draw [color=black,fill = black] (0,2) circle (1.5pt);
				\node[above] at (0,-1.8) {\redcolor{$0$}};
				\node [black,right, xshift = 7pt] at (-1.85,-0.77) {$x_0$};
				\node [black,above left, xshift = -5pt] at (1.41,-1.41) {$x_{1}$};
				
				\draw [color=black,fill = white] (-1.85,-0.77) circle (3.5pt);
				\draw [color=black,fill = white] (1.41,-1.41) circle (3.5pt);
				
				\node [align=left] at (2.5,1.8) { {$h_1 \times \hspace{-0.1cm}\vcenter{\hbox{ \begin{tikzpicture}[scale = 0.7]
									\draw [color=black,fill = white] (1.85,-0.77) circle (3.5pt);
						\end{tikzpicture} }}\hspace{-0.1cm} + b_1 \times \hspace{-0.1cm}\vcenter{\hbox{ \begin{tikzpicture}[scale = 0.7]
									\draw [color=black,fill = black] (1.85,-0.77) circle (3.5pt);
						\end{tikzpicture} }}\hspace{-0.1cm} $}
				};
				\node [] at (-0.2,-2.5) { {
						$b_0 \times \hspace{-0.1cm}\vcenter{\hbox{ \begin{tikzpicture}[scale = 0.7]
									\draw [color=black,fill = white] (1.85,-0.77) circle (3.5pt);
						\end{tikzpicture} }}\hspace{-0.1cm} + b_0 \times \hspace{-0.1cm}\vcenter{\hbox{ \begin{tikzpicture}[scale = 0.7]
									\draw [color=black,fill = black] (1.85,-0.77) circle (3.5pt);
						\end{tikzpicture} }}\hspace{-0.1cm} $}};
				\fill[cyan, opacity = 0.3] (0,0) ++(187.5:2.2)
				arc(187.5:-30:2.2)
				-- ++(-210:0.4)
				arc(-30:187.5:1.8)
				-- cycle;
				
				\fill[green, opacity = 0.3] (0,0) ++(-60:2.2)
				arc(-60:-142.5:2.2)
				-- ++(37.5:0.4)
				arc(-142.5:-60:1.8)
				-- cycle;
				\draw [color=orange] (-1.85,-0.77) circle (7pt);
				\draw [color=orange] (1.41,-1.41) circle (7pt);

			\end{scriptsize}
			\node [align=left] at (2.9,0) {\LARGE {$\displaystyle = $}};
			
			\begin{scope}[shift = {(4.8,0)}]
				\def\ray{1}
				\draw(0,0) circle (\ray);
				\begin{scriptsize}
					\draw [color=black,fill = black] (54:\ray) circle (1.5pt);
					\draw [color=black,fill = black] (198:\ray) circle (1.5pt);
					\draw [color=black,fill = black] (126:\ray) circle (1.5pt);
					\draw [color=black,fill = black] (-18:\ray) circle (1.5pt);
					\draw [color=black,fill = black] (-90:\ray) circle (1.5pt);
					\node[above] at (0,-0.8*\ray) {\redcolor{$0$}};
					\draw [color=black,fill = white](-90:\ray) circle (3.5pt);
					
					\node [] at (-0,1.5) { {$b_0 \times \hspace{-0.1cm}\vcenter{\hbox{ \begin{tikzpicture}[scale = 0.7]
										\draw [color=black,fill = white] (1.85,-0.77) circle (3.5pt);
							\end{tikzpicture} }}\hspace{-0.1cm} + b_0 \times \hspace{-0.1cm}\vcenter{\hbox{ \begin{tikzpicture}[scale = 0.7]
										\draw [color=black,fill = black] (1.85,-0.77) circle (3.5pt);
							\end{tikzpicture} }}\hspace{-0.1cm} $}};
					\fill[green, opacity = 0.3] (0,0) ++(-54:\ray+0.2)
					arc(-54:234:\ray+0.2)
					-- ++(54:0.4)
					arc(234:-54:\ray-0.2)
					-- cycle;
					
					\draw [color=orange] (-90:\ray) circle (7pt);

				\end{scriptsize}
			\end{scope}
			
			\node [align=left] at (6.4,0) {\LARGE {$\times$}};
			
			\begin{scope}[shift = {(8.5,0)}]
				\def\ray{1.5}
				\draw(0,0) circle (\ray);
				\begin{scriptsize}
					\draw [color=black,fill = black] (237.2:\ray) circle (1.5pt);
					\draw [color=black,fill = black] (204.5:\ray) circle (1.5pt);
					\draw [color=black,fill = black] (171.8:\ray) circle (1.5pt);
					\draw [color=black,fill = black] (139.1:\ray) circle (1.5pt);
					\draw [color=black,fill = black] (106.3:\ray) circle (1.5pt);
					\draw [color=black,fill = black] (73.6:\ray) circle (1.5pt);
					\draw [color=black,fill = black] (41:\ray) circle (1.5pt);
					\draw [color=black,fill = black] (8.1:\ray) circle (1.5pt);
					\draw [color=black,fill = black] (-24.5:\ray) circle (1.5pt);
					\draw [color=black,fill = black] (-57.3:\ray) circle (1.5pt);
					\draw [color=black,fill = black] (-90:\ray) circle (1.5pt);
					\node[above] at (0,-0.8*\ray) {\redcolor{$0$}};
					
					\draw [color=black,fill = white](-90:\ray) circle (3.5pt);
					\node [] at (-0,2) { {$h_1 \times \hspace{-0.1cm}\vcenter{\hbox{ \begin{tikzpicture}[scale = 0.7]
										\draw [color=black,fill = white] (1.85,-0.77) circle (3.5pt);
							\end{tikzpicture} }}\hspace{-0.1cm} + b_1 \times \hspace{-0.1cm}\vcenter{\hbox{ \begin{tikzpicture}[scale = 0.7]
										\draw [color=black,fill = black] (1.85,-0.77) circle (3.5pt);
							\end{tikzpicture} }}\hspace{-0.1cm} $}};
					
					\fill[cyan, opacity = 0.3] (0,0) ++(-73.5:\ray+0.2)
					arc(-73.5:253.5:\ray+0.2)
					-- ++(73.5:0.4)
					arc(253.5:-73.5:\ray-0.2)
					-- cycle;
					\draw [color=orange] (-90:\ray) circle (7pt);	
				\end{scriptsize}
			\end{scope}
			
	\end{tikzpicture}}
	\hfill
	\caption[]{Illustration of the computation of $\Prob{A \middle| x_0,x_1\in\Tinit \et \card{ \llbracket x_0 , x_1\rrbracket \cap \Binit} = b_0  }$. Here, we give an example for $n = 16$, $\nballs = 5$, $\nholes =9$ and $\ell_0 = x_1 - x_0 - 1 \mod n = 4$. In the green interval (the smaller one), there need to be the same number of balls and holes (equal to $b_0$, which is either 1 or 2, hence the notation $b_0 \times \hspace{-0.1cm}\vcenter{\hbox{ \begin{tikzpicture}[scale = 0.7]
					\draw [color=black,fill = white] (1.85,-0.77) circle (3.5pt);
		\end{tikzpicture} }}\hspace{-0.1cm} + b_0 \times \hspace{-0.1cm}\vcenter{\hbox{ \begin{tikzpicture}[scale = 0.7]
					\draw [color=black,fill = black] (1.85,-0.77) circle (3.5pt);
		\end{tikzpicture} }}\hspace{-0.1cm} $
		). This way, there are no remaining holes in the green interval. In the blue interval, there are $h_1 = \nholes - b_0 - 2$ holes and $b_1 = \nballs - b_0$ balls. Thus, at time 1, the number of holes in $\TL$ in the blue interval is equal to $h_1 - b_1 = \nholes - \nballs - 2$ (equal to 2 in this example).
	}
	\label{fig:bloc_contenant0}
\end{figure}
\end{proof}

\begin{proof}[Proof of Claim \ref{claim:c'est_bien_une_proba}.]
We consider the random walk $(W_i)_{i\geq 0}$ defined as $W_0 = 0$ and with independent increments such that the law of $W_{i+1} - W_i$ is $\db \delta_{1} + (1- \db - \dt)\delta_0 + \dt \delta_{-1}$, for any $i$. The key is the following equality
\begin{align}
	\frac{f(\ell_0)}{\dt - \db} &=  \sum_{b_0 : 2 b_0 \leq \ell_0} \frac{\ell_0 + 1}{b_0 + 1}\binom{\ell_0}{b_0, b_0, \ell_0 - 2 b_0} \db^{b_0} \dt^{b_0 + 1}(1 - \db - \dt)^{\ell_0 - 2 b_0}\\
	&=  \sum_{b_0 : 2 b_0 \leq \ell_0} \binom{\ell_0+1}{b_0, b_0+1, \ell_0 - 2 b_0} \db^{b_0} \dt^{b_0 + 1}(1 - \db - \dt)^{\ell_0 - 2 b_0}\\
	&= \Prob{W_{\ell_0} = -1}.
\end{align}
Therefore, if we define $N\coloneqq \sum_{\ell_0 = 0}^{+\infty} \ind{W_{\ell_0 = -1}}$ as the number of passages of $W$ at $-1$, we have 
\begin{equation}
	\sum_{\ell_0 \geq 0} f(\ell_0) = (\dt - \db) \sum_{\ell_0 \geq 0} \Esp{\ind{W_{\ell_0} = -1}} =  (\dt - \db) \Esp{N}.
\end{equation}
To compute $\Esp{N}$, we let $N_{-1} \overset{(d)}= \sum_{\ell_0 > \tau_{-1}} {\ind{W_{\ell_0} = -1}}$ be the number of passages at -1 after time $\tau_{-1}$. We observe that $N \overset{(d)}= 1 + N_{-1}$ and we can decompose $N_{-1}$ according to the first step of the walk after time $\tau_{-1}$, which we denote by $X = W_{\tau_{-1}+1} - W_{\tau_{-1}}$:
\begin{equation}
	N_{-1} \overset{(d)}= \left\{\begin{array}{lll}
		N & \text{with probability } \db & \text{(if $X = 1$)}\\
		1 + N_{-1} & \text{with proba } 1 - \db - \dt & \text{(if $X = 0$)} \\
		1 + N_{-1} & \text{with proba } \dt p_{-2}  & \text{(if $X = -1$ and $W$ returns to -1)} \\
		0 & \text{with proba } \dt (1-p_{-2} ) & \text{(if $X = -1$ and $W$ never returns to -1)}
	\end{array}\right. \label{eqn:N}
\end{equation}
where $p_{-2} \coloneqq \P_{-2}\left(\tau_{-1} < \infty\right)$ refers to the probability that $W$ returns to $-1$ after having reached $-2$. We now show that $\P_{-2}\left(\tau_{-1} < \infty\right) =  \frac{ \db}{\db + \dt} \mathcal{G}\left(\frac{\db \dt}{(\db + \dt)^2}\right)$ (recall that $\mathcal{G}$ is the Catalan generating function). We consider the walk $W'$, starting at $-2$ and with i.i.d.\ increments having distribution $\mathcal{L}\left(W'_{i+1} - W'_i\right) = \frac{\db}{\db + \dt} \delta_1 + \frac{\dt}{\db + \dt} \delta_{-1}$. Then, since the return of $W$ to $-1$ (after time $\tau_{-2}$) does not depend on its steps $0$, it means that
$\P_{-2}\left(\tau_{-1} < \infty\right) = \P_{-2}\left(\tau_{-1}(W') < \infty\right)$. We can thus compute it:
\begin{align}
	\P_{-2}\left(\tau_{-1} < \infty\right) &= \sum_{k\geq 0} \P\left(\tau_{-1}(W') = 2k+1\right) \\
	&= \sum_{k\geq 0}C_k \left(\frac{\db}{\db+\dt}\right)^{k+1} \left(\frac{\dt}{\db+\dt}\right)^{k} = \frac{\db}{\db+\dt} \mathcal{G}\left(\frac{\db \dt}{(\db+\dt)^2}\right).
\end{align}
In particular, it is straightforward to compute that $\mathcal{G}\left(\frac{\db \dt}{(\db+\dt)^2}\right) = \frac{\db+ \dt}{\dt}$ (since $\dt > \db > 0$). Moreover, Equation (\ref{eqn:N}) gives the following system:
\begin{equation}
	\begin{cases}
		\Esp{N} = 1 + \Esp{N_{-1}}, \\
		\Esp{N_{-1}} = \db \Esp{N} + \left(1 - \db - \dt + \frac{ \db \dt}{\db + \dt} \mathcal{G}\left(\frac{\db \dt}{(\db + \dt)^2}\right)\right) (1 + \Esp{N_{-1}}).
	\end{cases}
\end{equation}
which can be solved to obtain $\Esp{N} = \frac{1}{\dt - \db}$. Finally, $\sum_{\ell_0 \geq 0} f(\ell_0) = (\dt - \db) \Esp{N} = 1$ as wanted.

\end{proof}

\subsection{Law of the set of remaining holes when $\db = \dt$ - proof of Theorem \ref{prop:Z_same_density_bh_empty_holeset}}\label{subsect:Z_pas_de_trous_si_db=dt}

\begin{proof}[Proof of Theorem \ref{prop:Z_same_density_bh_empty_holeset}]
We assume that $\db = \dt$. The case $\db = 0$ is trivial, so we assume that $\db>0$. We let $k\geq1$ and $\eps>0$. 

As $\Holeset^t$ is well-defined for every $t$, we can consider the process $\left(\Delta_0 \Holeset^t\right)_{t\in\intzo}$. It is a non-decreasing process (since $\Holeset^t$ is a set that decreases over time), and thus for any $t< 1$, \[\forall i,\ \pdhdh(\Delta_0\TL \leq i) \leq \pdhdh(\Delta_0\Holeset^t \leq i).\]

Then (also for any $t$), $\Delta_0\Holeset^t$ under $\pdhdh$ has the same distribution as $\Delta_0\Holeset^1$ under $\ptdhdh$ (the key being that, in both cases, the probability that some vertex contains a ball that has been activated at the considered time is $t\dt$). In particular, $\pdhdh(\Delta_0\Holeset^t \leq k) = \ptdhdh(\Delta_0\Holeset^1 \leq k)$.

Equation (\ref{eqn:lalimite!}) implies that for every $i$, $\pdbdh(\Delta_0\TL = i)$ goes to $0$ as $\db$ tends to $\dt$. Thus there exists $t_\eps<1$ large enough so that for every $i\in\llbracket 0, k\rrbracket$, $\ptepsdhdh(\Delta_0\TL = i) \leq \eps$. It enables to conclude :
\begin{align}
	\pdhdh(\Delta_0\TL \leq k) &\leq \ptdhdh(\Delta_0\Holeset^1 \leq k) \leq (k+1)\eps.
\end{align}
As this is true for any $\eps$, it implies that $\pdhdh(\Delta_0\TL \leq k) = 0$. Therefore, under $\pdhdh$, $\Delta_0\TL = +\infty$ almost surely. Finally, thanks to Remark \ref{rem:TL_vide_ou_tres_infini}, $\TL = \emptyset$ almost surely.
\end{proof}

\addition{We therefore obtain the following results, completing Corollary \ref{cor:Zdef_tant_que_pas_trop_de_balles}: we know that when $\db > \dt$ the process is valid until time $\dt/\db$ (included), and this corollary tells us what happens after time $\dt/\db$.

\begin{corollary}\label{cor:Zdef_tant_que_pas_trop_de_balles2}
	If $\db > \dt$, the golf process on $\Z$ is almost surely frozen at any time $t>\dt/\db$.
\end{corollary}

Recall that a process is frozen when a ball cannot reach a hole. Theorem \ref{prop:Z_same_density_bh_empty_holeset} implies that at time $\dt/\db$ all the holes are filled, and thus the process becomes frozen as soon as other balls are activated.}
%

\subsection{Coupling with the parking on $\Z$}\label{subsect:golf=parking_onZ}

\addition{We described in Section \ref{subsect:litterature} the parking process on $\Z$ (see \cite{przykucki2019parking}), defined with the same initial configuration as the golf process on $\Z$ (i.e.\ each site contains a ball with probability $\db$ or a hole with probability $\dt$, independently of the other sites). In the parking process, all the balls move simultaneously, doing one random step at each integer time, until they have reached a free hole, which becomes occupied, as usual. More precisely, each ball steps are independent, and such that a ball moves to the right with probability $p$, to the left with probability $1-p$. Recall that $\Holeset^t$ is the set of free holes at time $t$ for the golf process. Similarly, we let $P^t$ be the set of free holes at time $t$ in the parking process. We should emphasize that the temporality in the golf and parking processes is completely different: the golf process ends at time $t=1$, once every ball has moved (recall that they move only once, directly to a free hole) whereas the parking process almost surely do not reach a final configuration in finite time. 

In the following theorem, we prove that when $\db \leq \dt$, the two processes have the same final configuration (asymptotically).
}

\begin{theorem}[Coupling with the parking process on $\Z$]\label{cor:Zgolf=parking}
\addition{
	When $\db \leq \dt$, $\lim_{t\to\infty}P^t \overset{(d)}= \TL $. Moreover, the two processes can be coupled so that almost surely,
	$$\lim_{t\to\infty}P^t= \TL .$$
}
\end{theorem}

\begin{proof}\addition{		
	As for Theorem \ref{thm:defZ}, we first prove the result in the case $\db < \dt$. We consider $\eta$ the golf process, and $\xi$ the parking process, coupled so that $\eta^0 = \xi^0$ (this is possible since the two initial configurations have the same distribution). We let, as in the proof of theorem \ref{thm:defZ}, $\Sep $ be the set of separators in $\eta^0$. We use Lemma \ref{lemma:infinite_nb_separators} and consider the processes $\eta$ and $\xi$ between two consecutive elements of $\Sep$: it suffices to do a finite coupling between those two processes. This coupling comes directly thanks to Diaconis and Fulton argument (\cite{diaconis1991growth}, see Remark \ref{remark:prop_commutation}), or equivalently from the \textit{space-based} model introduced in \cite{przykucki2019parking}: we associate to each vertex $u$ an infinite sequence $V^u = (V_1^u, V_2^u,\cdots)$ such that $\forall v, \Prob{V_i^u = v} = P_{u,v}$. We let $\Veta$ be this sequence for the golf process $\eta$, and $\Vxi = \Veta$ be a copy of this sequence, which is going to be used for the parking process. We then proceed as explained in Remark \ref{remark:prop_commutation}: when a ball at vertex $u$ needs to do a step in the golf process (respectively in the parking process), it moves according to the first $\Veta_i^u$ (respectively $W_i^u$) that has not been used yet. For the parking case, we use random variables to break ties if needed. Then, we know from Diaconis-Fulton that whatever the order in which the balls move, the same steps of $V$ will be used, so in particular the same holes will be reached in both the golf and the parking processes. We thus have built a coupling of the two final configurations $\eta^1$ and $\lim_{t \to \infty}\xi^t$, between every pair of consecutive elements of $\Sep$. This induces a coupling on the final configurations on $\Z$, since we have seen (Lemma \ref{lemma:no_ball_exits_enclosed_intervals}) that separators are free holes, at any time step (and it is straight forward to check that it is also true for the parking process).\\
	
	When $\db = \dt$, we have proved that, for the golf process, every ball finds a hole and that at time 1, every hole is filled (see Theorem \ref{prop:Z_same_density_bh_empty_holeset}). It suffices to prove the same result for the parking process. First, every ball finds a hole almost surely (it is a direct consequence of Theorem 1.2 in \cite{przykucki2019parking}, applying Markov property to conclude that $\P(\tau>t) \leq C t^{-1/4}$, where $\tau$ is the time at which a ball finds a hole). Second, we can prove that $\lim_{t\to\infty} P^t = \emptyset$: we consider $ P^t$ as a function of $\db$, and now write $P^t(\db)$, for the process in which we only have activated balls with clocks $<\db$ (clocks being the same as clocks for the golf process, as in the previous paragraph). With this notation, we thus want to prove that $\lim_{t\to\infty} P^t(\dt) = \emptyset$. Since we do not care which ball fills which hole in the parking, it is straightforward that $P^t(\db) $ is a decreasing function of both $\db$ and $t$, and therefore $$\lim_{t\to\infty} P^t(\dt) = \lim_{t\to\infty} \lim_{\db \to \dt, \db < \dt} P^t(\db) = \lim_{\db \to \dt, \db < \dt}  \lim_{t\to\infty} P^t(\db) $$
	Finally, we use the proof for the case $\db < \dt$ to couple $\lim_{t\to\infty} P^t(\db)$ with $\TL$ for the golf process with $\db < \dt$. Similar reasoning as in the proof of Theorem \ref{prop:Z_same_density_bh_empty_holeset} implies that, when the size of the block containing 0 in $\lim_{t\to\infty} P^t(\db)$ converges to $\infty$ when $\db \to \dt$, and in particular $\lim_{t\to\infty} P^t(\dt)=\lim_{\db \to \dt, \db < \dt}  \lim_{t\to\infty} P^t(\db) = \emptyset$.
	
}\end{proof} 

\section{Golf and parking with different moving strategies}\label{subsect:proof_prop_variantes}

\addition{
In this section we prove Proposition \ref{claim:variantes}. It is a kind of corollary of Theorems \ref{thm:loitrousresiduels}, \ref{thm:parking_loi_TL} and \ref{thm:defZ}, or rather can be proven by more or less the same proofs, therefore we only sketch the main arguments. We first give two reasons why the proof of Theorem \ref{thm:loitrousresiduels} is robust and applies to all the variants of the golf process that are shift invariant and local. Then, we discuss how to conclude for the other results of Proposition \ref{claim:variantes}.

\paragraph{Robustness of the proof of Theorem \ref{thm:loitrousresiduels} - exchangeability argument}
We can use an exchangeability argument to prove that the distribution of $\TL$ does not depend on the strategy of the balls (but it does not give this distribution). Exchangeability is a notion discussed, in particular, by Aldous \cite{Aldous_exchangeability}. Here, the block sizes are exchangeable (except for the size of the block containing 0, which is size-biased): $(\Delta_0\TL, \Delta_1 \TL, \ldots, \Delta_{\nl-1} \TL) \overset{(d)}= (\Delta_0\TL, \Delta_{\sigma_1} \TL, \ldots, \Delta_{\sigma_{\nl-1}} \TL)$, for any permutation $\sigma$ of $\{1,\ldots,\nl-1\}$. This can be proven by showing that the blocks (that is, the restrictions of the golf process between two consecutive free holes) are exchangeable, at any time $t\leq 1$ (this is the identity of the block that contains 0 which is not). This comes from the fact that the balls have shift invariant and local strategies. 

Exchangeability implies that the strategy of the balls do not really matter: 
If at time $t$, a ball $b$, coming from a block $B_j(\tmoins)$ existing at this time, becomes active, it will fill the space between $B_j(\tmoins)$ and $B_k(\tmoins)$ where $k=j+1$ or $j-1$ (that is an adjacent block) and this will produce the coalescence of $B_k(\tmoins)$ and $B_j(\tmoins)$ at time $t$. The distribution of $k$ depends on the position of the activated ball $b$ in its block and of the transition matrix of the Markov chain it used (or more generally, of the displacement politics it employed); nevertheless, since $(|B_{j+1}(t)|,|B_{j-1}(t)|)$ has same distribution as  $(|B_{j-1}(t)|,|B_{j+1}(t)|)$ by exchangeability (recall that $(\smallcard{B_{i}(t)} = \Delta_i\Holeset^t$, for any $i$), the distribution of the resulting coalescent process (seen at the level of the block sizes) does not depend on the displacement politics.

This reasoning can be made rigorous to deal with the size-bias of $\Delta_0 \Holeset^t$, but we do not do it here.

\paragraph{Robustness of the proof of Theorem \ref{thm:loitrousresiduels} - digging into the details of the proof of Lemma \ref{lemma:simpl_thm_loi_trous_residuels}} 
For every shift invariant strategy, it is straightforward that Lemma \ref{lemma:casinitloitrousresiduels} is true. Lemma \ref{lemma:simpl_thm_loi_trous_residuels} relies on the block decomposition principle, and the property of ``local decision'' is the key that enables to see that Lemma \ref{lemma:simpl_thm_loi_trous_residuels} is also true here: this notion of local decision implies that a ball strategy does not depend on what is outside its block, and thus that we can decompose $\znz$ as independent blocks, that we then treat independently.

At the end of the computation (to get Equation \ref{eqn:conclu_loi_TL_lemme_simpl_thm_loi_trous_residuels} in the proof of Lemma \ref{lemma:simpl_thm_loi_trous_residuels}), we use the commutation property. We can still conclude if the commutation property is not true, as it suffices to sum over all the permutations of $\znz$ (i.e.\ over every order of the balls) : the uniform distribution of clocks gives a uniform permutation of the balls, which induces a uniform permutation of the balls inside each interval between consecutive elements of $X$, and thus allows to conclude. Finally, if the strategy of a ball depends on its clock, the proof of Lemma \ref{lemma:simpl_thm_loi_trous_residuels} is a bit more subtle, because it involves to deal with continuous variables, but the same structure of proof works in this continuous setting.

\paragraph{Conclusion}

Once Theorem \ref{thm:loitrousresiduels} is valid, with the same formula for the distribution of $\TL$, it follows directly that Theorems \ref{thm:phaseT} and \ref{thm:blocs_cercle_nb_nt_fixes} also apply for different moving strategies.

For the parking process, the proof of Theorem \ref{thm:parking_loi_TL} follows exactly the same steps as the one of Theorem \ref{thm:loitrousresiduels} for the golf process (see Section \ref{subsect:parking_loi_TL}), and our two previous paragraphs justify why the theorem also applies to shift invariant and local variants.

On $\Z$, the proof of validity of the golf process relies on a notion of separators, that are holes that cannot be filled by any ball (almost surely, for a given initial distribution), regardless of the balls strategy, and thus applies to these variants. For the proof of Theorem \ref{thm:loiblocsZ}, the reasoning also works for any variant, with the subtlety that the coupling defined in the conclusion of the proof needs to be slightly modified (so that the coupling of the trajectories becomes a coupling of the ball strategies, such that two balls on the same vertex with the same local information do the same choice). Notice that Lemma \ref{lemma:loi_Z_lemme_technique_memes_trous} can also be extended to any local variant.
}

\section*{Acknowledgments}

The author would like to thank Jean-François Marckert for his great supervision, for the many fruitful discussions we had together and for his guiding in the writing of this paper.

We also thank Laurent Tournier and Matthew Junge who communicated us additional references that we were not aware of in a preliminary version of the paper. 

Finally, we would like to express our gratitude to the reviewers for their careful reading and their numerous corrections and suggestions, which allowed us to improve the quality and clarity of the paper.

\bibliographystyle{alpha}
\bibliography{biblio.bib}

\appendix

\section{Useful theorems}\label{annex:thms}
\subsection{Central local limit theorem}\label{annex:TCLL}

We use the following {central local limit theorem} :
\begin{theorem}[Central local limit theorem] \label{thm:TCLL}
Let $(X_i)_{i\in\N}$ be a sequence of i.i.d.\ random variables, such that $\Esp{X_1} = \mu$ and $\Var{X_1} = \sigma^2 \in (0,\infty)$. We assume that the support of $X_1$ is included in the set $b+h\Z$, for some $b$ and $h$, with $h$ maximal for this property. Then, 
\begin{equation}
	\sup_x \left| \frac{\sigma \sqrt{n}}{h} \Prob{\frac{\sum_{i=1}^{n}X_i}{\sigma\sqrt{n}} = x + \sqrt{n}\frac{\mu}{\sigma} } - \frac{\exp{-\frac{x^2}{2}}}{\sqrt{2\pi}} \right|  \underset{n\to\infty}\to 0
\end{equation}
where $x$ takes values in $\frac{b-\mu}{\sigma}\sqrt{n} + \frac{h}{\sigma \sqrt{n}}\Z$.
\end{theorem}

We can deduce that, under the same hypothesis, for any $(X_i)_{i\in\N}$ a sequence of i.i.d.\ random variables with $\Esp{X_1} = \mu$ and for any sequence $K_n = o(\sqrt{n})$, 
\begin{equation}
\sup_{x\in {b-\mu}n + h\Z :\ |x| \leq K_n} \left| \frac{\sigma \sqrt{n}}{h} \Prob{ \sum_{i=1}^{n}X_i = x + n\mu } - \frac{1}{\sqrt{2\pi}} \right|  \underset{n\to\infty}\to 0 \label{eqn:TCLL2}.
\end{equation}

In particular, for any sequence $x_n = o(\sqrt{n})$, 
\begin{equation}
\frac{\sigma \sqrt{n}}{h} \Prob{ {\sum_{i=1}^{n}X_i} = n{\mu} + x_n} \underset{n\to\infty}\to \frac{1}{\sqrt{2\pi}}.\label{eqn:TCLL3}
\end{equation}

\subsection{Scheffé's lemma and its use to prove convergence of rescaled discrete random variables} \label{annex:scheffe}
We recall the classical Scheffé's lemma (see \cite{MR1324786}):
\begin{lemma}[Scheffé's lemma]\label{lemma:scheffe}
Let $f, f_0,f_1\ldots$ be a sequence of probability density functions on $\R^d$. If $(f_n)$ converges almost everywhere to $f$, and if the random variables $X_n$ and $X$ have respective densities $f_n$ and $f$, then $X_n \overset{(d)}\to X$.

\end{lemma}

We will use the following corollary, which is folklore:
\begin{corollary}[Discrete Scheffé's lemma] \label{lemma:discreteScheffe}
We consider a sequence of random variables $(X_n)_n$ on $\Z^d$, and a sequence $(a_n)_n \in (\R_+)^\N$ such that $a_n \to +\infty$ as $n$ goes to infinity. We also let $X$ be a random variable with probability density function $f$, defined on $\R^d$. If for almost every $x\in \R^d$, $\lim_n {a_n}^d \Prob{X_n = \lfloor a_n x \rfloor} = f(x)$, then $X_n/a_n \overset{(d)}\to X$ as $n \to\infty$.
\end{corollary}

\begin{proof}It suffices to notice that if we set, for every $n$, 
\begin{equation*}
	Y_n \coloneqq \frac{X_n + U}{a_n}
\end{equation*}
where $U \sim \mathcal{U}(\intzo^d)$ is independent of $(X_n)$, then $Y_n$ is a continuous random variable with density $f_n(.) \coloneqq {a_n}^d \Prob{X_n = \lfloor a_n . \rfloor}$, that converges almost everywhere to $f$. Lemma \ref{lemma:scheffe} implies that $Y_n \overset{(d)}\to X$ as $n \to \infty$, and Slutsky's theorem enables to conclude that $X_n \overset{(d)}\to X$ as $n \to\infty$.
\end{proof}

\addition{
\subsection{Proof of Proposition \ref{prop:cv_marchecondit_vers_browniencondit}}\label{annex:cv_marchecondit_vers_browniencondit}

Since the event $\{\tau_{-\lambda}(B)=1\}$ has probability zero, we first explain what we call a Brownian motion conditioned by $\tau_{-\lambda}(B)=1$. Consider the three following kernels (defined for all $t> 0, x \in \R, y \in \R$):
\begin{eqnarray*}
	\Phi_t(x)&=&\exp(-x^2/(2t))/\sqrt{2\pi t},\\
	\Phi^{+}_t(x,y)&=& \Phi_t(y-x) - \Phi_t(x+y+2\lambda),\\
	\Phi^m_t(x)&=&\frac{x}{t} \Phi_t(x).
\end{eqnarray*}
The first kernel is the standard Gaussian kernel, the second one ``is the weight of Brownian paths'' going from $x$ to $y$ in an interval of length $t$ and which stays above the line $y=-\lambda$, the last one, is a ``meander type weight'', corresponding to the weight of Brownian paths of length $t$, starting at (0,0), staying above the $x$ axis, and ending at height $x$ at time $t$.
The process that we call Brownian motion, conditioned by $\tau_{-\lambda}=1$ has finite-dimensional distributions (FDD)
\[f_{t_1,\cdots,t_k}(x_1,\cdots,x_k)=\frac{\sqrt{2\pi}}{\lambda\exp(-\lambda^2/2)} \left[\prod_{j=1}^k \Phi^+_{t_j-t_{j-1}}(x_{j-1},x_{j})\right]\Phi^m_{1-t_k}(x_k+\lambda)\] 
for all $(x_i)$ such that $\forall i,x_i>-\lambda$ (if otherwise $x_i < -\lambda$ for some $i$, then $f_{t_1,\cdots,t_k}(x_1,\cdots,x_k) \linebreak= 0$).

\begin{proof}[Proof of Proposition \ref{prop:cv_marchecondit_vers_browniencondit}]
	Fix $\lambda>0$, take $k(n)\sim \lambda\sqrt{n}$. 
	
	The proof of the convergence of the FDD of the rescaled path ${W^{(n, \lambda)}}$ is a classical proof: fix some $0<t_1<\cdots < t_i<1$ (with $i$ an arbitrary integer), some $x_1,\ldots,x_i$ all greater than $-\lambda$, and compute the number of trajectories with step $+1$, $-1$ staying above level $-\lambda\sqrt{n}$ and passing by the points $(\floor{nt_j},\floor{x_j\sqrt{n}})$ (for any $j\leq i$); exact formulas are available, using the reflection principle of Désiré André (and also, for the last section between times $t_i$ and 1, that the probability that a random walk of length $\ell$ ends at its lowest level $-y$ has probability $(y/\ell)\P( W_l=-y)$). They allow to establish the convergence of the FDD, using a local limit theorem.
	
	Now, to conclude the proof of the convergence, it suffices to prove that the sequence (of the distributions) of $(W^{(n, \lambda)})_{n\geq 1}$ is tight.\footnote{In fact, it is well-known that, in $\Czo$ (the set of continuous functions on the unit interval) equipped with the uniform topology, convergence of the FDD of a sequence $(P_n)$ of probability measures and tightness of the sequence $(P_n)$ imply the convergence of $P_n$ in distribution in $\Czo$ (and the FDD of the limit are the limit of the FDD of $(P_n)$). See for example \cite[Theorem 7.1]{Billingsley1999}. } To do this, we consider two other sequences $X^{(n, \lambda)}$ and $Y^{(n)}$ that we define as follows. We let, as in the statement of Proposition \ref{prop:cv_marchecondit_vers_browniencondit},  $(W_i,i\geq 0)$ be a simple random walk (without any conditioning). We linearly interpolate it, so that it is defined for every real number (and continuous), and we let $Y^{(n)} \coloneqq \left(\frac{W_{nt}}{\sqrt{n}}, 0\leq t\leq 1\right)$. We also let, for every $n$,  $X^{(n,\lambda)} \coloneqq \left(\frac{W_{nt}}{\sqrt{n}}, 0\leq t\leq 1\right)$ conditioned by $W_n = k(n)$. In other words, $X^{(n,\lambda)}$ is a bridge of length $n$ ending at $-k(n) \sim -\lambda \sqrt{n}$.
	One can note that $\mathcal{L}\left(X^{(n, \lambda)}\right) =  \mathcal{L}\left(Y^{(n)}\middle| Y^{(n)}_1 = -\lambda\right)$ and $\mathcal{L}\left(W^{(n, \lambda)}\right) = \mathcal{L}\left( Y^{(n)} \middle| \tau_{-\lambda}(Y^{(n)}) = 1\right)$ (Here and below, when we write $ Y^{(n)}_1 = -\lambda$ or $\tau_{-\lambda}(Y^{(n)}) = 1$, $\lambda$ is an abuse of notation for $k(n)/\sqrt{n}$).

	\medskip
	We now prove that $(X^{(n, \lambda)})_{n\geq 1}$ is tight. We show that the tightness of the laws of our conditioned paths can be deduced from that of the sequence of the laws of $(Y^{(n)})$ of simple random walks suitably rescaled (the latter's tightness is a consequence of Donsker's theorem,  see e.g. \cite[Theorem 8.2]{Billingsley1999}). We use arguments sketched in \cite[proof of Lemma 1]{snakesJansonMarckert}.
	
	We use Theorem 7.3 (ii) of \cite{Billingsley1999}, to assert that for each $\eps>0, \eta>0$, there exist $\delta\in(0,1)$ and $n_0$ such that
	\begin{equation}
		\P (\omega_{[0,1]}(Y^{(n)},\delta)\geq \eps) \leq \eta, \textrm{ for } n \geq n_0 \label{eqn:tight=prop_mod_continuity}
	\end{equation}where 
	$\omega_A(f,\delta)=\max\{ |f(x)-f(y)|, |x-y|\leq \delta, x,y\in A\}$ 
	is the standard modulus of continuity of $f$ on $A$. 
	Of course, this property implies that  for each $\eps>0, \eta>0$, there exists $\delta\in(0,1)$ and $n_0$ such that
	\[\P (\omega_{[0,1/2]}(Y^{(n)},\delta)\geq \eps) \leq \eta, \textrm{ for } n \geq n_0.\]
	
	In what follows, we cover $[0,1]$ by three intervals $[0,b_n]$, $[a_n,d_n]$ and $[c_n,1]$ of size $<1/2$, where $a_n,b_n,c_n,d_n$ are discretization points of the form $j/n$, and such that each entry of $(a_n,b_n,c_n,d_n ) $ is at distance smaller than $1/n$ to the corresponding entry of $(1/4,2/5,3/5,3/4)$.
	Using that for all $\delta<1/10$ (this is sufficient), and all $f$,
	\[ \omega_{[0,1]}(f,\delta)\leq \omega_{[0,b_n]}(f,\delta)+ \omega_{[a_n,c_n]}(f,\delta)+ \omega_{[b_n,1]}(f,\delta) \]
	and invariance by rotation (its distribution is preserved by the rotation in $\znz$ of its increments), we get
	\begin{align}\P\left( \omega_{[0,1]}(X^{(n,\lambda)},\delta)) \geq \eps \right)
		&\leq 3 \P\left( \omega_{[0,1/2]}(X^{(n,\lambda)},\delta) \geq \eps /3 \right) \\
		&= 3  \P\left( \omega_{[0,1/2]}(Y^{(n)},\delta) \geq \eps /3  \middle| Y^{(n)}(1)=-\lambda\right) .
	\end{align}
	We now prove that there exists $C_\lambda$ such that 
	\begin{equation}
		\P\left( \omega_{[0,1/2]}(Y^{(n)},\delta) \geq \eps /3  \middle| Y^{(n)}(1)=-\lambda\right) \leq C_\lambda  \P\left( \omega_{[0,1/2]}(Y^{(n)},\delta) \geq \eps /3 \right) . \label{eqn:trucamontrer}
	\end{equation} 
	
	Recall that $Y^{(n)} = \left(\frac{W_{nt}}{\sqrt{n}}, 0\leq t \leq 1\right)$ and $\mathcal{L}\left(X^{(n, \lambda)}\right) =  \mathcal{L}\left(Y^{(n)}\middle| Y^{(n)}_1 = -\lambda\right)$. We focus on $W$, the non-normalized path. In what follows, we denote by $W[0,a]$ the restriction of the path $W$ to $[0,a]$.
	For $A$ an event measurable with respect to $\sigma\{W_0,\cdots,W_{\lfloor n/2 \rfloor}\}$,
	\begin{align*}
		&	\P( W[0,\lfloor n/2\rfloor ] \in A~|~ W_{n}=-k(n))\\&=\sum_m \P(  W[0,\lfloor n/2\rfloor ] \in A, W_{\lfloor n/2\rfloor }=m~|~ W_{n}=-k(n))\\
		&= \sum_m \frac{\P(  W[0,\lfloor n/2\rfloor]\in A| W_{\lfloor n/2\rfloor}=m )\P(   W_{n}=-k(n)|W_{\lfloor n/2\rfloor}=m )\P(W_{\lfloor n/2\rfloor}=m)}{\P(W_{n}=-k(n) )}
	\end{align*} 
	and using that $\frac{\P(   W_{n}=-k(n)|W_{\lfloor n/2\rfloor}=m )}{\P(W_{n}=-k(n) )}=\frac{\sqrt{n} \P(W_{{\lfloor n/2\rfloor}}=-m-k(n))}{\sqrt{n}  \P(W_{n}=-k(n) )}$ is bounded (the central local limit theorem {(see Theorem \ref{thm:TCLL} in Annex \ref{annex:TCLL})} enables to show that the numerator is bounded and that the denominator converges), one gets
	\[\P( W[0,\lfloor n/2 \rfloor ] \in A~|~ W_{n}=-k(n)) \leq C_\lambda\P( W[0,\lfloor n/2 \rfloor] \in A)\] for a constant $C_\lambda$.
	
	We directly deduce that Equation (\ref{eqn:trucamontrer}) is true. For any $\eps>0$ and $\eta^X>0$, by (\ref{eqn:tight=prop_mod_continuity}), considering $\eta^Y = \eta^X/ (3 C_\lambda)$, there exists $n_0$ and $\delta>0$ such that $\forall n \geq n_0$, 
	$P (\omega_{[0,1]}(Y^{(n)},\delta)\geq \eps/3) \leq \eta^Y$.
	Then,
	\begin{align}\P( \omega_{[0,1]}(X^{(n,\lambda)},\delta)) \geq \eps )
		&\leq 3 C_\lambda  \P( \omega_{[0,1/2]}(Y^{(n)},\delta) \geq \eps /3 ) \\
		&\leq 3 C_\lambda \P( \omega_{[0,1]}(Y^{(n)},\delta) \geq \eps /3 ) \\
		&\leq 3 C_\lambda \eta^Y = \eta^X.
	\end{align}
	
	Theorem 7.3 of \cite{Billingsley1999} thus implies the tightness of $(X^{(n,\lambda)})$. Actually, this theorem of Billingsley gives a necessary and sufficient condition, that is why we used it twice here.
	
	\medskip
	From here to conclude, it suffices to notice that the tightness of $(X^{(n, \lambda)})_{n\geq 1}$ implies that of $(W^{(n, \lambda)})_{n\geq 1}$.
	Indeed, an argument that we develop below, for the sake of completeness, is that the ``rotation'' multiplies the modulus of continuity by at most 2.
	
	Recall that the rotation of a function was defined in Equation (\ref{eqn:def_de_rot_continue}).
	
	As a matter of fact, it is known that, if ${\bf r} \sim \mathcal{U}(\llbracket 0,n-1 \rrbracket)$,
	\begin{equation}
		\rot(W^{(n,\lambda)},{\bf r}/n) \overset{(d)}= X^{(n,\lambda)}.
	\end{equation}
	It can be seen as a consequence of the rotation principle (see again Equation (\ref{eqn:principe_de_rotation}) and the references given above).
	Then, it is straightforward that, for any $f$ and $r$, $\omega_{[0,1]}(f) \leq 2 \omega_{[0,1]}(\rot(f,r))$. Thus the tightness of $X^{(n,\lambda)}$ implies that of $W^{(n,\lambda)}$.

\end{proof}

}
	
\end{document}